\documentclass{amsart}
\usepackage{amssymb}
\usepackage{mathrsfs}
\usepackage{stmaryrd}
\usepackage{bbm}
\usepackage{oldgerm}
\usepackage[english]{babel}
\usepackage[T1]{fontenc}
\usepackage[latin1]{inputenc}
\usepackage[all]{xy}
\usepackage{hyperref}
\usepackage{color}
\usepackage{enumerate}

\newtheorem{theorem}[subsection]{Theorem}
\newtheorem{proposition}[subsection]{Proposition}
\newtheorem{corollary}[subsection]{Corollary}
\newtheorem{lemma}[subsection]{Lemma}

\theoremstyle{definition}

\newtheorem{definition}[subsection]{Definition}
\newtheorem{remark}[subsection]{Remark}

\numberwithin{equation}{subsection}

\usepackage{amsmath}
\usepackage{amsthm}
\usepackage{mathrsfs}
\usepackage{amsfonts}
\usepackage{extarrows}
\usepackage{tikz-cd}

\begin{document}

\title {Hypergeometric $\mathcal D$-modules and exponential sums for reductive groups}
\thanks{We would like to thank Qifeng Li, Daqing Wan and Dingxin Zhang for helpful discussions. 
The research of Lei Fu is supported by the National Key R\&D Program of China 2023YFA10097032023YFA1009703.}

\author{Lei Fu}
\address{Yau Mathematical Sciences Center, Tsinghua University, Beijing 100084, P. R. China}
\email{leifu@tsinghua.edu.cn}
\author {Xuanyou Li}
\address{Qiuzhen College, Tsinghua University, Beijing 100084, P. R. China}
\email{lixuanyo21@mails.tsinghua.edu.cn}

\date{}
\maketitle

\begin{abstract}
We define the hypergeometric exponential sum associated to a family of representations of a reductive group 
over a finite field. We introduce the hypergeometric $\ell$-adic sheaf to describe the hypergeometric exponential sum. 
Motivated by the definition of the hypergeometric sheaf, we introduce the hypergeometric $\mathcal D$-module, 
prove it is holonomic and estimate its rank. 
Using the theory of the Fourier transform for vector bundles over a general 
base developed by Wang, we show how the hypergeometric $\mathcal D$-module controls the general behavior of the 
hypergeometric sheaf. We apply our results to the estimation of the hypergeometric exponential sum. 



\end{abstract}

\section*{Introduction}

Let $k$ be a finite field with $q$ elements of characteristic $p$. For 
any finite dimensional vector space $V$ over $k$, let $\mathrm{GL}(V)$ be the general linear algebraic
group $k$-scheme, and let $\mathrm{End}(V)$ be the ring $k$-scheme parametrizing linear endomorphisms of $V$. For any 
$k$-scheme $X$ and any extension $k'$ of $k$, denote by 
$X(k')$ the set of $k'$-points in $X$. For any regular function $f\in\mathcal O_X(X)$, 
denote also by $f$ the $k$-morphism $f: X\to\mathbb A^1_k$ corresponding to the $k$-homomorphism
$$k[x]\to \mathcal O(X) ,\quad x\mapsto f.$$ 
Our motivation comes from the study of exponential sums. Let $\ell$ be a prime number distinct from $p$. 
Fix a nontrivial additive character $\psi: k\to\overline {\mathbb Q}_\ell^\ast$. In number theory, many problems lead to the study of the exponential sum 
$$S=\sum_{x\in X(k)} \psi(f(x)).$$ In this paper, we are concerned with the case where $X$ is a reductive group.

\subsection{Hypergeometric exponential sums on a reductive group}

Let $G$ be a reductive group over $k$, and let $$\rho_j: G\to \mathrm{GL}(V_j)\quad (j=1, \ldots, N)$$ be 
a family of irreducible representations of $G$, 
where $V_j$ are finite dimensional vector spaces over $k$. For any $k$-point $A=(A_1,\ldots, A_N)$ of 
$\prod_{j=1}^N\mathrm{End}(V_j)$, 
consider the morphism
\begin{eqnarray}\label{eqn:laurent1}
f: G\to\mathbb A_k^1, \quad g\mapsto\sum_{j=1}^N \mathrm{Tr}(A_j \rho_j(g)).
\end{eqnarray}
It defines a regular function on $G$. We call it a \emph{Laurent polynomial}. (One can show that over
algebraic closed ground field, all regular functions are Laurent polynomials.)
In this paper, we study the exponential sum
\begin{eqnarray}\label{eqn:Hypsum}
\mathrm{Hyp}(A)=\sum_{g\in G(k)} \psi(f(g))=\sum_{g\in G(k)} \psi\Big(\sum_{j=1}^N \mathrm{Tr}(A_j\rho_j(g))\Big)
\end{eqnarray} defined by the Laurent polynomial (\ref{eqn:laurent1}).
We call it the \emph{hypergeometric exponential sum} associated to the representations $\rho_j$ ($j=1,\ldots, N$). 

Consider the special case where $G=\mathbb G_m^n$ is a torus. 
Then each $\rho_j$ is a character
$$\rho_j: \mathbb G_m^n\to \mathrm{GL}(1, k), \quad (t_1, \ldots, t_n)\mapsto t_1^{w_{1j}}\cdots t_n^{w_{nj}}.$$  
For any $k$-point $A=(a_1, \ldots, a_N)\in k^n$ in $\prod_{j=1}^N\mathrm{End}(1)$, the Laurent polynomial 
\begin{eqnarray*}
f=\sum_{j=1}^N a_j t_1^{w_{1j}}\cdots t_n^{w_{nj}}\in k[t_1^{\pm 1}, \ldots, t_n^{\pm 1}]
\end{eqnarray*}
is a regular function on $G$.
The corresponding exponential sum 
\begin{eqnarray*}
\mathrm{Hyp}(a_1, \ldots, a_n)=\sum_{t_1,\ldots, t_n\in k^\ast}
\psi\Big(\sum_{j=1}^N a_j t_1^{w_{1j}}\cdots
t_n^{w_{nj}}\Big)
\end{eqnarray*} is called the Gelfand-Kapranov-Zelevinsky (GKZ) hypergeometric
exponential sum in \cite{F2}. This sum is studied in \cite{AS1,DL,F2}. 

\subsection{Hypergeometric sheaf}
The Lang isogeny 
$$L: \mathbb G_{a, k}\to \mathbb G_{a,k} \quad x\mapsto x^q-x$$ on the additive group $\mathbb G_{a,k}\cong \mathbb A^1$ 
is an \'etale isogeny and its kernel is the additive group $\mathbb A^1(k)=k$. It is a 
$k$-torsor. Pushing forward this torsor by the additive character $\psi^{-1}:k\to\overline{\mathbb Q}_\ell^*$, 
we get a lisse $\overline{\mathbb Q}_\ell$-sheaf $\mathcal L_\psi$ on 
$\mathbb A^1$, which we call the \emph{Artin-Schreier sheaf}. 
For any $k$-rational point $x\in \mathbb A^1(k)=k$, we have 
$$\mathrm{Tr}(\mathrm{Frob}_x, \mathcal L_{\psi, \bar x})=\psi(x),$$ 
where $\mathrm{Frob}_x$ is the geometric Frobenius element at $x$. 
By the Grothendieck trace formula \cite[Rapport 3.2]{SGA41/2}, we have the following expression for the 
hypergeometric exponential sum (\ref{eqn:Hypsum})
\begin{eqnarray*}
\mathrm{Hyp}(A)&=& 
\sum_{i=0}^{2\mathrm{dim}\,G} \mathrm{Tr}\Big(\mathrm{Fr},H^i_c(G\otimes_{k}\bar k, f^*\mathcal L_\psi)\Big),
\end{eqnarray*} where $\mathrm{Fr}$ is the geometric Frobenius correspondence.

For any vector space $\mathbb V$ over $k$, denote also by $\mathbb V$ the underlying $k$-scheme. 
Take $\mathbb V=\prod_{j=1}^N \mathrm{End}(V_j)$. Let $D_c^b(\mathbb V,\overline{\mathbb Q}_\ell)$
be the derived category of $\overline{\mathbb Q}_\ell$-sheaves with constructible cohomologies constructed in 
\cite{BS}.

\begin{definition}\label{defn:hypsheaf} Let $\pi_2:G\times_k \mathbb V\to \mathbb V$ be the projection,  let $F$ be the morphism
$$F:G\times_k  \mathbb V\to
\mathbb A_k^1,\quad F(g,A)=
\sum_{j=1}^N\mathrm{Tr}(A_j\rho_j(g)),$$ 
and let $d=\mathrm{dim}\, G$, $m=\mathrm{dim}\, \mathbb V.$
We define the
\emph{$\ell$-adic hypergeometric sheaves} $\mathrm{Hyp}_{\psi,!}, \mathrm{Hyp}_{\psi,*}, \in \mathrm{ob}\,
D_c^b(\mathbb V, \overline{\mathbb Q}_\ell)$ associated to the additive character $\psi$
and the representations $\rho_j$ $(j=1, \ldots, N)$ by
$$\mathrm{Hyp}_{\psi, !}:=R\pi_{2!}F^\ast \mathcal
L_\psi[d+m],\quad  \mathrm{Hyp}_{\psi, *}:=R\pi_{2*}F^\ast \mathcal
L_\psi[d+m].$$
\end{definition}

For any $k$-point $A=(A_1,\ldots, A_N)\in \prod_{j=1}^N\mathrm{End}(V_i) (k)$, we have
\begin{eqnarray*}
\mathrm{Tr}(\mathrm
{Frob}_A,(\mathrm{Hyp}_{\psi,!})_{\bar A})=(-1)^{d+m} 
\mathrm{Tr}(\mathrm{Fr},R\Gamma_c(G\otimes_{k}\bar k, f^*\mathcal L_\psi))=(-1)^{d+m} \mathrm{Hyp}(A).
\end{eqnarray*}
So $\mathrm{Hyp}_{\psi, !}$ describes 
how the family of exponential sums $\mathrm{Hyp}(A)$ behaves as $A$ varies.
We have 
$$\mathrm{Hyp}_{\psi,*}(d+m)\cong \mathbb D_{\mathbb V}(\mathrm{Hyp}_{\psi^{-1},!}),$$ 
where $\mathbb D_{\mathbb V}=R\mathcal Hom(-, \overline{\mathbb Q}_\ell)(m)[2m]$ is the Verdier duality functor. 
In section 1, we prove the following.

\begin{proposition}\label{perverse} Suppose the morphism
$$\iota: G\to \mathbb V=\prod_{j=1}^N \mathrm{End}(V_i), \quad \iota(g)= (\rho_1(g), \ldots, \rho_N(g))$$
is quasi-finite. Then $\mathrm{Hyp}_{\psi, !}$ and $\mathrm{Hyp}_{\psi, *}$ are mixed perverse sheaves on 
$\mathbb V=\prod_{j=1}^N \mathrm{End}(V_i)$. The 
weights of $\mathrm{Hyp}_{\psi, !}$ are $\leq \mathrm{dim}\,G+\mathrm{dim}\,\mathbb V$.
\end{proposition}

\subsection{} Parallel to the hypergeomeric $\ell$-adic sheaves, we have the 
$\ell$-adic hypergeometric $\mathcal D$-modules. 
We now work over an algebraically closed ground field $\overline K$ of characteristic $0$. 
For any smooth complex variety $X$, let $\mathcal D_X$ be the sheaf of differential operators, and let 
$$D^b_{qc}(\mathcal D_X), \; D^b_{c}(\mathcal D_X), \; D^b_h(\mathcal D_X), \; D^b_{rh}(\mathcal D_X)$$ be 
the derived category of bounded complexes of left $\mathcal D_X$-modules with cohomology sheaves being 
quasi-coherent, coherent, holonomic, holonomic with regular singularities, 
respectively. Similarly we have
the derived category of right $\mathcal D_X$-modules 
$$D^b_{qc}(\mathcal D_X^{\mathrm{op}}), \; D^b_{c}(\mathcal D_X^{\mathrm{op}}), \; D^b_h(\mathcal D_X^{\mathrm{op}}), \; 
D^b_{rh}(\mathcal D_X^{\mathrm{op}}).$$
We mainly works with the derived category $D^b_h(\mathcal D^{\mathrm{op}})$ on which we have the six-functor 
formalism by \cite[3.2]{Hottabook}.
On $\mathbb A^1=\mathrm{Spec}\,\overline K[x]$, let $\mathcal L$ be the right $\mathcal D_{\mathbb A^1}$-module 
$$\mathcal L:=\mathcal D_{\mathbb A^1}/(1-\partial_x)\mathcal D_{\mathbb A^1}.$$
Formally, we may write sections of $\mathcal L$ as $fe^{-x}\mathrm dx$ with 
$f$ being a section of $\mathcal O_{\mathbb A^1}$,
and $$(fe^{-x}\mathrm dx)\partial_x=-L_{\partial_x}(fe^{-x}\mathrm dx)=\Big(-\frac{\mathrm df}{\mathrm dx}+f\Big)e^{-x}\mathrm dx,$$
where $L_{\partial_x}$ is the Lie derivative. Motivated by Definition \ref{defn:hypsheaf} for hypergeometric $\ell$-adic sheaves, 
we define the hypergeometric $\mathcal D$-modules as follows.

\begin{definition}\label{defn:hypergeometricD} Let $G$ be a reductive algebraic group over $\overline K$, 
$\rho_j: G\to\mathrm{GL}(V_j)\quad (j=1, \ldots, N)$
a family of representations of $G$, $\mathbb V=\prod_{j=1}^N \mathrm{End}(V_j)$, 
$\pi_2:G\times \mathbb V\to \mathbb V$ the projection,  and
$F$ the morphism
$$F:G\times   \mathbb V\to
\mathbb A^1,\quad F(g,A)=
\sum_{j=1}^N\mathrm{Tr}(A_j\rho_j(g)).$$ 
We define the
\emph{hypergeometric $\mathcal D$-modules} $\mathcal Hyp_+, \mathcal Hyp_! \in \mathrm{ob}\,
D_h^b(\mathcal D_{\mathbb V}^{\mathrm{op}})$ associated to the representations 
$\rho_j$ $(j=1, \ldots, N)$ by
$$\mathcal Hyp_+:=R\pi_{2+}F^\ast \mathcal
L,\quad  \mathcal Hyp_!:=R\pi_{2!}F^\ast \mathcal
L.$$
\end{definition}

Let $\mathcal L^\vee$ be the right $\mathcal D_X$-module $\mathcal L^\vee=
\mathcal D_{\mathbb A^1}/(1+\partial_x)\mathcal D_{\mathbb A^1}$.
We have $$\mathbb D(R\pi_{2+}F^\ast \mathcal
L)\cong R\pi_{2!}F^\ast \mathcal
L^\vee.$$ 
In section 1, we also prove the following.

\begin{proposition}\label{holonomic} Suppose the morphism
$$\iota: G\to\prod_{j=1}^N \mathrm{End}(V_i), \quad \iota(g)= (\rho_1(g), \ldots, \rho_N(g))$$
is quasi-finite. Then $\mathcal Hyp_+$ and $\mathcal Hyp_!$ are holonomic $\mathcal D_{\mathbb V}$-modules.
\end{proposition}

Our goal is to describe the Zariski open set $U$ of $\mathbb V$ so that  
$\mathcal Hyp_+|_U$ and $\mathcal Hyp_!|_U$ are integrable connections, and 
estimate their ranks.

\subsection{Nondegenerate Laurent polynomials}\label{nondegeneracy} Let $B$ be a 
Borel subgroup of $G$, $T$ the maximal torus of $B$, 
$\Lambda=\mathrm{Hom}(T, \mathbb G_m)$ the weight lattice, 
$W=N_G(T)/T$ the Weyl group, and
$\lambda_j$ $(j=1, \ldots, N)$ the maximal weights of $\rho_j$. We define the \emph{Newton 
polytope at infinity} $\Delta_\infty$ for the family $\rho_1, \ldots, \rho_N$ to be the convex hull in 
$\Lambda_{\mathbb R}:=\Lambda\otimes \mathbb R$ of $W(0, \lambda_1, \ldots, \lambda_N)$. Occasionally we also 
use the \emph{Newton polytope} $\Delta$ which is defined to be the convex hull of $W(\lambda_1, \ldots, \lambda_N)$. 
For any face $\tau\prec \Delta_\infty$, let $e(\tau)=(e(\tau)_j)\in\prod_{j=1}^N \mathrm{End}(V_j)$ 
be defined as follows: Let
$$V_j=\bigoplus_{\lambda\in \Lambda} V_j(\lambda)$$ be the decomposition so that 
$V_j(\lambda)$ is the component of weight $\lambda$ under the action of $T$. We define $e(\tau)_j$ to be the block-diagonal 
linear transformation so that 
\begin{eqnarray}\label{etau}
e(\tau)_j\Big|_{ V_j(\lambda)}=\begin{cases}
\mathrm{id}_{V_j(\lambda)}&\hbox{if }\lambda\in \tau, \\
0&\hbox{otherwise}.
\end{cases}
\end{eqnarray}
For any point $A=(A_1, \ldots, A_N)$ in $\mathbb V=\prod_{j=1}^N\mathrm{End}(V_j)$, let $f_A$ be the morphism 
$$f_A:\mathbb V\to 
\mathbb A^1,\quad f_A(B)=\sum_{j=1}^N\mathrm{Tr}(A_jB_j).$$ 
We have an action $$(G\times G)\times \mathbb V\to \mathbb V, \quad
\Big((g, h), (B_1, \ldots, B_N)\Big)\mapsto \Big(\rho_1(g) B_1 \rho_1(h^{-1}), \ldots, \rho_N(g) B_N\rho_N(h^{-1})\Big).$$
The Laurent polynomial 
\begin{eqnarray}\label{generalLaurent}
f: G\to\mathbb A^1, \quad f(g)=\sum_{j=1}^N \mathrm{Tr}(A_j \rho_j(g)).
\end{eqnarray} 
determines the restriction of $f_A$ to the $(G\times G)$-orbit of 
$(\mathrm{id}_{V_1},\ldots, \mathrm{id}_{V_N})$. 

\begin{definition} We say the Laurent polynomial (\ref{generalLaurent})
is \emph{nondegenerate} if for 
any nonempty face $\tau$ of $\Delta_\infty$ not containing the origin, the restriction of $f_A$ to the orbit $Ge(\tau) G$
has no critical points, that is, the function 
\begin{eqnarray}\label{ftauA}
f_{\tau, A}:G\times G
\to\mathbb A^1, \quad f_{\tau, A}(g, h)=\sum_{j=1}^N \mathrm{Tr}(A_j \rho_j(g)e(\tau)_j\rho_j(h^{-1}))
\end{eqnarray}
has no critical point, which means that $\mathrm df_{\tau, A}=0$ has no solution 
on $G\times G$. Otherwise, we say $f$ is \emph{degenerate}. 
\end{definition}

Let $\mathbb V^{\mathrm{degen}}$ be the set of points $A$ so that the corresponding Laurent polynomial 
(\ref{generalLaurent}) is degenerate. In the Appendix (Proposition \ref{degenclosed}), 
we prove that $\mathbb V^{\mathrm{degen}}$ is a Zariski closed subset of $\mathbb V$.
Let $\mathbb V^{\mathrm{gen}}$ be 
the complement in $\mathbb V$ of $\mathbb V^{\mathrm{degen}}$. 
Then $\mathbb V^{\mathrm{gen}}$ is a Zariski open subset of $\mathbb V$ parametrizing nondegenerate Laurent polynomials.
The main result of the paper is the following. 

\begin{theorem}\label{thm:hypD} Suppose the morphism $\iota: G\to\prod_{j=1}^N \mathrm{End}(V_j)$ 
is quasi-finite. The hypergeometric $\mathcal D_{\mathbb V}$-modules $\mathcal Hyp_{+}$ and 
$\mathcal Hyp_{!}$ are holonomic, and their restrictions to $\mathbb V^{\mathrm{gen}}$ are integrable connections with rank 
$$\leq d!\int_{\Delta_\infty\cap\mathfrak C} 
\prod_{\alpha \in R^+} \frac{\lambda(H_\alpha)^2}{\rho(H_\alpha)^2}\mathrm d\lambda,$$
where 
$d=\mathrm{dim}\,G$, $\mathfrak C$ is the dominant Weyl chamber in 
$\Lambda_{\mathbb R}$, $R^+$ is the set of positive roots, 
$\{H_\alpha:\alpha\in R\}$ is the set of co-roots, 
$\rho=\frac{1}{2}\sum_{\alpha\in R^+}\alpha$, and $\mathrm d\lambda$ is the 
Lebesgue measure on $\Lambda_{\mathbb R}$
normalized by $\mathrm{vol}(\Lambda_{\mathbb R}/\Lambda)=1$. 
\end{theorem}

The general principle is that hypergeometric $\mathcal D$-modules control the generic behavior 
of the hypergeometric exponential sum. Let us make it precise. 

Let $R$ be a Dedekind domain so that its reside fields at maximal ideals are finite fields, 
and its fraction field $K$ is of characteristic $0$. Fix an algebraic closure $\overline K$ of $K$. 
Let $G_R$ be a split reductive group $R$-scheme, and let
$$\rho_{j, R}: G_R\to \mathrm{GL}(V_{j,R})\quad (j=1,\ldots, N)$$ 
be representations so that $\rho_{j, \overline K}\to \mathrm{GL}(V_{j,\overline K})$ are irreducible, where 
$V_{j,R}$ are projective $R$-modules of finite ranks, and $V_{j,\overline K}:=V_{j, R}\otimes_R\overline K$. 
Let $\Delta_\infty$ be the Newton polytope of $\rho_{j, \overline K}$ $(j=1, \ldots, N)$. As an application of 
Theorem \ref{thm:hypD}, we prove the following. 

\begin{theorem}\label{thm:expsum} Notation as above.
Suppose the morphism $$\iota_R: G_R\to\prod_{j=1}^N \mathrm{End}(V_j),\quad g\mapsto (\rho_{1,R}(g),
\ldots, \rho_{N, R}(g))$$ 
is quasi-finite. Let $A=(A_1, \ldots, A_N)\in \prod_{j=1}^N \mathrm{End}(V_{j, R})$
such that the Laurent polynomial $f_A(g)=\sum_{j=1}^N \mathrm{Tr}(A_j\rho_j(g))$ is nondegenerate over $\overline K$. 
Then there exists a finite set $S$ of maximal ideals $R$ such that for any maximal ideal $\mathfrak m\not\in S$, 
any finite extension $k'$ of the residue field $k=R/\mathfrak m$ and any nontrivial additive character $\psi:k\to\overline {\mathbb Q}_\ell$,
we have
$$\Big\vert \sum_{g\in G(k')} \psi \Big(\mathrm{Tr}_{k'/k}\Big(\sum_{j=1}^N 
\mathrm{Tr}(A_{j}\rho_j(g))\Big)\Big)\Big\vert\leq  q'^{\frac{d}{2}}d!\int_{\Delta_\infty\cap\mathfrak C} 
\prod_{\alpha \in R^+} 
\frac{\lambda(H_\alpha)^2}{\rho(H_\alpha)^2}\mathrm d\lambda,$$ where $q'$ is the number of elements in $k'$. 
\end{theorem}

The paper is organized as follows. In section 1, we express the hypergeometric sheaves in terms of the Deligne-Fourier transform. 
Motivated by this expression, we introduce the hypergeometric sheaves over a general base
using the theory of the Fourier transform over a general base developed by Wang (\cite{Wang}). Over the field of complex numbers, 
hypergeometric sheaves are the solution complexes of the hypergeometric $\mathcal D$-modules. 
In section 2, we calculate the 
$\mathcal D$-modules involved in the definition of hypergeometric $\mathcal D$-modules. 
In section 3, we prove Theorem \ref{thm:hypD}. In section 4, we use information on the hypergeometric $\mathcal D$-modules 
over the characteristic 0 ground field $K$ to get results for the hypergeometric sheaves over the field $k$ of characteristic $p$.
We deduce Theorem \ref{thm:expsum} from the properties of the hypergeometric sheaves. In the Appendix, we prove some
technical results used in the paper. 

\section{Fourier transformations and hypergeometric sheaves}

\subsection{The Deligne-Fourier transform}

In this subsection, we work with a ground field $F$ containing the finite field is $k$.
Let $S$ be a scheme of finite type over $F$, let 
$\mathbb V$ be a vector bundle over $S$ of rank $m$, and let $\mathbb V^\vee$ be the dual vector bundle of $\mathbb V$. 
Recall that the \emph{Deligne-Fourier transform} associated to a nontrivial additive
character $\psi:k\to\overline{\mathbb Q}_\ell^*$  is the functor 
$$\mathfrak F_{\psi}: D_c^b(\mathbb V,\overline{\mathbb Q}_\ell)\to D_c^b(\mathbb V^\vee,\overline {\mathbb Q}_\ell),\quad
\mathfrak F_{\psi}(K)=Rp_{2!}(p_1^\ast K\otimes \langle\;,\;\rangle^\ast \mathcal
L_{\psi})[m],$$ where
$$p_1: \mathbb V\times_S \mathbb V^\vee\to V,\quad
p_2:\mathbb V\times_S\mathbb  V^*\to \mathbb V^\vee$$ are the
projections, and $\langle\,,\,\rangle$ is the morphism defined by the duality pairing 
$$\langle\,,\,\rangle:\mathbb V\times_S \mathbb V^\vee \to\mathbb
A_k^1.$$ Confer \cite{L} for properties of the
Deligne-Fourier transform.

\subsection{The Wang-Fourier transform}

In \cite{Wang}, Wang introduces a Fourier transform functor for monodromic sheaves 
on a vector bundle over a scheme $S$ which is of finite type over a field $F$ of arbitrary characteristic, 
compatible with the Deligne-Fourier transform for monodromic $\ell$-adic sheaves in characteristic $p$, and with 
the usual Fourier transform for monodromic $\mathcal D$-modules in characteristic $0$ via the Riemann-Hilbert correspondence. 
Wang's construction can be generalized to the case where the ground field $F$ is replaced by a regular noetherian 
scheme of dimension $\leq 1$.

\begin{definition} Let $T$ be a regular noetherian scheme of dimension $\leq 1$, $S$ a separated scheme of finite type over $T$, 
$\mathbbm 1: S\to \mathbb{A}_S^1$ the $1$-section of $\mathbb A_S^1$ defined by the $\mathcal O_S$-algebra morphism 
$$\mathcal O_S[t]\to \mathcal O_S, \quad t\mapsto 1,$$
$u:\mathbb{A}_S^1-\{1\}\hookrightarrow\mathbb{A}_S^1$ the open immersion for the complement of the 1-section, and 
$$B=Ru_* \overline{\mathbb Q}_\ell.$$ 
For any vector bundle $\mathbb V$ of rank $m$ over $S$, let $\mathbb V^\vee$ be its dual, 
$p_1:\mathbb V^\vee \times_S \mathbb V \to \mathbb V$ and $p_2: 
\mathbb V^\vee \times_S \mathbb V \to \mathbb V^\vee$ 
the projections, and $\langle\;,\;\rangle : \mathbb V^\vee \times_S \mathbb V \to \mathbb{A}_S^1$
the canonical pairing. The \emph{Wang-Fourier transform} is defined to be 
\begin{eqnarray*}
\mathrm{Four}_{B}:D_{c}^{b}(\mathbb V, \overline {\mathbb Q}_\ell)\rightarrow  D_{c}^{b}(\mathbb V^{\vee},\overline {\mathbb Q}_\ell),\quad 
\mathrm{Four}_{B}(K)= Rp_{2!}(p_1^*K\otimes\langle\;,\;\rangle^*B)[m].
\end{eqnarray*}
\end{definition}

\begin{remark} We assume the base $T$ is a regular noetherian scheme of dimension $\leq 1$ so that we can apply 
\cite[Finitude 1.5-1.7]{SGA41/2} to conclude that the Grothendieck six operators are defined on the category 
$D^b_c(\hbox{-}, \overline{\mathbb Q}_\ell)$. 
\end{remark}

\begin{definition} Let 
$\mathrm{pr}_2:\mathbb{G}_{m,S}\times_S \mathbb V\to  \mathbb V$ be the projection, and let 
$\theta(n)$ be the morphism 
$$\theta(n): \mathbb{G}_{m,S}\times_S  \mathbb V\to  \mathbb V,\quad (\lambda,v)\to \lambda^n v$$
An object $K\in D_c^{b}( \mathbb V, \overline {\mathbb Q}_\ell)$ is called \emph{monodromic} if there exists an integer $n$ invertible on $S$ 
such that we have an isomorphism $\theta(n)^*K\cong\mathrm{pr}_2^*K$.
Denote the full subcategory of $D_c^{b}( \mathbb V, \overline {\mathbb Q}_\ell)$ formed by monodromic objects by  
$D_{\mathrm{mon}}^b( \mathbb V, \overline {\mathbb Q}_\ell)$.
\end{definition}

\begin{proposition}\label{WangDeligne} ${}$ 

\begin{enumerate}[(i)]
\item $\mathrm{Four}_{B}$ commutes with any base change $S'\to S$.

\item Suppose $S$ is a scheme of finite type over an algebraically closed field $F$ containing the finite field $k$.  
Restricting to monodromic objects, we have an isomorphism of functors  
$$\mathfrak {F}_{\psi} \cong \mathrm{Four}_{B}: D_{\mathrm{mon}}^b(\mathbb V,\overline{\mathbb{Q}}_{\ell})\to
D_{\mathrm{mon}}^b(\mathbb V^\vee,\overline{\mathbb{Q}}_{\ell}).$$ 
\end{enumerate}
\end{proposition}

\begin{proof} (i) 
We have a distinguished triangle
$$\begin{array}{ccc}
\mathbbm 1_*\overline{\mathbb{Q}}_{\ell}(-1)[-2]\to \overline{\mathbb{Q}}_{\ell}\to &Ru_*\overline{\mathbb{Q}}_{\ell}&\to.\\
&\parallel&\\
&B&
\end{array}$$
So the formation of $B$ and hence $\mathrm{Four}_{B}$ commute with any base change $S'\to S$. 

(ii) This is \cite[Corollary 4.5]{Wang}.
\end{proof}

\subsection{Fourier transforms for $\mathcal D$-modules}\label{FourD}

Suppose $\overline K$ is an algebraically closed field of characteristic $0$, and 
$\mathbb V$ a vector space over $\overline K$ of dimension $m$. We have an isomorphism
$$\Gamma(\mathbb V^\vee, \mathcal D_{\mathbb V^\vee})\to\Gamma(\mathbb V, 
\mathcal D_{\mathbb V}),\quad x'_i\mapsto -\partial_{x_i}, \quad 
\partial_{x'_i}\mapsto x_i,$$
where $(x_1,\ldots, x_m)$ is a linear coordinate on $\mathbb V$, and $(x'_1,\ldots,x'_m)$ is the dual coordinate on $\mathbb V^\vee$. 
It transforms quasi-coherent right $ \mathcal D_{\mathbb V}$-modules to quasi-coherent 
right $\mathcal D_{\mathbb V^\vee}$-modules, which we call the 
\emph{Fourier transform}. It transforms coherent (resp. holonomic) right $\mathcal D_{\mathbb V}$-modules 
to coherent (resp. holonomic) right $\mathcal D_{\mathbb V^\vee}$-modules, and it can be extended to an 
exact functor
$$\mathfrak F: D^b_{qc}(\mathcal D_{\mathbb V}^{\mathrm{op}})\to  D^b_{qc}(\mathcal D_{\mathbb V^\vee}^{\mathrm{op}}).$$
On $\mathbb A_K^1=\mathrm{Spec}\,\overline K[x]$
we have right $\mathcal D$-modules
\begin{eqnarray*}
\mathcal L:=\mathcal D_{\mathbb A^1_{\overline K}}/(1-\partial_x)\mathcal D_{\mathbb A^1_{\overline K}},\quad
\mathcal B:=\mathcal D_{\mathbb A^1_{\overline K}}/\partial_x(x-1)\mathcal D_{\mathbb A^1_{\overline K}}.
\end{eqnarray*}
By \cite[Lemme 7.1.4]{KL}, the Fourier transform $\mathfrak F$ can be identified with
$$\mathfrak F_{\mathcal L}:  D^b_{qc}(\mathcal D_{\mathbb V}^{\mathrm{op}})\to  
D^b_{qc}(\mathcal D_{\mathbb V^\vee}^{\mathrm{op}}),\quad 
\mathfrak F_{\mathcal L}(\mathcal K)
=R\mathrm{pr}_+^\vee(R\mathrm{pr}^!\mathcal K\otimes^!R\langle\;,\;\rangle^!\mathcal L)[1-m].$$
The \emph{Wang-Fourier transform} is 
$$\mathrm{Four}_{\mathcal B}:  D^b_{qc}(\mathcal D_{\mathbb V}^{\mathrm{op}})\to  
D^b_{qc}(\mathcal D_{\mathbb V^\vee}^{\mathrm{op}}),\quad 
\mathrm{Four}_{\mathcal B}(\mathcal K)=R\mathrm{pr}_+^\vee(R\mathrm{pr}^!\mathcal K\otimes^!R\langle\;,\;
\rangle^!\mathcal B)[1-m].$$
An object $\mathcal K$ in $D_{c}^b(\mathcal D_{\mathbb V}^{\mathrm{op}})$ is called 
\emph{monodromic} if for any $i$, any global section of 
$\mathcal H^i(\mathcal K)$ is annihilated by some nonzero polynomial of the operator 
$\sum_{j=1}^m x_j\partial_{x_j}$.
Wang proves the following.

\begin{proposition}[\text{\cite[Corollary 5.8]{Wang}}]\label{isoFour}  Restricting to monodromic objects, we have an isomorphism of functors 
$$\mathfrak F_{\mathcal L}\cong\mathrm{Four}_{\mathcal B}: D^b_{\mathrm{mon}}(\mathcal D_{\mathbb V}^{\mathrm{op}})\to 
D^b_{\mathrm{mon}}(\mathcal D_{\mathbb V^\vee}^{\mathrm{op}}).$$ 
\end{proposition}

Assume $\overline K=\mathbb C$ is the complex number field. 
For any smooth algebraic variety $X$ over $\mathbb C$, 
let $X^{\mathrm{an}}$ be the complex analytic variety associated to $X$, 
let $\mathcal O_{X^{\mathrm{an}}}$ be the sheaf of holomorphic functions,
and let $\mathcal D_{X^{\mathrm{an}}}$ be
the sheaf of holomorphic differential operators. For any $\mathcal K\in\mathrm{ob}\,D_c^b(\mathcal D_X^{\mathrm{op}})$, 
let $\mathcal K^{\mathrm{an}}\in \mathrm{ob}\, D_c^b(\mathcal 
D_{X^{\mathrm{an}}}^{\mathrm{op}})$ be its analytification. The \emph{solution complex} of $\mathcal K$ 
is defined to be  
$$\mathrm{Sol}_X(\mathcal K):=
R\mathcal Hom_{\mathcal D_{X^{\mathrm{an}}}}(\mathcal K^{\mathrm{an}}, \omega_{X^{\mathrm{an}}}),$$ 
where $\omega_{X^{\mathrm{an}}}$ is the sheaf of holomorphic top forms on $X^{\mathrm{an}}$ viewed as a right 
$\mathcal D_{X^{\mathrm{an}}}$-module. 
The \emph{de Rham complex} of an object $\mathcal M\in\mathrm{ob}\, D^b_c(\mathcal D_X)$ is defined to be
$$DR_X(\mathcal M):= \omega_{X^{\mathrm{an}}} \otimes^L_{\mathcal D_{X^{\mathrm{an}}}} \mathcal M^{\mathrm{an}}.$$
By \cite[Proposition 4.7.4]{Hottabook}, we have
$$\mathrm{Sol}_X(\mathcal K)\cong DR_X(\mathbb D_X\mathcal K)[-\mathrm{dim}\,X],$$
where $$\mathbb D_X: D^b_c(\mathcal D^{\mathrm{op}}_X)\to D^b_c(\mathcal D_X),\quad 
\mathbb D_X(\mathcal K)=R\mathcal Hom_{\mathcal D_X}(\mathcal K,\mathcal D_X)[\mathrm{dim}\, X]$$ 
is the duality functor. 
For any complex vector space $\mathbb V$ of rank $m$, we also have the \emph{Wang-Fourier transform} 
$$\mathrm{Four}_B: D_c^b(\mathbb V, \mathbb C)\to D_c^b(\mathbb V^\vee, \mathbb C), \quad 
\mathrm{Four}_B(K)= Rp_{2!}^\vee(p_1^*K\otimes\langle\;,\;\rangle^*B)[m],$$
where $B=Ru_* \mathbb C$, $D_c^b(\mathbb V, \mathbb C)$ 
is the derived category of bounded complexes of sheaves of complex vector spaces with constructible cohomology
with respect to the usual topology on complex analytic varieties. 

\begin{proposition}\label{solFour} 
For any $\mathcal K\in \mathrm{ob}\,  D^b_{{rh}}(\mathcal D^{\mathrm{op}}_{\mathbb V})$, we have an isomorphism
$$\mathrm{Four}_{B}(\operatorname{Sol}(\mathcal K))\cong \operatorname{Sol}(\mathrm{Four}_{\mathcal{B}}(\mathcal K)).$$
\end{proposition}

\begin{proof}  We have $\mathcal B=Ru_! Ru^!\omega_{\mathbb A^1_{\mathbb C}}.$ (Confer
Remark \ref{mathcalB} in the next section.)
By \cite[Theorem 7.1.1]{Hottabook},
we have 
\begin{align*}
\mathrm{Sol}(\mathcal B)\cong\; &DR(\mathbb D(Ru_! Ru^!\omega_{\mathbb A^1_{\mathbb C}}))[-1]
\cong DR(Ru_+ Ru^+ \mathbb D(\omega_{\mathbb A^1_{\mathbb C}}))[-1]\\
\cong \;& Ru_* u^* DR(\mathbb D(\omega_{\mathbb A^1_{\mathbb C}}))[-1]\cong Ru_* u^* \mathrm{Sol}(\omega_{\mathbb A^1_{\mathbb C}})
\cong  Ru_* u^*\mathbb C=B,
\end{align*}
Let $\Delta$ be the diagonal morphism for $\mathbb V\times \mathbb V$. Again by \cite[Theorem 7.1.1]{Hottabook}, we have 
\begin{align*}
\operatorname{Sol}(\mathrm{Four}_{\mathcal{B}}(\mathcal K))\cong\;& 
\operatorname{Sol}\Big(R\mathrm{pr}^\vee_+ R\Delta^!
\Big(R\mathrm{pr}^! \mathcal K\boxtimes^L R\langle\;,\;\rangle^! \mathcal B\Big)[1-m]\Big)\\
\cong\;&DR \mathbb D \Big(R\mathrm{pr}^\vee_+R \Delta^!\Big(R\mathrm{pr}^! \mathcal K\boxtimes^L R\langle\;,\;\rangle^! \mathcal B\Big)\Big)
[-1]\\
\cong\;&DR  \Big(R\mathrm{pr}^\vee_! R\Delta^+ \Big(R\mathrm{pr}^+ \mathbb D(\mathcal K)\boxtimes^LR \langle\;,\;\rangle^+ 
\mathbb D(\mathcal B)\Big)\Big)[-1]\\
\cong\;& R\mathrm{pr}^\vee_! \Delta^* \Big(\mathrm{pr}^*DR(\mathbb D(\mathcal K))\boxtimes^L \langle\;,\;\rangle^* 
DR(\mathbb D(\mathcal B))\Big)[-1]\\
\cong\;& R\mathrm{pr}^\vee_! (\mathrm{pr}^*\mathrm{Sol}(\mathcal K)\otimes^L \langle\;,\;\rangle^* \mathrm{Sol}(\mathcal B))[n]
=\mathrm{Four}_B(\mathrm{Sol}(\mathcal K)).\qedhere
\end{align*}
\end{proof}

\begin{remark} Note that the argument in the proof of Proposition \ref{solFour} does not apply to the usual Fourier transform 
$\mathfrak F_{\mathcal L}$ since $\mathcal L$ is irregular and the Riemann-Hilbert correspondence is not applicable. 
But combining Propositions \ref{isoFour}
and \ref {solFour} together, we get 
$$\mathrm{Four}_{B}(\operatorname{Sol}(\mathcal K))\cong \operatorname{Sol}(\mathfrak F_{\mathcal{L}}(\mathcal K))$$
for any $\mathcal K\in\mathrm{ob}\,D^b_{\mathrm{mon}}(\mathcal D^{\mathrm{op}}_{\mathbb V})\cap \mathrm{ob}\,
D^b_{{rh}}(\mathcal D^{\mathrm{op}}_{\mathbb V})$. 
\end{remark}

Propositions \ref{perverse} and \ref{holonomic} follows from the following. 

\begin{proposition}\label{prop:perverseholonomic}
Let $G$ be a reductive group over a field $F$, 
$\rho_j: G\to\mathrm{GL}(V_j)$ $(j=1, \ldots, N)$ a family of representations of $G$, 
$\mathbb V=\prod_{j=1}^N\mathrm{End}(V_j)$, $\iota$ the morphism
$$\iota:G\to \prod_{j=1}^N\mathrm{End}(V_j), \quad g\mapsto (\rho_1(g),\ldots, \rho_N(g)),$$ 
$m=\mathrm{dim}\,V$ and $d=\mathrm{dim}\,G$. 

\begin{enumerate}[(i)]
\item Suppose $F$ is a field of containing the finite field $k$. Let $\psi: k\to \overline{\mathbb Q}_\ell$ be 
a nontrivial additive character. We have
$$\mathrm{Hyp}_{\psi,!}\cong \mathfrak F_{\psi}(\iota_!
\overline{\mathbb Q}_\ell[d]),\quad \mathrm{Hyp}_{\psi,*}\cong \mathfrak F_{\psi}(R\iota_*
\overline{\mathbb Q}_\ell[d])$$ If $\iota$ is a quasi-finite morphism, then $\mathrm{Hyp}_{\psi,!}$ and 
$\mathrm{Hyp}_{\psi,*}$ are perverse sheaves. Suppose furthermore that $F$ is a finite field. Then
$\mathrm{Hyp}_{\psi,!}$ is mixed of weights $\leq d+m$.  
\item Suppose $F$ is an algebraically closed field of characteristic $0$. We have
$$\mathcal Hyp_+\cong \mathfrak F_{\mathcal L}(R\iota_+ \omega_G),
\quad \mathcal Hyp_!\cong \mathfrak F_{\mathcal L}(R\iota_!
\omega_G)$$ If $\iota$ is a quasi-finite morphism, then $\mathcal Hyp_+$ and 
$\mathcal Hyp_!$ are holonomic right $\mathcal D_{\mathbb V}$-modules.
\end{enumerate}
\end{proposition}

\begin{proof} We can identify $\mathbb V$ with its dual
$\mathbb V^\vee$ via the pairing 
$$\langle\;,\;\rangle: \prod_{j=1}^N\mathrm{End}(V_j)\times \prod_{j=1}^N\mathrm{End}(V_j) \to \mathbb A^1, 
\quad ((A_1, \ldots, A_N), (A'_1, \ldots, A'_N))\mapsto \sum_{j=1}^N \mathrm{Tr}(A_jA'_j).$$
Fix notation by the following commutative diagram, where all squares
are Cartesian:
$$\begin{tikzcd}
G\times_k\mathbb V\arrow[r, "\iota\times \mathrm
{id}"]\arrow[d,"\pi_1"] &\mathbb V\times_k\mathbb  V\arrow[r, "p_2"]\arrow[d, "p_1"]& \mathbb V\arrow[d]\\
G\arrow[r, "\iota"]&\mathbb V\arrow[r]&\mathrm{Spec}\,k.
\end{tikzcd}$$

(i) By the proper base change theorem and the projection formula, we
have
\begin{eqnarray*}
\mathfrak F_\psi(\iota_! \overline{\mathbb Q}_\ell[d]) &\cong& Rp_{2!} \big(p_1^\ast
\iota_!\overline{\mathbb Q}_\ell\otimes
\langle\;,\;\rangle^\ast \mathcal L_\psi\big)[d+m]\\
&\cong&  Rp_{2!}\big((\iota\times\mathrm{id})_! \overline{\mathbb Q}_\ell \otimes
\langle\;,\;\rangle^\ast \mathcal L_\psi\big)[d+m]\\
&\cong& Rp_{2!}(\iota\times\mathrm{id})_! (\iota\times\mathrm{id})^\ast\langle\;,\;\rangle^\ast
\mathcal L_\psi[d+m].
\end{eqnarray*}
We have $p_2(\iota\times\mathrm{id})=\pi_2$ and
$\langle\;,\;\rangle\circ (\iota\times\mathrm{id})=F$, where
$\pi_2:G\times_k \mathbb V\to \mathbb V$ is the
projection and $F$ is the morphism $$G\times \prod_{j=1}^N\mathrm{End}(V_j)\to \mathbb A^1,
\quad F(g,(A_j))=\sum_{j=1}^N \mathrm{Tr}(A_j\rho_j(g)).$$ So we have
$$\mathfrak F_\psi(\iota_!\overline{\mathbb Q}_\ell[d])\cong 
R\pi_{2!}F^\ast\mathcal
L_\psi[d+m]=\mathrm{Hyp}_{\psi,!}.$$
Note that $\iota$ is an affine morphism. 
If $\iota$ is quasi-finite, then $\iota_!\overline{\mathbb Q}_\ell[d]$ is a perverse sheaf by 
\cite[4.1.3]{BBD}. By \cite[1.3.2.3]{L}, $\mathrm{Hyp}_{\psi,!}\cong \mathfrak F_\psi(\iota_!\overline{\mathbb Q}_\ell[d])$
is perverse. By \cite[3.3.1]{D}, $\mathrm{Hyp}_{\psi, !}\cong R\pi_{2!}F^\ast \mathcal
L_\psi[d+m]$ is mixed of weights $d+m$. The assertions for $\mathrm{Hyp}_{\psi,*}$ follows by duality.

(ii) By the base change theorem \cite[Theorem 1.7.3]{Hottabook}, the projection formula \cite[Corollary 1.7.5]{Hottabook} and the fact that 
$\mathcal L$ is an integrable connections, we have 
\begin{eqnarray*}
\mathfrak F_{\mathcal L}(R\iota_+ \omega_G)
&\cong& Rp_{2+} \Big(Rp_1^! R\iota_+ \omega_G\otimes
_{\mathcal O_{\mathbb V\times_k\mathbb V}}^! R\langle\,,\,\rangle^! \mathcal L\Big)[1-m]\\
&\cong& Rp_{2+} \Big(R(\iota\times\mathrm{id})_+ R\pi_1^! \omega_G
\otimes_{\mathcal O_{\mathbb V\times_k \mathbb V}}^! R\langle\,,\,\rangle^! \mathcal L\Big)[1-m]\\
&\cong& R(p_{2} \circ (\iota\times \mathrm{id}))_+ \Big(R\pi_1^!\omega_G\otimes^!
_{\mathcal O_{G\times_k\mathbb V}} 
R(\langle\,,\,\rangle\circ (\iota\times \mathrm{id}))^! \mathcal L\Big)[1-m]\\
&\cong& R\pi_{2+} (R\pi_1^! \omega_G\otimes^!_{\mathcal O_{G\times_k\mathbb V}} RF^! \mathcal L)[1-m]
\cong R\pi_{2+} F^* \mathcal L
\cong\mathcal Hyp_{+}.
\end{eqnarray*}
Suppose $\iota$ is quasi-finite. Let $\tilde G$ be the image of $\iota$, which is also an algebraic group. 
We can factorize $\iota$ as the composite 
$$G\to \tilde G\to \prod_{j=1}^N \mathrm{GL}(V_j)\to \prod_{j=1}^N \mathrm{End}(V_j),$$
where each morphism is either finite, or an affine open immersion. For an affine open immersion $j: U\to X$, we have
$Rj_+\mathcal N=j_*\mathcal N$. If $\mathcal N$ is a holonomic $\mathcal D_U$-module, then 
$Rj_+\mathcal N$ lies in the derived category $D^b_h(\mathcal D_X)$ by \cite[Theorem 3.2.3]{Hottabook}. 
So $Rj_+\mathcal N$ is a holonomic $\mathcal D_X$-module.
For a finite morphism $f:X\to Y$, we have 
$$Rf_{+}(-)=f_*(-\otimes^L_{\mathcal D_{\overline X}}\mathcal D_{X\to Y}).$$
It follows that $Rf_+$ maps $D^{\leq 0}_h(\mathcal D_X)$ to $D^{\leq 0}_h(\mathcal D_Y)$. 
On the other hand, we have $$Rf_+\cong Rf_!\cong \mathbb D Rf_+\mathbb D,$$ 
where $\mathbb D$ is the duality functor.
So $Rf_+$ maps $D^{\geq 0}_h(\mathcal D_X)$ to $D^{\geq 0}_h(\mathcal D_Y)$. 
Hence for any holonomic $\mathcal D_X$-module $\mathcal M$, 
$Rf_+\mathcal M$ is a holonomic $\mathcal D_Y$-module. Thus 
$R\iota_+\omega_G$ is a holonomic $\mathcal D_{\mathbb V}$-module. Therefore its Fourier transform $\mathcal Hyp_+$ 
is a holonomic $\mathcal D_{\mathbb V}$-module. The assertion for $\mathcal Hyp_!$ follows by duality. 
\end{proof}

Let $R$ be a Dedekind domain so that its fraction field $K$ 
is of characteristic $0$ and its reside fields at maximal ideals are finite of characteristic distinct from $\ell$, 
$G_R$ a split reductive group scheme over $R$, 
$\rho_{j,R}: G_R\to\mathrm{GL}(V_{j,R})$ $(j=1, \ldots, N)$ representations such that the representations 
$\rho_{j, \overline K}: G_{\overline K}\to \mathrm{GL}(V_{j,\overline K})$ are irreducible, and the morphism 
$$\iota_R: G_R\to \mathbb V_R=\prod_{j=1}^N \mathrm{End}_R(V_{j, R}), \quad g\mapsto (\rho_{1,R}(g),\ldots, \rho_{N,R}(g))$$
is quasi-finite.

\begin{definition} \label{homocond} 
We say the \emph{the homogeneity condition} holds over $R$ if $\mathbb G_{m, R}$ is contained in the center of $G_R$ and
for each $j$, the composite $$\mathbb G_{m, R}\hookrightarrow G_R\stackrel{\rho_{j, R}}\to \mathrm{GL}(V_{j, R})$$ coincides with 
the homomorphism
$$\mathbb G_{m, R}\to \mathrm{GL}(V_{j, R}),\quad 
\lambda \mapsto \lambda\,\mathrm{id}_{V_{j,R}}.$$
Similarly we can define the homogeneity condition over $K$. 
\end{definition}

\begin{definition}\label{definitionhypD} Suppose the homogeneity condition holds over $R$. 

\begin{enumerate}[(i)]
\item The $\ell$-adic 
\emph{hypergeometric sheaves} $\mathrm{Hyp}_{!,R}, \mathrm{Hyp}_{*,R}\in\mathrm{ob}\,D_c^b(\mathbb V_R, \overline{\mathbb Q}_\ell)$ 
over $R$ are defined to be $$\mathrm{Hyp}_{!, R}=\mathrm {Four}_B(R\iota_{R, !}\overline{\mathbb Q}_\ell[\mathrm{dim}\,G]),
\quad \mathrm{Hyp}_{*, R}=\mathrm{Four}_B(R\iota_{R, *}\overline{\mathbb Q}_\ell[\mathrm{dim}\,G]),$$
where $\mathrm{Four}_B$ is 
the Wang-Fourier transform for the vector bundle $\mathbb V_R$ over $\mathrm{Spec}\,R$. 

\item Suppose $\mathrm{Spec}\,\mathbb C\to \mathrm{Spec}\,R$ is a $\mathbb C$-point over the generic point of $\mathrm{Spec}\,R$. 
The \emph{hypergeometric $\mathcal D$-modules $\mathcal Hyp_+,  \mathcal Hyp_!\in 
\mathrm{ob}\,D_{qc}^b(\mathcal D_{V_{\mathbb C}})$} are defined to be 
$$\mathcal Hyp_+=\mathrm{Four}_{\mathcal B}(R\iota_{\mathbb C, +}\omega_{G_{\mathbb C}}),
\quad \mathcal Hyp_!=\mathrm{Four}_{\mathcal B}(R\iota_{\mathbb C, !}\omega_{G_{\mathbb C}}),$$ 
where $\mathrm{Four}_{\mathcal B}$ is the Wang-Fourier transform for right $\mathcal D$-modules on the vector bundle $\mathbb V_{\mathbb C}$ over 
$\mathrm{Spec}\,\mathbb C$, and $\iota_{\mathbb C}$ is the base change of $\iota_R$.
By Proposition \ref{basichyp_D} below, $R\iota_{\mathbb C, !}\omega_{G_{\mathbb C}}$ and 
$R\iota_{\mathbb C, *}\omega_{G_{\mathbb C}}$ are monodromic. So we have
$$\mathcal Hyp_+\cong\mathfrak F_{\mathcal L}(R\iota_{\mathbb C, +}\omega_{G_{\mathbb C}}),
\quad \mathcal Hyp_!\cong\mathfrak F_{\mathcal L}(R\iota_{\mathbb C, !}\omega_{G_{\mathbb C}}).$$ 
By Proposition \ref{prop:perverseholonomic} (ii), this definition of hypergeometric $\mathcal D$-modules coincides
with the one in Definition \ref{defn:hypergeometricD}. 
We define the \emph{complex hypergeometric sheaves} $\mathrm{Hyp}_{!, \mathbb C}$ and $\mathrm{Hyp}_{*, \mathbb C}$ to be
\begin{eqnarray*}
\mathrm{Hyp}_{!, \mathbb C}:=\mathrm {Four}_B(R\iota_{\mathbb C, !}\mathbb C[\mathrm{dim}\,G]),\quad 
\mathrm{Hyp}_{*, \mathbb C}:=\mathrm{Four}_B(R\iota_{\mathbb C, *}\mathbb C[\mathrm{dim}\,G]).
\end{eqnarray*}
\end{enumerate}
\end{definition} 

\subsection{Equivariant objects}\label{equiv} Let $H$ be a group $S$-scheme, $\mu: H\times_S H\to H$ the multiplication 
on $H$, $X$ an $S$-scheme provided with an $H$-action $A:H\times_S X\to X$, and 
$$p_2:H\times_S X\to X, \quad p_{23}: H\times_SH\times_SX\to H\times_SX$$ the projections. An 
\emph{$H$-equivariant object} in $D_c^b(X,\overline{\mathbb Q}_\ell)$ is a pair $(K, \phi)$ consisting of an object 
$K$ in $D_c^b(X,\overline{\mathbb Q}_\ell)$ and an isomorphism $$\phi:A^*K\stackrel\cong \to p_2^*K$$ such that 
$$(\mu\times\mathrm{id}_X)^*(\phi)=p_{23}^*(\phi)\circ (\mathrm{id}_H\times A)^*(\phi),$$ 
that is, the following diagram commutes:
$$ \begin{array}{ccccccc}
(\mathrm{id}_H\times A)^*A^*K&&&\cong&&&    (\mu\times\mathrm{id}_X)^*A^*K \\
{\scriptstyle (\mathrm{id}_H\times A)^*(\phi)}\downarrow\qquad&&&&&&\qquad\downarrow {\scriptstyle  (\mu\times\mathrm{id}_X)^*(\phi)}\\
(\mathrm{id}_H\times A)^*p_2^*K&\cong&p_{23}^*A^*K&\stackrel{p_{23}^*(\phi)}\to
&p_{23}^*p_2^*K&\cong &    (\mu\times\mathrm{id}_X)^*p_2^*K.
\end{array}
$$
We call the last condition the \emph{cocycle condition}. Similarly, we can define $H$-equivariant objects in
$D_c^b(\mathcal D_X^{\mathrm{op}})$. One can show equivariant objects are monodromic. 
In the case where $S=\mathrm{Spec}\,\overline K$ for an algebraically closed field $\overline K$ of characteristic $0$, an equivariant 
object $\mathcal K$ in $D_c^b(\mathcal D_X^{\mathrm{op}})$ lies in $D_{{rh}}^b(\mathcal D_X^{\mathrm{op}})$ if
$\mathrm{supp}\,\mathcal H^i(K)$ is a union of finitely many orbits for each $i$ by \cite[II 5 Theorem]{Hotta}.

\begin{proposition}\label{basichyp_D} Suppose the homogeneity condition holds over $R$.
Let $\mathrm{Spec}\,\mathbb C\to \mathrm{Spec}\,R$ be a $\mathbb C$-point over the generic point of $\mathrm{Spec}\,R$. 

\begin{enumerate}[(i)]
\item $\iota_{R, !}\overline{\mathbb Q}_\ell$ and $R\iota_{R, *}\overline{\mathbb Q}_\ell$ are equivariant objects in
$D_c^b(\mathbb V,\overline{\mathbb Q}_\ell)$, and $R\iota_{\mathbb C, !}\omega_{G_{\mathbb C}}$ and 
$R\iota_{\mathbb C, *}\omega_{G_{\mathbb C}}$ are equivariant objects in 
$D^b_{{rh}}(\mathcal D_{\mathbb V_{\mathbb C}}^{\mathrm{op}})$.
They are all monodromic.

\item Let $k$ be the residue field of $R$ at a maximal ideal $\mathfrak m$. We have 
$$\mathrm{Hyp}_{!, R}|_{V_{\bar k}}\cong \mathrm{Hyp}_{\psi, !}|_{V_{\bar k}}.$$ 

\item We have 
$$\mathrm{Sol}(\mathcal Hyp_+)\cong \mathrm{Hyp}_{!,\mathbb C},
\quad \mathrm{Sol}(\mathcal Hyp_!)=\mathrm{Hyp}_{*,\mathbb C}.$$ 
\end{enumerate}
\end{proposition}

\begin{proof} By the homogeneity condition, $\mathbb G_{m,R}$ can be viewed as a subgroup scheme contained in 
the center of $G_R$. Let $\mathbb G_{m,R}$ act on $G_R$ by left multiplication, and on 
$\mathbb V_R$ by scalar multiplication. The homogeneity condition implies that the morphism $\iota_R: G_R\to \mathbb V_R$ is 
$\mathbb G_{m,R}$-equivariant. 
The sheaf $\overline{\mathbb Q}_\ell$ is a $\mathbb G_{m,R}$-equivariant sheaf on $G_R$. So 
$\iota_{R, !}\overline{\mathbb Q}_\ell$ is a $\mathbb G_{m,R}$-equivariant sheaf on $\mathbb V_R$ and hence monodromic. The other 
assertions in (i) can be proved by the same argument. (ii) then follows from Proposition \ref{WangDeligne}, and (iii) follows from 
Proposition \ref{solFour}.
\end{proof}

\section{Calculation on $D$-modules}

In this section we work over an algebraically closed field $\overline K$ of characteristic zero.
Let $H$ be a reductive algebraic group over $\overline K$, let $B$ be a Borel subgroup of $H$,
and let $Y$ be an \emph{$H$-spherical variety}, that is, $Y$ is a normal variety with a left $H$-action
containing an open $B$-orbit $By_0$ for some $\overline K$-point $y_0$ in $Y$. 
For surveys of spherical varieties, see \cite{Brion, Knop, Timashev}. 
Let $V$ be a smooth $\overline K$-variety provided with an $H$-action, let $f: Y\to V$ be a finite equivariant morphism, and let 
$\iota: Hy_0\to V$ be the composite $$Hy_0\hookrightarrow Y\stackrel{f}\to V.$$ 
The main result of this section Propositions \ref{sphericalDmodisogen} says that if $Hy_0$ is affine, then 
$R\iota_!\omega_{Hy_0}$ is a direct summand of an explicit $\mathcal D$-module. 
This technical result will be used in the next section. 

\subsection{}
First consider the case where $Y$ is a smooth toroidal $H$-spherical variety and $V=Y$. Let 
$D_1, \dots, D_s$ be all the $B$-stable but not $H$-stable prime divisors of $Y$, $\delta=\bigcup_{i=1}^{s}D_{i}$, and
$P=\{g\in H:g\delta=\delta\}$. By \cite[Proposition 1 in 2.4]{Brion} or \cite[Theorem 29.1]{Timashev},  
there exists a Levi subgroup $L$ of $P$ and 
a closed $L$-stable subvariety $S$ 
of $Y-\delta$ such that we have a $P$-equivariant isomorphism
\[R_{u}(P)\times S\rightarrow Y-\delta,\quad(g,y)\mapsto gy,\]
where $R_u(P)$ is the unipotent radical of $P$. 
Moreover, $[L,L]$ acts trivially on $S$ and $S$ is a toric variety for a quotient torus of $L^{\mathrm{ab}}:=L/[L,L]$ with 
the same fan as $Y$. 
For each $\overline K$-point $y$ in $S$, we have $$Hy\cap(Y-\delta)=R_{u}(P)Ly.$$ 
We choose $y_0$ lying in $S$. Then $Ly_0$ is open in $S$. 

Let $\mathcal T_Y$, $\mathcal T_S$ and $\mathcal T_{R_u(P)}$ 
be the sheaves of tangent vectors of $Y$, $S$ and $R_u(P)$, respectively. For any 
$\xi$ lying in the Lie algebra $\mathfrak L(H)$ of $H$, let $L_\xi$ be the vector field on $Y$ whose 
value at a $\overline K$-point $y$ in $Y$ is given by 
$$L_\xi(y)=\frac{\mathrm d}{\mathrm dt}\Big|_{t=0}(e^{t\xi}y).$$
Let $\mathcal T_0$ be the 
$\mathcal{O}_{Y}$-submodule of $\mathcal T_Y$ generated by $L_{\xi}$ ($\xi\in\mathfrak L(H)$), and let
$\mathcal T_1$ be the $\mathcal{O}_{S}$-submodule of $\mathcal T_S$ generated by $L_{\xi}$ for $\xi$ lying in the
Lie algebra $\mathfrak L(L)$ of $L$. Denote the Lie algebra of $R_u(P)$ by $\mathfrak L(R_u(P))$. 

\begin{lemma}\label{reducetotoric} 
$\mathcal T_0|_{Y-\delta}$ coincides with the $\mathcal{O}_{Y-\delta}$-submodule 
of $\mathcal T_{Y-\delta}$ generated by $L_{\xi}$ with $\xi\in\mathfrak L(L)\cup \mathfrak L(R_{u}(P))$, 
and $\mathcal T_0|_{Y-\delta} \cong \mathcal T_{R_{u}(P)}\boxplus \mathcal T_1$.
\end{lemma}

\begin{proof}
Let $\mathcal F$ be the $\mathcal{O}_{Y-\delta}$-submodule 
of $\mathcal T_{Y-\delta}$ generated by $L_{\xi}$ with $\xi\in\mathfrak L(L)\cup\mathfrak L(R_{u}(P))$. 
We have an inclusion
$$\mathcal F\hookrightarrow \mathcal T_0|_{Y-\delta}.$$ This inclusion is $P$-equivariant.
Note that any $P$-orbit in $Y- \delta$ is of the form $Py$ for some $\overline K$-point $y$ in $S$.
To prove the above inclusion is an isomorphism, 
it suffices to show $$\mathcal F_y/\mathfrak m_y\mathcal F_y \cong  
\mathcal T_{0,y}/\mathfrak m_y\mathcal T_{0,y}$$ for any $\overline K$-point $y$ in $S$ by Nakayama's lemma.
Since $\overline K$ is of characteristic $0$, $Hy$ and $Py$ are smooth and 
$$\mathcal T_{0,y}/\mathfrak m_y\mathcal T_{0,y}\cong T_y(Hy),\quad 
\mathcal F_y/\mathfrak m_y\mathcal F_y\cong T_y(Py),$$ where $T_y(Hy)$ and $T_y(Py)$ are the tangent spaces 
of $Hy$ and $Py$ at $y$. We may identify $T_y(Hy)$ with $T_y(Py)$ since
$$Hy\cap(Y-\delta)=R_{u}(P)Ly=Py.$$ So we have $\mathcal F_y/\mathfrak m_y\mathcal F_y \cong  
\mathcal T_{0,y}/\mathfrak m_y\mathcal T_{0,y}$. The second assertion follows from the first assertion and the fact that 
$Y-\delta\cong R_u(P)\times S$. 
\end{proof}

\begin{proposition}\label{toroidalDmodiso} Suppose $Y$ is a smooth toroidal spherical variety with the open $H$-orbit $Hy_0$. 
Let $l: Hy_0\hookrightarrow Y$ be the open immersion. In the derived category $D_h^b(\mathcal D_Y)$, 
we have an isomorphism 
$$\gamma: Rl_{!}\mathcal{O}_{Hy_0}\stackrel\cong\to \mathcal{D}_Y\big/\sum_{\xi\in \mathfrak L(H)}\mathcal{D}_YL_{\xi}.$$ 
\end{proposition}

\begin{proof} We work on the derived category of $\mathcal D$-modules consisting of bounded complexes 
with holonomic cohomologies, on which we have the six-functor formalism. By Proposition \ref{Nequi} below, 
$\mathcal{D}_Y\big/\sum_{\xi\in \mathfrak L(H)}\mathcal{D}_YL_{\xi}$ is a coherent 
equivariant left $\mathcal D_Y$-module. By \cite[II 5 Theorem]{Hotta}, it is (regular) holonomic.
We have
\begin{align*}
&\mathcal D_{Hy_0} 
\big/\sum_{\xi\in\mathfrak L(H)}\mathcal D_{Hy_0} L_\xi  
\cong \mathcal O_{Hy_0}\nonumber
\end{align*}
as left $\mathcal D_{Hy_0}$-modules. 
So we have 
$$\mathcal O_{Hy_0}\cong Rl^!\Big(\mathcal{D}_Y\big/\sum_{\xi\in \mathfrak L(H)}\mathcal{D}_YL_{\xi}\Big).$$ 
Via adjoint we get a morphism $$\gamma: Rl_!\mathcal O_{Hy_0}\to \mathcal{D}_Y\big/\sum_{\xi\in \mathfrak L(H)}\mathcal{D}_YL_{\xi}.$$
Proving $\gamma$ is an isomorphism is a local problem. 
So we may assume $Y$ is a simple spherical variety, that is, $Y$ has a unique closed $H$-orbit. 
The morphism $\gamma$ is $H$-equivariant.  Since $Y$ is toroidal, we have 
$H(Y-\delta)=Y$. So it suffices to prove $\gamma|_{Y-\delta}$ is an isomorphism. 
Fix notation by the following diagram of Cartesian squares
$$\begin{tikzcd}&Hy_0\cap (Y-\delta)\arrow[d]\arrow[r, "l",hook] &Y-\delta \arrow[d]\arrow[r]&R_{u}(P)\ar[d]\\
&Ly_0\arrow[r, "l'",hook] &S\arrow[r] &\operatorname{Spec} k.
\end{tikzcd}$$
We have $$Rl_!\mathcal{O}_{Hy_0}|_{Y-\delta}\cong \mathcal{O}_{R_{u}(P)}\boxtimes Rl'_!\mathcal{O}_{Ly_0}.$$ 
Since $Y$ is simple and smooth, so is $S$. Hence $\iota': Ly_0\hookrightarrow S$ can be identified with 
$$\mathbb{G}_m^s\times \mathbb{G}_m^t\hookrightarrow \mathbb{G}_m^s\times \mathbb{A}^t$$  
for some integers $s$ and $t$. Let $j: \mathbb G_m\hookrightarrow \mathbb A^1$ be the open immersion. We have
\begin{eqnarray}\label{leftside}
\big(Rl_!\mathcal{O}_{Hy_0}\big)|_{Y-\delta}\cong \mathcal{O}_{R_{u}(P)}
\boxtimes \mathcal{O}_{\mathbb{G}_m^s}\boxtimes \big(\boxtimes^{t}Rj_{!}\mathcal{O}_{\mathbb{G}_m}\big).
\end{eqnarray}
Let $x$ be the coordinate of $\mathbb{A}^1$.
The action of $L$ on $S$ factors through a quotient torus isomorphic to $\mathbb{G}_m^s\times \mathbb{G}_m^t$.  We have
$$\mathcal T_1\cong \mathcal T_{\mathbb{G}_m^s}\boxplus (\mathcal{O}_{\mathbb{A}^1} x\partial_{x})^{\boxplus t}.$$ 
Combined with Lemma \ref{reducetotoric}, we get
$$\mathcal T_0|_{Y-\delta}\cong \mathcal T_{R_{u}(P)}\boxplus \mathcal T_1\cong \mathcal T_{R_{u}(P)}
\boxplus \mathcal T_{\mathbb{G}_m^s}\boxplus (\mathcal{O}_{\mathbb{A}^1} x\partial_{x})^{\boxplus t}.$$ 
We conclude that 
\begin{align}\label{rightside}
&\big(\mathcal{D}_Y\big/\sum_{\xi\in\mathfrak L(H)}\mathcal{D}_YL_{\xi}\big)|_{Y-\delta}
=(\mathcal{D}_Y/\mathcal{D}_Y\mathcal T_0)\big |_{Y-\delta}\\
\cong\;& \big(\mathcal{D}_{R_{u}(P)}/\mathcal{D}_{R_{u}(P)}\mathcal T_{R_{u}(P)}\big)\boxtimes
\big(\mathcal{D}_{\mathbb{G}_m^s}/\mathcal{D}_{\mathbb{G}_m^s} \mathcal 
T_{\mathbb{G}_m^s}\big)\boxtimes (\mathcal{D}_{\mathbb{A}^1}/\mathcal{D}_{\mathbb{A}^1}x\partial_{x})^{\boxtimes t}\nonumber \\
\cong\;&  \mathcal{O}_{R_{u}(P)}\boxtimes  \mathcal{O}_{\mathbb{G}_m^s}
\boxtimes(\mathcal{D}_{\mathbb{A}^1}/\mathcal{D}_{\mathbb{A}^1}x\partial_{x})^{\boxtimes t}.\nonumber
\end{align} 
By Lemma \ref{affineline} below, we have 
$\mathcal{D}_{\mathbb{A}^1}/\mathcal{D}_{\mathbb{A}^1}x\partial_{x}\cong R j_!\mathcal{O}_{\mathbb{G}_m}$. 
Our assertion follows from (\ref{leftside}) and (\ref{rightside}).
\end{proof}

Let $j: \mathbb{G}_m\to \mathbb{A}^1$ be the open immersion. Since 
$\mathcal{D}_{\mathbb G_m} \partial_x =\mathcal{D}_{\mathbb G_m} x\partial_x$, we have canonical isomorphisms
$$\mathcal {O}_{\mathbb{G}_m}\cong \mathcal{D}_{\mathbb G_m }/\mathcal{D}_{\mathbb G_m} \partial_x
=\mathcal{D}_{\mathbb G_m }/\mathcal{D}_{\mathbb G_m} x\partial_x 
\cong Rj^!(\mathcal{D}_{\mathbb{A}^1}/\mathcal{D}_{\mathbb{A}^1}x\partial_{x}).$$
By adjoint, it induces a morphism 
$$\alpha: Rj_!\mathcal{O}_{\mathbb{G}_m}\to \mathcal{D}_{\mathbb{A}^1}/\mathcal{D}_{\mathbb{A}^1}x\partial_{x}.$$

\begin{lemma}\label{affineline} The above morphism $\alpha:
Rj_!\mathcal{O}_{\mathbb{G}_m}\stackrel\cong\to \mathcal{D}_{\mathbb{A}^1}/\mathcal{D}_{\mathbb{A}^1}x\partial_{x}$
is an isomorphism.
\end{lemma}

\begin{proof} We first construct an isomorphism 
\begin{eqnarray}\label{eqn:Rj_+}
\beta: \mathcal{D}_{\mathbb{A}^1}/\mathcal{D}_{\mathbb{A}^1}\partial_{x} x
\stackrel\cong\to Rj_+\mathcal O_{\mathbb G_m}.
\end{eqnarray}
Since $j$ is an affine open immersion, we have
$Rj_+\mathcal O_{\mathbb G_m}\cong j_* \mathcal O_{\mathbb G_m}$. The 
$\overline K[x,\partial_x]$-module corresponding to $j_* \mathcal O_{\mathbb G_m}$ is 
$\overline K[x, x^{-1}]$. It is generated by $x^{-1}$. 
We have an epimorphism 
$$\overline K[x,\partial_x]\to \overline K[x, x^{-1}], \quad P\mapsto Px^{-1}.$$
Clearly $\partial_x x$ is contained in the kernel. We claim that the kernel is exactly
$\overline K[x,\partial_x] \partial_x x$, which proves (\ref{eqn:Rj_+}). 
For any $P$ lying in the kernel, $P x^{-1}$ annihilates  $1$. The annihilator of 
$1\in \overline K[x,x^{-1}]$ over the Weyl algebra
$$\overline K[x,x^{-1}, \partial_x]\cong \Gamma(\mathbb G_m, \mathcal D_{\mathbb G_m})$$
is $\overline K[x,x^{-1}, \partial_x]\partial_x$. 
So we have $Px^{-1}\in \overline K[x,x^{-1}, \partial_x] \partial_x$ and hence $P\in \overline K[x,x^{-1}, \partial_x] \partial_xx$. Thus $x^m P\in \overline K[x,\partial_x] \partial_x x$ for some nonnegative integer 
$m$. Suppose $m\geq 1$. Let's deduce $x^{m-1} P\in \overline K[x,\partial_x] \partial_x x$, which proves the claim. We can write
$$x^m P=\Big(\sum_{i\geq 0}  a_i \partial^i_x +xQ\Big) \partial_x x$$ 
for some $a_i\in \overline K$ and $Q\in \overline K[x,\partial_x]$. We then have 
$$\sum_{i\geq 0}  a_i \partial^{i+1}_x x = x^mP- xQ \partial_x x\in x\overline K[x,\partial_x].$$ Since $\partial_x^k x=k \partial_x^{k-1}+x\partial_x^k$, we have 
$$\sum_{i\geq 0}  a_i \partial^{i+1}_x x=\sum_{i\geq 0}  (i+1) a_i \partial^i_x +\sum_{i\geq 0}  a_i x \partial^{i+1}_x,$$ and it lies in 
$x \overline K[x,\partial_x]$ only when $a_i=0$ for all $i$. So $x^{m-1} P=Q\partial_x x\in \overline K[x,\partial_x] \partial_x x$.
We thus get the isomorphism 
$\beta: \mathcal{D}_{\mathbb{A}^1}/\mathcal{D}_{\mathbb{A}^1}\partial_{x} x\stackrel\cong\to Rj_+\mathcal O_{\mathbb G_m}$. 
By adjoint, $\beta$ corresponds to a morphism 
$$Rj^+(\mathcal{D}_{\mathbb{A}^1}/\mathcal{D}_{\mathbb{A}^1}\partial_{x} x)\to \mathcal O_{\mathbb G_m},$$
which can be identified with the morphism
\begin{eqnarray*}
\mathcal{D}_{\mathbb G_m}/\mathcal{D}_{\mathbb G_m}\partial_{x} x \to 
\mathcal{D}_{\mathbb G_m}/\mathcal{D}_{\mathbb G_m}\partial_{x}, \quad P\mapsto P x^{-1}
\end{eqnarray*}
by our construction.
Taking the dual of this morphism
and switching from right $\mathcal D$-modules to left $\mathcal D$-modules, we get the identity isomorphism 
$$\mathrm{id} :\mathcal{D}_{\mathbb G_m}/\mathcal{D}_{\mathbb G_m}\partial_{x} \stackrel{=}\to 
\mathcal{D}_{\mathbb G_m}/\mathcal{D}_{\mathbb G_m}x\partial_{x}.$$
Recall that $\alpha: Rj_!\mathcal{O}_{\mathbb{G}_m}\to \mathcal{D}_{\mathbb{A}^1}/\mathcal{D}_{\mathbb{A}^1}x\partial_{x}$
is obtained by adjoint from the composite of the canonical isomorphisms
$$\mathcal {O}_{\mathbb{G}_m}\cong \mathcal{D}_{\mathbb G_m }/\mathcal{D}_{\mathbb G_m} \partial_x
\stackrel{\mathrm{id}}=\mathcal{D}_{\mathbb G_m }/\mathcal{D}_{\mathbb G_m} x\partial_x 
\cong Rj^!(\mathcal{D}_{\mathbb{A}^1}/\mathcal{D}_{\mathbb{A}^1}x\partial_{x}).$$
Thus $\alpha$ is obtained from $\beta$ by duality. Since $\beta$ is an isomorphism, so is 
$\alpha$.
\end{proof}

\begin{remark}\label{mathcalB} Let $u: \mathbb A^1-\{1\}\hookrightarrow \mathbb A^1$ be the open immersion for
the complement of the 1-section. Via the translation $x\mapsto x-1$, Lemma \ref{affineline} implies that 
$$Ru_!Ru^!\mathcal O_{\mathbb A^1}\cong \mathcal D_{\mathbb A^1}/\mathcal D_{\mathbb A^1}(x-1)\partial_x.$$ 
The right $\mathcal D$-module version of this isomorphism is 
$$Ru_!Ru^!\omega_{\mathbb A^1}\cong\mathcal D_{\mathbb A^1}/\partial_x(x-1)\mathcal D_{\mathbb A^1},$$
that is, $\mathcal B\cong Ru_!Ru^!\omega_{\mathbb A^1}$ in the notation of \ref{FourD}. 
\end{remark}

\subsection{} For any smooth variety $X$ over $\overline K$, we have canonical monomorphisms
$$\mathcal D_X\to \mathcal End_{\overline K}(\mathcal O_X), \quad \mathcal D_X\to 
\mathcal End_{\overline K}(\omega_X)^{\mathrm{op}}$$ between 
$(\mathcal O_X,\mathcal O_X)$-bialgebras induced by the canonical left $\mathcal D_X$-module 
structure on $\mathcal O_X$ and the canonical right $\mathcal D_X$-module structure on $\omega_X$.
For any sections $\omega_0$ and $\omega$ of $\omega_X$ such that $\omega_0$ is a local basis, 
let $\frac{\omega}{\omega_0}$ be the section of 
$\mathcal O_X$ such that 
$$\omega=\frac{\omega}{\omega_0}\cdot \omega_0.$$ We have a canonical homomorphism 
\begin{eqnarray*}
\omega_X\otimes_{\mathcal O_X}\mathcal End_{\overline K}(\mathcal O_X)\to  \omega_X\otimes_{\mathcal O_X}\mathcal 
End_{\overline K}(\omega_X)^{\mathrm{op}}, 
\quad \eta \otimes P\mapsto \omega_0\otimes \Big(\omega\mapsto P\Big(\frac{\omega}{\omega_0}\Big)\eta\Big).
\end{eqnarray*}
Note that the expression on the righthand side is independent of the choice of the local basis $\omega_0$. It 
induces an isomorphism 
\begin{eqnarray}\label{leftright}
\omega_X\otimes_{\mathcal O_X}\mathcal D_X\stackrel\cong \to \omega_X\otimes_{\mathcal O_X} \mathcal D_X.
\end{eqnarray}
On $\omega_X\otimes_{\mathcal O_X}\mathcal D_X$, we have the following two right $\mathcal D_X$-module structures: Let
$\omega\otimes P$ be a section of $\omega_X\otimes_{\mathcal O_X}\mathcal D_X$ and let $\theta$ be a vector field on $X$
viewed as a section of $\mathcal D_X$, 
\begin{enumerate}[(1)]
\item The naive right action is given by 
$$(\omega\otimes P)\cdot\theta=\omega\otimes P\theta.$$
\item The right action is given by Leibniz rule 
$$(\omega\otimes P)\cdot \theta=\omega \theta\otimes P-\omega\otimes \theta P.$$
\end{enumerate} 
Note that these two right $\mathcal D_X$-module structures on $\omega_X\otimes_{\mathcal O_X}\mathcal D_X$
commute with each other. One can verify the following.

\begin{proposition} The isomorphism (\ref{leftright}) switches the above two right $\mathcal D_X$-module structures on 
$\omega_X\otimes_{\mathcal O_X}\mathcal D_X$.
\end{proposition}

As a corollary of this proposition, we have the following right $\mathcal D$-module version of Proposition \ref{toroidalDmodiso}.

\begin{corollary}\label{rightD-modversion}	
Keep the assumption and notation of Proposition \ref{toroidalDmodiso}. 

(i) We have an isomorphism of right $\mathcal{D}_Y$-modules  
$$Rl_{!}\omega_{Hy_0}\cong \big(\omega_{Y}\otimes_{\mathcal{O}_Y}\mathcal D_Y \big)
\big/\sum_{\xi\in\mathfrak L(H)}L'_\xi \big(\omega_{Y}\otimes_{\mathcal{O}_Y}\mathcal D_Y \big),$$ 
where the right $\mathcal{D}_Y$-module structure is induced by that on $\omega_{Y}\otimes_{\mathcal{O}_Y}\mathcal D_Y$
given by Leibniz rule, and  
$L'_\xi$ acts on $\omega_{Y}\otimes_{\mathcal{O}_Y}\mathcal D_Y$ via the naive right $\mathcal{D}_Y$-module structure. 

(ii) We have an isomorphism of right $\mathcal{D}_Y$-modules  
$$Rl_{!}\omega_{Hy_0}\cong\big(\omega_{Y}\otimes_{\mathcal{O}_Y}\mathcal D_Y \big)\big
/\sum_{\xi\in\mathfrak L(H)}L_\xi \big(\omega_{Y}\otimes_{\mathcal{O}_Y}\mathcal D_Y \big),$$
where the right $\mathcal{D}_Y$-module structure is induced by the naive right $\mathcal D_Y$-module structure 
on $\omega_{Y}\otimes_{\mathcal{O}_Y}\mathcal D_Y$, and  
$L_\xi$ acts on $\omega_{Y}\otimes_{\mathcal{O}_Y}\mathcal D_Y$ by the Leibniz rule. 
\end{corollary}

\subsection{}\label{resolution}
Assume $Y$ is an $H$-spherical variety, and let $Hy_0$ be the open $H$-orbit in $Y$. 
Choose a complete normal toroidal $H$-spherical variety $\tilde{Y}_0$ containing $Hy_0$
as an open $H$-orbit.
Let $\tilde{Y}_1$ be the normalization of the Zariski closure of $Hy_0$ in $Y\times \tilde{Y}_0$. Then
$\tilde{Y}_1$ is a toroidal $H$-spherical variety proper over $Y$. Refining the fan of $\tilde{Y}_1$, 
we get a smooth toroidal $H$-spherical variety $\tilde Y$ together with 
an $H$-equivariant birational proper morphism $p: \tilde{Y}\to Y$. 
Let $\tilde l: Hy_0\hookrightarrow \tilde{Y}$ be the open immersion. 
We have $R^ip_*\omega_{\tilde{Y}}=0$ for all $i\geq 1$ by the Grauert-Riemenschneider theorem. Set 
$$\omega_Y=p_*\omega_{\tilde Y}.$$ By \cite[Theorem 15.20]{Timashev}, $Y$ is Cohen-Macaulay and has rational singularity.
The sheaf $\omega_Y$ coincides with the dualizing sheaf for $Y$. Let $L_\xi$ $(\xi\in\mathfrak L(H))$ act on 
$\omega_{Y}$ via its action on $\omega_{\tilde{Y}}$.

\begin{proposition}\label{sphericalDmodisogen} 
Suppose $Y$ is a spherical variety with an affine open $H$-orbit $Hy_0$. 
Let $V$ be a smooth variety provided with an $H$-action, let 
$f:Y\to V$ be a finite $H$-equivariant morphism, let $\iota:Hy_0\to V$ be the composite of $$Hy_0\hookrightarrow Y\stackrel{f}\to V.$$ 
and let \begin{eqnarray*}
\mathcal N:=\big(f_*\omega_{Y}\otimes_{\mathcal{O}_V}\mathcal D_V \big)
\big/\sum_{\xi\in\mathfrak L(H)}L_\xi \big(f_*\omega_{Y}\otimes_{\mathcal{O}_V}\mathcal D_V \big),
\end{eqnarray*}
where the right $\mathcal{D}_V$-module structure is induced by the naive right $\mathcal D_V$-module structure 
on $\omega_{Y}\otimes_{\mathcal{O}_V}\mathcal D_V$, and  
$L_\xi$ acts on $\omega_{Y}\otimes_{\mathcal{O}_V}\mathcal D_V$ by the Leibniz rule. 	
Then $R\iota_{!}\omega_{Hy_0}$ and $\mathcal N$ are regular holonomic right $\mathcal D_V$-modules, and 
$R\iota_{!}\omega_{Hy_0}$ is a direct summand of $\mathcal N$. 
\end{proposition}

\begin{proof} We will prove $\mathcal N$ is an $H$-equivariant $\mathcal D_V$-module in Proposition \ref{Nequi} below. 
Then by \cite[II 5 Theorem]{Hotta}, $\mathcal N$ is a regular holonomic. 

We can factorize $\iota$ as the composite of the affine open immersion 
$Hy_0\hookrightarrow Y$ and the finite morphism $f:Y\to V$. By the proof of Proposition \ref{prop:perverseholonomic} (ii), 
for any holonomic $\mathcal D_{Hy_0}$-module $\mathcal F$, $R\iota_+\mathcal F$ is a holonomic $\mathcal D_V$-module. 
So $R\iota_!\mathcal F\cong \mathbb DR\iota_{+}\mathbb D\mathcal F$ is also a holonomic $\mathcal D_V$-module. As 
$R\iota_{!}\omega_{Hy_0}$ is $H$-equivariant, it is a regular holonomic $\mathcal D_V$-module.

Let $U=V-f(Y-Hy_0)$. Since $f$ is a closed map, $U$ is open in $V$. 
We claim that $Hy_0=f^{-1}(U)$. Clearly 
$f^{-1}(U)$ is a subset of $Hy_0$ and is invariant under the action of $H$. It suffices to show 
$f^{-1}(U)$ is nonempty. If this is not true, then $f(Y)=f(Y-Hy_0)$. This is impossible since 
$f$ is a finite morphism and 
$$\mathrm{dim}\, f(Y)=\mathrm{dim}\, Y>\mathrm{dim}\, (Y-Hy_0)=\mathrm{dim}\,f(Y-Hy_0).$$
Denote by $f_0: Hy_0\to U$ the morphism induced by $f$. It is a finite morphism. Let $j: U\hookrightarrow V$ be the open immersion.
We have 
a commutative diagram
$$\begin{tikzcd}
&Hy_0 \arrow[r, "f_0"] \arrow[dl, "\tilde l",swap, hook] \arrow[d,hook]\arrow[dr, "\iota"]&U\arrow[d, "j",hook]\\
\tilde{Y} \arrow[r, "p"] & Y \arrow[r,"f"] & V.
\end{tikzcd}$$ 
We have
\begin{align*}
&\big(\omega_{Hy_0}\otimes_{\mathcal{O}_{Hy_0}}\mathcal D_{Hy_0} \big)
\big/\sum_{\xi\in\mathfrak L(H)}L_\xi \big(\omega_{Hy_0}\otimes_{\mathcal{O}_{Hy_0}}\mathcal D_{Hy_0} \big)\\
\cong\,&\big(\omega_{Hy_0}\otimes_{\mathcal{O}_{Hy_0}}\mathcal D_{Hy_0} \big)
\big/\sum_{\xi\in\mathfrak L(H)}L'_\xi \big(\omega_{Hy_0}\otimes_{\mathcal{O}_{Hy_0}}\mathcal D_{Hy_0} \big)
\cong \omega_{Hy_0}\nonumber
\end{align*}
as right $\mathcal D_{Hy_0}$-modules, where in $\big(\omega_{Hy_0}\otimes_{\mathcal{O}_{Hy_0}}\mathcal D_{Hy_0} \big)
\big/\sum_{\xi\in\mathfrak L(H)}L'_\xi \big(\omega_{Hy_0}\otimes_{\mathcal{O}_{Hy_0}}\mathcal D_{Hy_0} \big)$, the 
right $\mathcal D$-module structure is given by the Leibniz rule and $L'_\xi$ acts via the naive right $\mathcal D$-module 
structure. For the naive right $\mathcal D$-module structure, we have 
\begin{align*}
Rf_{0+}(\omega_{Hy_0}\otimes_{\mathcal{O}_{Hy_0}}\mathcal D_{Hy_0})\cong\;&
Rf_{0*} \Big((\omega_{Hy_0}\otimes_{\mathcal{O}_{Hy_0}}\mathcal D_{Hy_0}) \otimes^L_{\mathcal D_{Hy_0}}
\mathcal{D}_{Hy_0\to U}\Big)\\
\cong\;&f_{0*} \Big((\omega_{Hy_0}\otimes_{\mathcal{O}_{Hy_0}}\mathcal D_{Hy_0}) \otimes_{\mathcal D_{Hy_0}}
(\mathcal{O}_{Hy_0}\otimes_{f_0^{-1} \mathcal O_U}f_0^{-1} \mathcal D_U)\Big)\\
\cong\;&f_{0*} (\omega_{Hy_0}\otimes_{f_0^{-1} \mathcal O_U}f_0^{-1} \mathcal D_U)\cong 
f_{0*}\omega_{Hy_0}\otimes_{\mathcal O_U}\mathcal D_U, 
\end{align*} where for the second isomorphism, we use the fact $Rf_{0*}=f_{0*}$ since 
$f_0$ is a finite morphism, and the fact that $\omega_{Hy_0}$ is a free $\mathcal O_{Hy_0}$-module so that we can replace 
$\otimes^L$ by $\otimes$. For a general quasi-coherent 
$\mathcal D_{Hy_0}$-module $\mathcal K$, since $f_0$ is a finite morphism, we have 
\begin{align}\label{Rf_+}
R^0f_{0+}(\mathcal K)\cong\;&
\mathcal H^0(Rf_{0*} (\mathcal K \otimes^L_{\mathcal D_{Hy_0}}
\mathcal D_{Hy_0\to U}))
\cong f_{0*}\mathcal H^0 (\mathcal K \otimes^L_{\mathcal D_{Hy_0}}
\mathcal D_{Hy_0\to U})\\
\cong\;&f_{0*}(\mathcal K \otimes_{\mathcal D_{Hy_0}}
\mathcal D_{Hy_0\to U}). \nonumber
\end{align} 
So $R^0f_{0,+}$ is a right exact functor on the category of $\mathcal D_{Hy_0}$-modules. We thus have
\begin{align*}
R^0f_{0+} \omega_{Hy_0}\cong\;& R^0f_{0+} \Big(\big(\omega_{Hy_0}
\otimes_{\mathcal{O}_{Hy_0}}\mathcal D_{Hy_0} \big)\big/\sum_{\xi\in\mathfrak L(H)}L_\xi 
\big(\omega_{Hy_0}\otimes_{\mathcal{O}_{Hy_0}}\mathcal D_{Hy_0} \big)\Big)\\
\cong\;& \big(f_{0*}\omega_{Hy_0}\otimes_{\mathcal{O}_U}\mathcal D_U \big)\big/\sum_{\xi\in\mathfrak L(H)}L_\xi 
\big(f_{0*}\omega_{Hy_0}\otimes_{\mathcal{O}_U}\mathcal D_U \big).
\end{align*}
For the holonomic $\mathcal D_{Hy_0}$-modules $\omega_{Hy_0}$, $Rf_{0+} \omega_{Hy_0}$ is a holonomic 
$\mathcal D_U$-module, and we have 
$$Rf_{0!} \omega_{Hy_0}\cong Rf_{0+} \omega_{Hy_0}\cong R^0f_{0+} \omega_{Hy_0}\cong
 \big(f_{0*}\omega_{Hy_0}\otimes_{\mathcal{O}_U}\mathcal D_U \big)\big/\sum_{\xi\in\mathfrak L(H)}L_\xi 
\big(f_{0*}\omega_{Hy_0}\otimes_{\mathcal{O}_U}\mathcal D_U \big)\cong Rj^!\mathcal N.$$ 
Denote the composite of these isomorphisms by $\phi': Rf_{0!} \omega_{Hy_0}\stackrel\cong \to Rj^!\mathcal N$. 
By adjoint it corresponds to a morphism $\phi: R\iota_!\omega_{Hy_0}\to \mathcal N$. Let's construct a left inverse for $\phi$.

For the naive right $\mathcal D$-module structure, we have 
\begin{align*}
R(fp)_+(\omega_{\tilde Y}\otimes_{\mathcal{O}_{\tilde Y}}\mathcal D_{\tilde Y})&\cong
f_*Rp_* \Big((\omega_{\tilde Y}\otimes_{\mathcal{O}_{\tilde Y}}\mathcal D_{\tilde Y}) \otimes^L_{\mathcal D_{\tilde Y}}
(\mathcal{O}_{\tilde Y}\otimes_{p^{-1}f^{-1} \mathcal O_V}p^{-1}f^{-1} \mathcal D_V)\Big)\\
&\cong
f_*Rp_* (\omega_{\tilde Y}\otimes_{p^{-1}f^{-1} \mathcal O_V}p^{-1}f^{-1} \mathcal D_V)
\cong f_*\omega_{Y}\otimes_{\mathcal O_V}\mathcal D_V.
\end{align*}
By Corollary \ref{rightD-modversion},  
we have an exact sequence
\[\bigoplus_{\xi\in\mathfrak L(H)}\omega_{\tilde{Y}}\otimes_{\mathcal{O}_{\tilde{Y}}}\mathcal D_{\tilde{Y}} \xrightarrow{\sum L_\xi} 
\omega_{\tilde{Y}}\otimes_{\mathcal{O}_{\tilde{Y}}}\mathcal D_{\tilde{Y}} \to  R\tilde l_!\omega_{Hy_0}\to 0.\] 
Applying $R(fp)_{+}$ to this sequence, we get morphisms 
\[\bigoplus_{\xi\in\mathfrak L(H)}\omega_{Y}\otimes_{\mathcal{O}_V}\mathcal D_V 
\xrightarrow{\sum L_\xi} \omega_{Y}\otimes_{\mathcal{O}_V}\mathcal D_V \to  R\iota_!\omega_{Hy_0}\] whose composite is $0$. 
We thus get a morphism 
$\psi:\mathcal N\to R\iota_!\omega_{Hy_0}$. The restriction 
$\psi': Rj^!\mathcal N \to Rf_{0!}\omega_{Hy_0}$ of $\psi$ to $U$ coincides with the inverse of $\phi'$ constructed above.
We have a commutative diagram
$$\begin{tikzcd}
R\iota_!\omega_{Hy_0}\arrow[d,"\phi"]\arrow[r,"\cong"]&Rj_!Rf_{0!}\omega_{Hy_0}\arrow[d,"Rj_!(\phi')"]&\\
\mathcal N\arrow[d,"\psi"]&Rj_!Rj^!\mathcal N\arrow[l, "\mathrm{adj}_{\mathcal N}",swap]
\arrow[d,"Rj_!Rj^!(\psi)"]\arrow[r,"Rj_!(\psi')(=Rj_!(\phi'^{-1}))"]&Rj_!Rf_{0!}\omega_{Hy_0}
\arrow[d,"Rj_!(\mathrm{adj}_{Rf_{0!}\omega_{Hy_0}})","\cong"']\\
R\iota_!\omega_{Hy_0}&(Rj_!Rj^!)R\iota_!\omega_{Hy_0}\arrow[l, "\mathrm{adj}_{R\iota_!\omega_{Hy_0}}",swap]
&Rj_!(Rj^!Rj_!Rf_{0!})\omega_{Hy_0}\arrow[d,Rightarrow,no head]\\
&Rj_! Rf_{0!}\omega_{Hy_0}\arrow[ul,"\cong"]\quad &
\quad (Rj_!Rj^!)Rj_!Rf_{0!}\omega_{Hy_0}\arrow[ul,"\cong",swap]\arrow[l,"\mathrm{adj}_{Rj_! Rf_{0!} \omega_{Hy_0}}",swap].
\end{tikzcd}$$ 
The morphism 
$\psi$ is a left inverse of the morphism 
$\phi: R\iota_!\omega_{Hy_0}\to \mathcal N$ by the commutativity of the
above diagram and the fact that for $\mathcal F= Rf_{0!}\omega_{Hy_0}$, the composite
$$Rj_! \mathcal F\xrightarrow {Rj_!(\mathrm{adj}_{\mathcal F})} Rj_!(Rj^!Rj_!\mathcal F)
=(Rj_!Rj^!)Rj_!\mathcal F\xrightarrow{\mathrm{adj}_{Rj_!\mathcal F}}Rj_!\mathcal F$$
is identity.
\end{proof}

\subsection{Equivariant $\mathcal D$-modules} Suppose $X$ is a ${\overline K}$-variety endowed with an action 
$$A:H\times X\to X.$$ 
Let $p_2: H\times X\to X$ be the projection. A \emph{weakly $H$-equivariant left $\mathcal{D}_X$-module} is a pair 
$(\mathcal K, \phi)$ consisting of a 
left $\mathcal{D}_X$-module $\mathcal K$ and an $(\mathcal{O}_H\boxtimes \mathcal{D}_X)$-module isomorphism
\[\phi: A^*\mathcal K\to p_2^*\mathcal K\]
satisfying the cocycle condition in \ref{equiv}.
A weakly $H$-equivariant $\mathcal{D}_X$-module $(\mathcal K, \phi)$ is called \emph{strongly $H$-equivariant} or just
\emph{$H$-equivariant} for simplicity if $\phi$ is a 
$(\mathcal{D}_H\boxtimes \mathcal{D}_X)$-module isomorphism. 
For any $\xi\in\mathfrak L(H)$, denote by $L_\xi^H$ and $L_\xi$ the vector fields 
$$L_\xi^H(h)=\frac{\mathrm d}{\mathrm dt}\Big|_{t=0}(\exp(t\xi)h), \quad 
L_\xi(x)=\frac{\mathrm d}{\mathrm dt}\Big|_{t=0}(\exp(t\xi)x)$$ on $H$ and $X$, respectively.  
A weakly $H$-equivariant left $\mathcal{D}_X$-module $\mathcal K$  
is $H$-equivariant if and only if for any section $s$ of $\mathcal K$ and 
any $\xi\in \mathfrak L(H)$, we have
$$(L^H_\xi\otimes1)\phi(A^*(s))=\phi(A^*L_{\xi}(s)).$$ 
On the other hand, we have
\begin{align*}
(L^H_\xi\otimes1)\phi(A^*(s))|_{\{h\}\times X}&=\lim_{t\to 0}\frac{h^{-1}\exp(-t\xi)(s)-h^{-1}(s)}{t}
=h^{-1}\Big(\lim_{t\to 0}\frac{\exp(-t\xi)(s)-(s)}{t}\Big),\\
\phi(A^*L_{\xi}(s))|_{\{h\}\times X}&=h^{-1}L_\xi(s).
\end{align*}
So a weakly $H$-equivariant left $\mathcal{D}_X$-module $\mathcal K$  
is $H$-equivariant if and only if for any section $s$ of $\mathcal K$ and 
any $\xi\in \mathfrak L(H)$, we have
\begin{align}\label{lieder}
\lim_{t\to 0}\frac{\exp(-t\xi)(s)-(s)}{t}=L_\xi(s).
\end{align}
Similarly we can define weakly $H$-equivariant and $H$-equivariant right $\mathcal D_X$-modules. 

\begin{proposition}\label{Nequi} 
\begin{enumerate}[(i)]
\item Let $X$ be a variety with an $H$-action. Then 
$\mathcal D_X/\sum_{\xi\in\mathfrak L(H)} \mathcal D_X L_\xi$ is an $H$-equivariant left $\mathcal D_X$-module.
\item The $\mathcal D_V$-module $\mathcal N$ in Proposition \ref{sphericalDmodisogen} 
is an $H$-equivariant right $\mathcal D_V$-module. 
\end{enumerate}
\end{proposition}

\begin{proof} (i) The left $\mathcal D_X$-module $\mathcal D_X$ is equipped with a natural weakly $H$-equivariant 
$\mathcal{D}_X$-module structure $\phi: A^*\mathcal D_X\to p_2^*\mathcal D_X$.  
For any $(h,x)\in H\times X$, $\phi^{-1}_{(h,x)}$ maps a tangent vector 
$\theta\in \mathcal T_{X, x}\subset \mathcal D_{X,x}$ 
to $A(h,\hbox{-})_*(\theta) \in \mathcal D_{X,hx}$. The left ideal $\sum_{\xi\in\mathfrak L(H)}\mathcal D_{X}L_\xi$
of $\mathcal D_{X}$ is invariant under the action of $H$. So 
$\mathcal D_{X}/\sum_{\xi\in\mathfrak L(H)}\mathcal D_{X}L_\xi$ is a weakly $H$-equivariant 
$\mathcal{D}_{X}$-module. Let $\theta_1,\ldots, \theta_n$ be tangent vector fields 
on $X$. Note that $\lim_{t\to 0}\frac{\exp(-t\xi)(\theta_i)-\theta_i}{t}$ is just the Lie derivative $[L_\xi, \theta_i]$ of 
$L_\xi$ on $\theta_i$. For the sections $\theta_1\cdots\theta_n$ of 
$\mathcal D_X$, we have 
\begin{eqnarray*}
&&\lim_{t\to 0}\frac{\exp(-t\xi)(\theta_1\cdots\theta_n)-\theta_1\cdots\theta_n}{t}
=\sum_{i=1}^n \theta_1\cdots \theta_{i-1} [L_\xi, \theta_i]\theta_i\cdots\theta_n\\
&=& L_\xi\theta_1\cdots\theta_n- \theta_1\cdots\theta_n L_\xi 
\equiv L_\xi\theta_1\cdots\theta_n \mod 
\sum_{\xi\in\mathfrak L(H)}\mathcal D_XL_\xi
\end{eqnarray*}
So the condition (\ref{lieder}) holds for all sections of 
$\mathcal D_X/\sum_{\xi\in\mathfrak L(H)}\mathcal D_XL_\xi$. Hence $\mathcal D_X/\sum_{\xi\in\mathfrak L(H)}\mathcal D_XL_\xi$
is strongly $H$-equivariant. 

(ii) By (i), $\mathcal D_{\tilde{Y}}$ is weakly $H$-equivariant and 
$\mathcal D_{\tilde{Y}}/\sum_{\xi\in\mathfrak L(H)}\mathcal D_{\tilde{Y}}L_\xi$
is strongly $H$-equivariant.
So the right $\mathcal D_Y$-module 
$\omega_{\tilde Y}\otimes_{\mathcal O_{\tilde Y}}\mathcal D_{\tilde Y}$ is weakly $H$-equivariant, and 
$(\omega_{\tilde Y}\otimes_{\mathcal O_{\tilde Y}}\mathcal D_{\tilde Y})/
\sum_{\xi\in\mathfrak L(H)} L'_{\xi}(\omega_{\tilde Y}\otimes_{\mathcal O_{\tilde Y}}\mathcal D_{\tilde Y})$
is strongly $H$-equivariant, where the right $\mathcal D$-module structure on 
$\omega_{\tilde Y}\otimes_{\mathcal O_{\tilde Y}}\mathcal D_{\tilde Y}$ is given by the Leibniz rule, and 
$L'_\xi$ acts via the naive right $\mathcal D$-module structure. We switch the right $\mathcal D$-module structure on
$\omega_{\tilde Y}\otimes_{\mathcal O_{\tilde Y}}\mathcal D_{\tilde Y}$ so that the right $\mathcal D$-module structure is the naive one, 
and $L_\xi$ acts via the Leibniz rule. Then $\omega_{\tilde Y}\otimes_{\mathcal O_{\tilde Y}}\mathcal D_{\tilde Y}$ is a weakly 
$H$-equivariant right $\mathcal D$-module and the condition (\ref{lieder}) holds on the cokernel of the morphism
\[\bigoplus_{\xi\in\mathfrak L(H)}\omega_{\tilde{Y}}\otimes_{\mathcal{O}_{\tilde{Y}}}\mathcal D_{\tilde{Y}} \xrightarrow{\sum L_\xi} 
\omega_{\tilde{Y}}\otimes_{\mathcal{O}_{\tilde{Y}}}\mathcal D_{\tilde{Y}}.\] 
Applying $R(fp)_{+}$, we see $f_*\omega_{Y}\otimes_{\mathcal{O}_V}\mathcal D_V$ is a weakly 
$H$-equivariant right $\mathcal D$-module and the condition (\ref{lieder}) holds on the cokernel of the morphism
	\[\bigoplus_{\xi\in\mathfrak L(H)} f_*\omega_{Y}\otimes_{\mathcal{O}_V}\mathcal D_V
	 \xrightarrow{\sum L_\xi} f_*\omega_{Y}\otimes_{\mathcal{O}_V}\mathcal D_V.\]
So $\mathcal{N}$ is $H$-equivariant.
\end{proof}

\section{Hypergeometric $\mathcal D$-modules}

\subsection{The modified hypergeometric $\mathcal D$-module}
We first works over an algebraically closed field $\overline K$ of characteristic $0$. Let $G$ be 
a reductive algebraic group over $\overline K$, let 
$\rho_j: G\to \mathrm{GL}(V_j)$ $(j=1, \ldots, N)$ be irreducible representations, 
and let $\mathbb V=\prod_{j=1}^N\mathrm{End}(V)$. 
Suppose the morphism 
$$\iota: G\to \mathbb V
=\prod_{j=1}^N\mathrm{End}(V),\quad g\mapsto (\rho_1(g), \ldots,\rho_N(g))$$ 
is quasi-finite. 
Let $X$ be the scheme theoretic image of $\iota$, let $Y$ be the integral closure of 
$X$ in $G$, and let $f: Y\to \mathbb V$ be the composite $Y\to X\to \mathbb V$. 
Set $H=G\times G$. 
The two sided action of $G$ on 
$\mathbb V$ induces an action of $H$ on $Y$. Fix a Borel subgroup $B^+$ of $G$. Let $B^-$ be opposite 
Borel subgroup. 
Then $B:=B^+\times B^-$ is a Borel subgroup of $H$. 
The $B$-orbit  $B^+B^-$ in $G\subset Y$ is the open stratum of the Bruhat decomposition of 
$G$. Thus $Y$ is a spherical variety for the reductive group $H$. 
It is a normal algebraic monoid containing the reductive
group $G$. 

Consider the hypergeometric $\mathcal D$-module 
$$\mathcal Hyp_!=\mathfrak F_{\mathcal L}(R\iota_!\omega_G)$$
introduced in \ref{definitionhypD}. 
By Proposition \ref{sphericalDmodisogen}, 
$R\iota_!\omega_{\mathbb G}$ is a direct summand of 
$$\mathcal N=(f_*\omega_Y\otimes_{\mathcal{O}_{\mathbb V}}\mathcal D_{\mathbb V}
)\big/\sum_{\xi\in\mathfrak L(H)}L_\xi (\omega_{Y}\otimes_{\mathcal{O}_{\mathbb V}}
\mathcal D_{\mathbb V}),$$ 
where the right $\mathcal D_{\mathbb V}$-module structure is induced by the naive 
right $\mathcal D_{\mathbb V}$-module on $\omega_{Y}\otimes_{\mathcal{O}_{\mathbb V}}
\mathcal D_{\mathbb V}$, and $L_\xi$ acts on $\omega_{Y}\otimes_{\mathcal{O}_{\mathbb V}}
\mathcal D_{\mathbb V}$ by the Leibniz rule. 
We define the \emph{modified hypergeometric $\mathcal D$-module}  $\mathcal M$ to be the Fourier transform of $\mathcal N$:
$$\mathcal M:=\mathfrak F_{\mathcal L}(\mathcal N).$$ Then $\mathcal Hyp_!$ is a direct summand of $\mathcal M$.
In this section, we prove $\mathcal M$ is an integrable connection on $\mathbb V^{\mathrm{gen}}$ and estimate its rank. 

\begin{remark} In \cite[5.2]{K}, Kapranov defines the $A$-hypergeometric $\mathcal D$-module to be the Fourier transform
of the left $\mathcal D_{\mathbb V}$-module 
$$\mathcal D_{\mathbb V}\Big/\Big(\sum_{\xi\in\mathfrak L(H)}\mathcal D_{\mathbb V}L_\xi +\mathcal D_{\mathbb V}\mathcal I_X\Big),$$
where $\mathcal I_X$ is the ideal sheaf of the closed immersion $X\to \mathbb V$. We don't know the relationship between the 
$A$-hypergeometric $\mathcal D$-module and the modified hypergeometric $\mathcal D$-module.
\end{remark}

\subsection{Homogenization} 
Extend the representations $\rho_j: G\to \mathrm{GL}(V_j)$ $(j=1,\ldots, N)$
to representations $$\rho'_j:\mathbb G_m\times G\to \mathrm{GL}(V_j), \quad (t, g)\mapsto t\rho_j(g).$$ Let $\rho'_0$ be the
representation $$\rho'_0:\mathbb G_m\times G\mapsto \mathrm{GL}(\overline K), \quad (t, g)\mapsto t.$$ 
The representations $\rho'_0,\rho'_1,\ldots, \rho'_N$ of $\mathbb G_m\times G$ satisfy the homogeneity condition.
Let
$$\mathbb V'=\mathbb A^1\times \mathbb V=\mathbb A^1\times \prod_{j=1}^N \mathrm{End}(V_j),$$
let $\iota'$ be the morphism
$$\iota': \mathbb G_m\times G\to \mathbb V',\quad (t, g)\mapsto (\rho'_0(t,g), \rho'_1(t,g), \ldots, \mathbb \rho'_N(t,g)),$$
and let $X$ (resp. $X'$) be the Zariski closure of the image of $\iota$ (resp. $\iota'$). We have 
$$X\cong X'\cap (\{1\}\times \mathbb V).$$
Embed $\mathbb V$ into $\mathbb P(\mathbb V')$ by 
$$\mathbb V\hookrightarrow \mathbb P(\mathbb V'), \quad v\mapsto [1:v].$$ Let 
$\overline X$ be the closure of $X$ in $\mathbb P(\mathbb V')$. We have $X=\mathbb V\cap \overline X$.
Let $$\overline X_\infty=\overline X\cap \{[0:y]\in\mathbb P(\mathbb V'): y\in \mathbb V-\{0\}\},$$
and let $C(\overline X), C(\overline X_\infty)\subset \mathbb V'$ be the cones over $\overline X$ and $\overline X_\infty$,
respectively. We have 
$$\mathrm{im}\,\iota'=\bigcup_{t\in\overline K} t(1\times\mathrm{im}\,\iota),\quad 
X'=C(\overline X)=\bigcup_{t\in\overline K} t(1\times X)\sqcup C(\overline X_\infty).$$
Since $C(\overline X_\infty)\subset 0\times\mathbb V$, we often identify $C(\overline X_\infty)$ with 
a cone in $\mathbb V$.

By assumption, $\iota:G\to \mathbb V$ is quasi-finite. So $\iota':\mathbb G_m\times G\to \mathbb V'$ 
is quasi-finite. Let $Y'$ be the integral closure of $X'$ in $\mathbb G_m\times G$. It is a spherical variety
for the group $H':=(\mathbb G_m\times G)\times (\mathbb G_m\times G)$. Let $p':\tilde Y'\to Y'$ be a proper 
equivariant morphism so that $\tilde Y'$ is a smooth toroidal spherical $H'$-variety. Let $f':Y'\to \mathbb V'$ be the composite 
$Y'\to X'\to \mathbb V'$. We have a commutative diagram
$$\begin{tikzcd}
&\mathbb G_m\times G\arrow[dl,hook]\arrow[d,hook]\arrow[drr,"\iota'"]&&\\
\tilde Y'\arrow[r,"p'"]&Y'\arrow[r]\arrow[rr,bend right,"f'",swap]&X'\arrow[r]&\mathbb V'.
\end{tikzcd}$$

\begin{lemma}\label{omega} Let $i_0, i_1$ be the closed immersion
\begin{eqnarray*}
i_0: \mathbb V\to \mathbb V'=\mathbb A^1\times\mathbb V, && x\mapsto (0,x),\\
i_1: \mathbb V\to \mathbb V'=\mathbb A^1\times\mathbb V, && x\mapsto (1,x).
\end{eqnarray*}
We have $i_1^*f'_*\omega_{Y'}\cong f_*\omega_Y$, and $i_0^*f'_*\omega_{Y'}$ is supported on 
$C(\overline X_\infty)$, where we denote the inverse image by $i_0$ of $C(\overline X_\infty)\subset 
0\times\mathbb V$ also by $C(\overline X_\infty)$. 
\end{lemma}

\begin{proof}
Taking the base change of the above diagram with respect to $i_1$,
we get a commutative diagram
$$\begin{tikzcd}
&G\arrow[dl,hook]\arrow[d,hook]\arrow[drr,"\iota'"]&&\\
\tilde Y\arrow[r,"p"]&Y \arrow[r]\arrow[rr,bend right,"f",swap]&X\arrow[r]&\mathbb V,
\end{tikzcd}$$
where by abuse of notation we denote the base changes of 
$Y'$, $\tilde Y'$ and $f'$ by $Y$, $\tilde Y$ and $f$, respectively.
Then $\tilde Y\to Y$ is proper and $Y\to X$ is finite.
We have a commutative diagram
$$\begin{tikzcd}
\mathbb G_m\times \tilde Y\arrow[r,"\mathrm{id}\times p"]\arrow[d,hook]&\mathbb G_m\times Y \arrow[r]\arrow[d,hook]
\arrow[rr,bend left,"\mathrm{id}\times f"]&
\mathbb G_m\times X\arrow[r]\arrow[d,hook]&\mathbb G_m\times\mathbb V\arrow[d,hook,"A"]\\
\tilde Y'\arrow[r,"p'"]&Y'\arrow[r]\arrow[rr,bend right,"f'",swap]&X'\arrow[r]&\mathbb V',
\end{tikzcd}$$
where the vertical arrows are open immersions induced by the
$\mathbb G_m$-action on $\mathbb V'$:
$$A(t, y)=t(1, y)=(t, ty).$$
Thus $\mathbb G_m\times \tilde Y$ is smooth and $\mathbb G_m\times Y$ is normal. It follows that
$\tilde Y$ is smooth and $Y$ is normal. So $Y$ is the integral closure of $X$ in $G$,
$\tilde Y\to Y$ is proper birational equivariant, and $\tilde Y$ is a smooth toroidal spherical $H$-variety. 
 Moreover, we have isomorphisms
 $$A^* f'_*\omega_{Y'}\cong A^*f'_*p'_*\omega_{\tilde Y'}\cong \omega_{\mathbb G_m}\boxtimes f_*p_*\omega_{\tilde Y}
 \cong \omega_{\mathbb G_m}\boxtimes f_*\omega_Y.$$
It follows that $i_1^*f'_*\omega_{Y'}\cong f_*\omega_Y$. The support of $i_0^*f'_*\omega_{Y'}$ is 
\[X'\cap (\{0\}\times \mathbb V)=C(\overline X)\cap (\{0\}\times \mathbb V)=C(\overline X_\infty).\qedhere\]
\end{proof}

Let $T$ be the maximal torus of $B$, $\Lambda=\mathrm{Hom}(T, \mathbb G_m)$ 
the weight lattice, $\lambda_j\in\Lambda$ $(j=1, \ldots, N)$  the maximal weights of $\rho_j$, and $\Delta_\infty$ the
New polytope at $\infty$, that is, the convex hull in $\Lambda_{\mathbb R}=\Lambda\otimes_{\mathbb Z}\mathbb R$
of $W(0, \lambda_1,\ldots, \lambda_n)$, where $W$ is the Weyl group. For the group $G'=\mathbb G_m\times G$, 
the weight lattice is $\Lambda'=\mathbb Z\oplus \Lambda$, and the highest weights 
of $\rho'_j$ (resp. $\rho'_0$) are $(1, \lambda_j)$ (resp. $(1,0)$). Let $\Delta'$ be the Newton polytope of $\rho'_j$ $(j=0,1\ldots, N)$, 
that is, the convex hull of the orbits of $(1,0)$ and $(1, \lambda_j)$ $(j=1,\ldots, N)$ under the Weyl group. We have 
$$\Delta':=1\times\Delta_\infty.$$ 
The Newton polytope $\Delta'_\infty$ at $\infty$ of $\rho'_j$ $(j=0,1\ldots, N)$ is 
$$\Delta'_\infty=\bigcup_{t\in [0,1]} t\Delta'=\bigcup_{t\in [0,1]} t(1\times\Delta_\infty).$$
A face of $\Delta'$ is of the form $1\times \tau$, 
where $\tau$ is a face of $\Delta_\infty$. Define $e(\tau)=(e_j(\tau))\in \mathbb V=\prod_{j=1}^N\mathrm{End}(V_j)$ 
by the formula (\ref{etau}). 
For each $j\in \{1, \ldots, N\}$, 
regarding $V_j$ as a representation of $\mathbb G_m\times G$, $\mathbb G_m\times T$ acts on $V_j(\lambda)$ via the 
character $(1,\lambda)$. We define
$$e_j(1\times \tau)|_{V_j(\lambda)}=
\begin{cases} 
\mathrm{id}_{V_j(\lambda)}&\hbox{if } \lambda\in \tau,\hbox{ or equivalently, } (1,\lambda)\in 1\times \tau\\
0&\hbox{otherwise}.
\end{cases}$$
The representation $\rho_0$ acts on $\overline K$ via the character $(1,0)$. We define 
$$e_0(1\times \tau)|_{\overline K}=
\begin{cases} 
1&\hbox{if } 0\in \tau,\\
0&\hbox{otherwise}.
\end{cases}$$
Let $e(1\times\tau)=(e_0(1\times \tau), \ldots, e_N(1\times \tau))\in \mathbb V'=\prod_{j=0}^N\mathrm{End}(V_j)$. 
We have actions
\begin{align*}
(G\times G)\times \mathbb V\to\mathbb V, &\quad ((g, h),(A_j)_{j=1,\ldots, N})\mapsto (\rho_j(g)A_j\rho_j(h^{-1}))_{j=1,\ldots, N},\\
(\mathbb G_m\times G)\times (\mathbb G_m\times G)\times \mathbb V'\to\mathbb V',&\quad ((t, g), (s,h),(A_j)_{j=0,1,\ldots, N})\mapsto  
(\rho'_j(t, g)A_j\rho'_j((s,h)^{-1}))_{j=0,1,\ldots, N}.
\end{align*}
They induce actions of $H$ on $X$ and on $\overline X$, and an action of
$H'$ on $X'=C(\overline X)$.

\begin{lemma}\label{lm:Kapranovorbit} Notation as above.
\begin{enumerate}[(i)]
\item The map 
$$1\times \tau\mapsto (\mathbb G_m\times G) e_{1\times \tau}  (\mathbb G_m\times G) $$
defines a one-to-one correspondence between 
the $W$-orbits of faces of $\Delta'=1\times\Delta_\infty$ and orbits of $(\mathbb G_m\times G) \times 
(\mathbb G_m\times G) $ on $X'=C(\overline X)$.
\item The map 
$$\tau\mapsto Ge_\tau G$$
defines a one-to-one correspondence between 
the $W$-orbits of faces of $\Delta_\infty$ containing the origin and the orbits of $G\times G$ on $X$. 
\item The map 
$$\tau\mapsto (\mathbb G_m\times G) e_{1\times \tau}  (\mathbb G_m\times G)$$
defines a one-to-one correspondence between 
the $W$-orbits of faces of $\Delta_\infty$ not containing the origin and the orbits of $(\mathbb G_m\times G) \times 
(\mathbb G_m\times G) $ on $C(\overline X_\infty)$. 
\end{enumerate}
\end{lemma}

\begin{proof} Since $\rho'_0,\rho'_1,\ldots, \rho'_N$ satisfy the homogeneity condition, 
(i) follows from Kapranov's work \cite[Theorem 2.4.2]{K}.  
We have 
$$C(\overline X)=\bigcup_{t\in\overline K} t(1\times X)\sqcup C(\overline X_\infty).$$
If $0\in\tau$, then $e_{1\times\tau}\in 1\times X$. 
If $0\not\in \tau$, then $e_{1\times \tau}\in C(\overline X_\infty)$.
Moreover, we have a one-to-one correspondence
$$O\mapsto \bigcup_{t\in\overline K} t(1\times O)$$ between $(G\times G)$-orbits on $X$ and
$(\mathbb G_m\times G)\times (\mathbb G_m\times G)$-orbits on $\bigcup_{t\in\overline K} t(1\times X)$.
So (ii)-(iii) follow from (i).
\end{proof}

Choose a linear coordinate $(x_1, \ldots, x_m)$ on the vector space $\mathbb V$. 
The affine coordinate ring of the scheme $\mathbb V$ 
(resp. $\mathbb V'$) is the polynomial ring $$A:=\overline K[x_1, \ldots, x_m]\quad(\hbox{resp. } 
C:=\overline K [x_0, x_1, \ldots, x_m]).$$  
Let $A_d$ be the subspace of $A$ consisting of polynomials of degree $\leq d$,
and let $$C(A)=\bigoplus_{d=0}^\infty A_d.$$ The property $A_d A_e\subset A_{d+e}$ implies that 
$C(A)$ is a graded $\overline K $-algebra. We have an isomorphism 
$$C(A)\stackrel\cong \to C$$
of graded $\overline K$-algebras mapping a polynomial $f(x_1,\ldots, x_n)\in A_d \subset C(A)$
to $x_0^d f\big(\frac{x_1}{x_0},\ldots, \frac{x_m}{x_0}\big)\in C$. 
For any $A$-module $M$ provided with a filtration 
$$\cdots\subset M_n\subset M_{n+1}\subset \cdots$$ such that
\begin{eqnarray}\label{eqn:filtration}
M=\bigcup_n M_n, \quad A_nM_m\subset M_{m+n}.
\end{eqnarray}
Let
$$C(M)=\bigoplus_{n=-\infty}^{\infty} M_n.$$
Then $C(M)$ is a graded $C$-module, and multiplication by $x_0$ on $C(M)$ is injective. 

\begin{lemma}\label{filmod} 
The functor $M\mapsto C(M)$ is an 
equivalence from the category of filtered $A$-modules satisfying (\ref{eqn:filtration}) to the category 
of graded $C$-modules on which multiplication by $x_0$ is injective. We have 
$$M\cong C(M)/(x_0-1)C(M),\quad \mathrm{Gr}(M)\cong C(M)/x_0C(M).$$
If $C(M)$ is finitely generated over $C$, then
$\mathrm{Gr}(M)$ is finitely generated over $\mathrm{Gr}(A)\cong A$.
\end{lemma}

\begin{proof}
Let $N=\bigoplus_n N_n$ be a graded $C$-module such that multiplication by $x_0$ is injective. Multiplication by 
$x_0$ makes $(N_n)_{n\in\mathbb Z}$ a direct system. 
Let $M=\varinjlim_{n}N_n$, and let $M_n$ be the image of $N_n$ in $M$. 
For any polynomial $f\in A$ and any $s\in M$, choose an integer $d$ so that $f\in A_d$, and choose an some element $s_n\in N_n$
whose image in $M$ in $s$. We define $fs$ to be the image of 
$$\Big(x_0^d f\Big(\frac{x_1}{x_0},\ldots, \frac{x_m}{x_0}\Big)\Big) s_n \in N_{n+d}$$
in $M$. This image is independent of the choices of the integer $d$ and the element $s_n$. So we get a filtered $A$-module
$M$. The maps
\begin{eqnarray*}
\bigoplus_{n=-\infty}^{\infty} M_n\to \bigoplus_{n=-\infty}^{\infty} M_n/M_{n-1}, && (s_n)\mapsto (s_n + M_{n-1}),\\
\bigoplus_{n=-\infty}^{\infty} M_n\to M, && (s_n)\mapsto \sum_n s_n
\end{eqnarray*}
induce isomorphisms
$$C(M)/x_0C(M)\cong \mathrm{Gr}(M),\quad
C(M)/(x_0-1)C(M)\stackrel\cong\to M.$$
Indeed, if $\sum_n s_n=0$, then $(s_n)=(1-x_0)(t_n)$, where $(t_n)\in \bigoplus_{n=-\infty}^{\infty} M_n$ is defined by 
$t_n=\sum_{k=-\infty}^n s_k$.
The last assertion follows from the facts that 
\[\mathrm{Gr}(A)\cong C/x_0C\cong A,\quad \mathrm{Gr}(M)\cong C(M)/x_0C(M).\qedhere\]
\end{proof}

\begin{corollary}\label{cor:fom} Let $\Omega=\Gamma(\mathbb V, f_*\omega_Y)$. 
We have an $H$-stable filtration 
$$\cdots\subset \Omega_n\subset  \Omega_{n+1}\subset\cdots$$
such that $$\bigcup_n \Omega_n=\Omega, \quad A_n\Omega_m\subset \Omega_{m+n},$$
and $\mathrm{Gr}(\Omega)$ is finitely generated over $\mathrm{Gr}(A)$. Let 
$\mathrm{Gr}(\Omega)^\sim$ be the sheaf on $\mathrm{Spec}\,\mathrm{Gr}(A)$ corresponding 
to the $\mathrm{Gr}(A)$-module $\mathrm{Gr}(\Omega)$. Via the isomorphism 
$\mathrm{Spec}\,\mathrm{Gr}(A)\cong \mathrm{Spec}\,C/x_0C$,
we have 
$$\mathrm{Gr}(\Omega)^\sim\cong i_0^* (f'_*\omega_{Y'}),$$ and 
$\mathrm{Gr}(\Omega)^\sim$ is supported on $C(\overline X_\infty)$.
\end{corollary}

\begin{proof} Let $\Omega'=\Gamma(\mathbb V',f'_*\omega_{Y'})$. 
By Lemma \ref{omega}, we have $i_1^*f'_*\omega_{Y'}\cong f_*\omega_Y$. 
So $\Omega'/(x_0-1)\Omega'\cong \Omega$. 
The $\mathbb G_m$-action
on $f'_*\omega_{Y'}$ put a graded $C$-module structure on $\Omega'$ so that 
$\Omega'_d$ is the subspace on which $\mathbb G_m$ acts via the character $c\mapsto c^d$.
Multiplication by $x_0$ on $\Omega'$ is injective since $\omega_{\tilde Y'}$ is flat over $\mathbb A^1=\mathrm{Spec}\,\overline K[x_0]$. 
By Lemma \ref{filmod}, $\Omega$ is endowed with a filtration such that 
$\mathrm{Gr}(\Omega)\cong \Omega'/x_0\Omega'.$ We thus have 
$\mathrm{Gr}(\Omega)^\sim\cong i_0^* (f'_*\omega_{Y'}).$ By Lemma \ref{omega}, 
$\mathrm{Gr}(\Omega)^\sim$ is supported on $C(\overline X_\infty)$.
Since $f'_*\omega_{Y'}$ is a coherent $\mathcal O_{\mathbb V'}$-module, 
$\Omega'$ is a finitely generated over $C$, and hence $\mathrm{Gr}(\Omega)$ is finitely generated over 
$\mathrm{Gr}(A)$. Taking into account of the $H$-action, the filtration on $\Omega$ is $H$-stable. 
\end{proof}

We have 
$$f_*\omega_Y\otimes_{\mathcal O_{\mathbb V}}\mathcal D_{\mathbb V}\cong
f_*\omega_Y\otimes_{\mathcal O_{\mathbb V}} \mathcal O_{\mathbb V}[\partial_{x_1},\ldots, \partial_{x_m}]
\cong f_*\omega_Y\otimes_{\overline K}\overline K[\partial_{x_1},\ldots, \partial_{x_m}].$$
Regard $f_*\omega_Y\otimes_{\mathcal O_{\mathbb V}} \mathcal D_{\mathbb V}$ as a right $\mathcal D_{\mathbb V}$-module
with the naive action. Let $\mathbb V^\vee$ be the dual space of $\mathbb V$, and let 
$(x'_1,\ldots, x'_m)$ be the coordinate on $\mathbb V^\vee$ dual to the coordinate $(x_1,\ldots, x_m)$.
The Fourier transform 
$\mathfrak F_{\mathcal L}(f_*\omega_Y\otimes_{\mathcal O_{\mathbb V}} \mathcal D_{\mathbb V})$ is a coherent
$\mathcal D_{\mathbb V^\vee}$-module which can be described as follows.  As a vector space, 
$\Gamma(\mathbb V^\vee, \mathfrak F_{\mathcal L}(f_*\omega_Y\otimes_{\mathcal O_{\mathbb V}} \mathcal D_{\mathbb V}))$
coincides with $\Omega \otimes_{\overline K} \overline K[\partial_{x_1},\ldots, \partial_{x_m}]$.  
Right multiplication by $x'_i$ (resp. $\partial_{x'_i}$) on 
$\Gamma(\mathbb V^\vee, \mathfrak F_{\mathcal L}(f_*\omega_Y\otimes_{\mathcal O_{\mathbb V}} \mathcal D_{\mathbb V}))$
is the right multiplication by 
$-\partial_{x_i}$ (resp. $x_i$) on $\Omega\otimes_{\overline K}\overline K[\partial_{x_1},\ldots, \partial_{x_m}]$.
Endow $\mathfrak F_{\mathcal L}(f_*\omega_Y\otimes_{\mathcal O_{\mathbb V}} \mathcal D_{\mathbb V})$ with the 
$H$-equivariant filtration defined by $\{\Omega_n\otimes_{\overline K}\overline K[\partial_{x_1},\ldots, \partial_{x_m}]\}$.
Denote the image of $\partial_{x'_i}$ in $\mathrm{Gr}(\mathcal D_{\mathbb V^\vee})$
by $x_i$, where $\mathcal D_{\mathbb V^\vee}$ is filtrated by the degree of differential operators. Then we have 
$$\Gamma(\mathbb V^\vee, \mathrm{Gr}(\mathcal D_{\mathbb V^\vee}))
\cong \overline K[x_1,\ldots, x_n,x'_1,\ldots, x'_n].$$

\begin{lemma}\label{finiteom} The above $H$-stable filtration is a good filtration on the right $\mathcal D_{\mathbb V^\vee}$-module
$\mathfrak F_{\mathcal L}(f_*\omega_Y\otimes_{\mathcal O_{\mathbb V}} \mathcal D_{\mathbb V})$, and we 
have isomorphisms
\begin{eqnarray}\label{eqn:isogr}
\begin{array} {rcl}
\Gamma(\mathbb V^\vee,\mathrm{Gr}(\mathfrak F_{\mathcal L}(f_*\omega_Y\otimes_{\mathcal O_{\mathbb V}} \mathcal D_{\mathbb V})))
&\cong& \mathrm{Gr}(\Omega)\otimes_{\overline K}  \overline K[x'_1,\ldots, x'_m], \\
\mathrm{Gr}(\mathfrak F_{\mathcal L}(f_*\omega_Y\otimes_{\mathcal O_{\mathbb V}} \mathcal D_{\mathbb V}))
&\cong& \pi'_* \pi^* i_0^* f'_*\omega_{Y'},
\end{array}
\end{eqnarray} where $\pi: \mathbb V\times\mathbb V^\vee\to\mathbb V$ and $\pi': \mathbb V\times\mathbb V^\vee\to\mathbb V^\vee$ 
are the projections.
 \end{lemma}

\begin{proof} We need to prove the following statements.
\begin{enumerate}[(i)]
\item Each $\Omega_n\otimes_{\overline K}\overline K[\partial_{x_1},\ldots, \partial_{x_m}]$ is 
a $\overline K[x'_1,\ldots, x'_m]$-module. 
\item $\partial_{x'_1}^{i_1}\cdots \partial_{x'_m}^{i_m}$ maps 
$\Omega_n \otimes_{\overline K}\overline K[\partial_{x_1},\ldots, \partial_{x_m}]$ to 
$\Omega_{n+i_1+\cdots+i_m}\otimes_{\overline K}\overline K[\partial_{x_1},\ldots, \partial_{x_m}]$.
\item We have the isomorphisms (\ref{eqn:isogr}), and 
$\mathrm{Gr}(\Omega\otimes_{\overline K}\overline K[\partial_{x_1},\ldots, \partial_{x_m}])$
is finitely generated over 
$\Gamma(\mathbb V^\vee,\mathrm{Gr}(\mathcal D_{\mathbb V^\vee}))\cong  \overline K[x_1,\ldots, x_n,x'_1,\ldots, x'_n]$. 
\end{enumerate}

(i) Let us show each $\Omega_n\otimes_{\overline K}\overline K[\partial_{x_1},\ldots, \partial_{x_m}]$ is 
stable under the right multiplication by $x'_i$. Indeed, we have 
\begin{eqnarray}\label{Fouromega1}
(\Omega_n\otimes_{\overline K}\overline K[\partial_{x_1},\ldots, \partial_{x_m}])x'_i&=&
(\Omega_n\otimes_{\overline K}\overline K[\partial_{x_1},\ldots, \partial_{x_m}])(-\partial_{x_i})\\
&\subset& \Omega_n\otimes_{\overline K}\overline K[\partial_{x_1},\ldots, \partial_{x_m}]. \nonumber
\end{eqnarray}

(ii) Let us show the right multiplication by $\partial_{x'_i}$ 
maps $\Omega_n\otimes_{\overline K}\overline K[\partial_{x_1},\ldots, \partial_{x_m}]$ to 
$\Omega_{n+1}\otimes_{\overline K}\overline K[\partial_{x_1},\ldots, \partial_{x_m}]$. 
We have 
\begin{eqnarray}\label{Fouromega2}
&&(\omega\otimes \partial_{x_1}^{j_1}\cdots \partial_{x_n}^{j_n})\partial_{x'_i}
= \omega\otimes \partial_{x_1}^{j_1}\cdots \partial_{x_n}^{j_n}x_i \nonumber\\
&=&\omega x_i\otimes \partial_{x_1}^{j_1}\cdots \partial_{x_n}^{j_n}
+\omega\otimes j_i \partial_{x_1}^{j_1}\cdots \partial_{x_i}^{j_i-1}\cdots \partial_{x_n}^{j_n}.
\end{eqnarray}
If $\omega\in \Omega_n$, then $\omega x_i \in \Omega_{n+1}$. So the right hand side lies in 
$\Omega_{n+1}\otimes_{\overline K}\overline K[\partial_{x_1},\ldots, \partial_{x_m}]$. 

(iii) The formulas (\ref{Fouromega1})-(\ref{Fouromega2}) show that 
via the canonical isomorphism of vector spaces
$$\mathrm{Gr}(\Omega\otimes_{\overline K}\overline K[\partial_{x_1},\ldots, \partial_{x_m}])\cong 
\mathrm{Gr}(\Omega)\otimes_{\overline K}\overline K[x'_1,\ldots, x'_m],$$ 
the multiplication by 
$x'_i$ (resp. $x_i$) on $\mathrm{Gr}(\Omega\otimes_{\overline K}\overline K[\partial_{x_1},\ldots, \partial_{x_m}])$
coincides with the multiplication by 
$-x'_i$ (resp. $x_i$) on $\mathrm{Gr}(\Omega)\otimes_{\overline K}\overline K[x'_1,\ldots, x'_m]$. 
Let $a$ be the isomorphism of rings 
$$a: \overline K[x'_1,\ldots, x'_n]\to  \overline K[x'_1,\ldots, x'_n],\quad 
 x'_i\mapsto -x'_i.$$
 We then have the following isomorphisms of modules over $\overline K[x_1,\ldots, x_m,x'_1,\ldots, x'_m]$:
 \begin{eqnarray*}
 \mathrm{Gr}(\Omega\otimes_{\overline K}\overline K[\partial_{x_1},\ldots, \partial_{x_m}])&\cong &
(\mathrm{Gr}(\Omega)\otimes_{\overline K}\overline K[x'_1,\ldots, x'_m])\otimes_{\overline K[x'_1,\ldots, x'_m],a} \overline K[x'_1,\ldots, x'_n]\\
&\cong& \mathrm{Gr}(\Omega)\otimes_{\overline K}\overline K[x'_1,\ldots, x'_m].
\end{eqnarray*}
By Corollary \ref{cor:fom}, $\mathrm{Gr}(\Omega)$ is finitely generated over $\overline K[x_1,\ldots, x_m]$.
So $\mathrm{Gr}(\Omega)\otimes_{\overline K}\overline K[x'_1,\ldots, x'_m]$ is finitely generated over
$\overline K[x_1,\ldots, x_n,x'_1,\ldots, x'_n]$. 
\end{proof}

Let $\mathfrak g$ be the Lie algebra of $G$. We have
$\mathfrak L(H)=\mathfrak g\times\mathfrak g$. For any $\xi=(\xi_1, \xi_2)\in \mathfrak L(H)$, 
recall that $L_\xi$ is the section of the $\mathcal O_{\mathbb V}$-module 
$\mathcal T_{\mathbb V}$ of tangent vectors on $\mathbb V$
whose value at $A\in\mathbb V$ is 
the tangent vector 
$$L_\xi= \frac{\mathrm d}{\mathrm dt}\Big|_{t=0} (e^{t\xi_1} Ae^{-t\xi_2}).$$
So far we view $\mathcal T_{\mathbb V}$ as a subsheaf of $\mathcal D_{\mathbb V}$
and hence view $L_\xi$ as a differential operator.  Let 
$$\pi: T^*\mathbb V\cong \mathbb V\times \mathbb V^\vee\to \mathbb V$$ 
be the projection for the cotangent bundle $T^*\mathbb V\cong \mathbb V\times \mathbb V^\vee$ of $\mathbb V$. 
Then we have 
$$\pi_*\mathcal O_{T^*\mathbb V}\cong\bigoplus_{j=0}^\infty \mathrm{Sym}^j\mathcal T_{\mathbb V}.$$
Thus we can also view $\mathcal T_{\mathbb V}$ 
as a subsheaf of $\pi_* \mathcal O_{T^*\mathbb V}$, and view $L_\xi$ as a section in 
$$\Gamma(\mathbb V, \pi_* \mathcal O_{T^*{\mathbb V}})\cong 
\Gamma(\mathbb V\times \mathbb V^\vee, \mathcal O_{T^*{\mathbb V}}).$$ So $L_\xi$ defines a function on 
$T^*\mathbb V\cong 
\mathbb V\times \mathbb V^\vee$. 
For any point $$(A, A')=((A_1, \ldots, A_N), (A'_1, \ldots, A'_N))\in \mathbb V\times\mathbb V^\vee,$$
we have 
\begin{eqnarray*}
L_\xi(A, A')&=&\frac{\mathrm d}{\mathrm dt}\Big|_{t=0} \sum_{j=1}^N\mathrm{Tr}(\rho_j(e^{t\xi_1}) A_j\rho_j(e^{-t\xi_2}) A'_j)\\
&=& \frac{\mathrm d}{\mathrm dt}\Big|_{t=0} \sum_{j=1}^N\mathrm{Tr}(A_j\rho_j(e^{-t\xi_2}) A'_j\rho_j(e^{t\xi_1}) ).
\end{eqnarray*}
We also view $L_\xi$ as the section of the $\mathcal O_{\mathbb V^\vee}$-module 
$\mathcal T_{\mathbb V^\vee}$ of tangent vectors on $\mathbb V^\vee$ (and hence a section of $\mathcal D_{\mathbb V^\vee}$)
whose value at $A'\in\mathbb V^\vee$ is 
the tangent vector 
$$L^*_\xi= \frac{\mathrm d}{\mathrm dt}\Big|_{t=0} (e^{-t\xi_2} A'e^{t\xi_1}).$$

\begin{lemma}\label{lm:GrM} Endow the modified hypergeometric $\mathcal D_{\mathbb V^\vee}$-module 
$$\mathcal M=\mathfrak F_{\mathcal L}(\mathcal N)=
\mathfrak F_{\mathcal L}\Big((\omega_Y\otimes_{\mathcal{O}_{\mathbb V}}\mathcal D_{\mathbb V}
)\big/\sum_{\xi\in\mathfrak g\times\mathfrak g}L_\xi (\omega_{Y}\otimes_{\mathcal{O}_{\mathbb V}}
\mathcal D_{\mathbb V})\Big)$$
with the good filtration induced by the good $H$-stable filtration on the right $\mathcal D_{\mathbb V^\vee}$-module
$\mathfrak F_{\mathcal L}(\omega_Y\otimes_{\mathcal O_{\mathbb V}} \mathcal D_{\mathbb V})$ in Lemma \ref{finiteom}.
Then the canonical epimoprhism
$$\mathrm{Gr}(\mathfrak F_{\mathcal L}(\omega_Y\otimes_{\mathcal O_{\mathbb V}} \mathcal D_{\mathbb V}))\twoheadrightarrow
\mathrm{Gr}(\mathcal M)$$
induces an an epimorphism 
$$\mathrm{Gr}(\mathfrak F_{\mathcal L}(\omega_Y\otimes_{\mathcal O_{\mathbb V}} \mathcal D_{\mathbb V}))
\otimes_{\pi'_*\mathcal{O}_{T^*\mathbb V^\vee}}
\Big(\pi'_*\mathcal O_{T^*{\mathbb V^\vee}}\Big/\sum_{\xi\in\mathfrak g\times \mathfrak g}
L_\xi\pi'_*\mathcal O_{T^*{\mathbb V^\vee}}\Big)\twoheadrightarrow \mathrm{Gr}(\mathcal M),$$
where we regard $L_\xi$ as a section of 
$\pi'_*\mathcal{O}_{T^*\mathbb V^\vee}\cong \bigoplus_{j=0}^\infty \mathrm{Sym}^j\mathcal T_{\mathbb V^\vee}$ 
on the left-hand side.  
\end{lemma}

\begin{proof} It suffices to show $L_\xi$ lies 
in the annihilator of $\mathrm{Gr}(\mathcal M)$, that is, 
$$(\Omega_n\otimes_{\overline K}\overline K[\partial_{x_1},\ldots, \partial_{x_m}])L^*_\xi 
\subset \Omega_n\otimes_{\overline K}\overline K[\partial_{x_1},\ldots, \partial_{x_m}].$$
Fix a basis of $V_j$ for each $j\in\{1,\ldots, N\}$, and write elements in $\mathrm{End}(V_j)$ by matrices $\big(x^{(j)}_{k_jl_j}\big)$
with respect to this basis. We thus get a linear coordinate system
$$\Big(x^{(j)}_{k_jl_j}\Big)_{j\in\{1,\ldots, N\}, \, k_j, l_j\in\{1, \ldots, \mathrm{dim}\,V_j\}}$$  
for $\mathbb V=\prod_{j=1}^N\mathrm{End}(V_j)$. Note that $\mathbb V$ is self-dual with respect to the pairing 
\begin{eqnarray*}
\mathbb V\times \mathbb V&\to& \overline K, \\
\Big(\big((x^{(1)}_{k_1l_1}),\ldots, (x^{(N)}_{k_Nl_N})\big), \big((y^{(1)}_{k_1l_1}),\ldots, (y^{(N)}_{k_Nl_N})\big)\Big)
&\mapsto& \sum_{j=1}^N \mathrm{Tr}((x^{(j)}_{k_jl_j})(y^{(j)}_{k_jl_j}))=\sum_{j=1}^N \sum_{k_j, l_j=1}^{\mathrm{dim}\,V_j} 
x^{(j)}_{k_jl_j}y^{(j)}_{l_jk_j}.
\end{eqnarray*}
So the dual coordinate system of $\Big(x^{(j)}_{k_jl_j}\Big)_{j\in\{1,\ldots, N\}, \, k_j, l_j\in\{1, \ldots, \mathrm{dim}\,V_j\}}$ is 
$\Big(x'^{(j)}_{k_jl_j}\Big)_{j\in\{1,\ldots, N\}, \, k_j, l_j\in\{1, \ldots, \mathrm{dim}\,V_j\}}$ 
with $x'^{(j)}_{k_jl_j}=x^{(j)}_{l_jk_j}$. The Fourier transform $\Gamma(\mathbb V^\vee,\mathcal M)$ of 
$\Gamma(\mathbb V, \mathcal N)$ 
is the same as $\Gamma(\mathbb V, \mathcal N)$ as a vector space, and the right multiplication by $x^{(j)}_{k_jl_j}$
(resp. $\partial_{x^{(j)}_{k_jl_j}}$) on $\Gamma(\mathbb V^\vee,\mathcal M)$ is the right multiplication by $-\partial_{x^{(j)}_{l_jk_j}}$
(resp. $x^{(j)}_{l_jk_j}$) on $\Gamma(\mathbb V, \mathcal N)$ .

For any $\xi=(\xi_1,\xi_2)\in \mathfrak g\times\mathfrak g$, let $(c^{(j)}_{1k_jl_j})$ and $(c^{(j)}_{2k_jl_j})$ be the matrices 
for $\frac{\mathrm d}{\mathrm dt}\big|_{t=0}\rho_j(e^{t\xi_1})$ and $\frac{\mathrm d}{\mathrm dt}\big|_{t=0}\rho_j(e^{t\xi_2})$, respectively.
Then $L_\xi$ considered as a tangent vector field on $\mathbb V$ is 
$$L_\xi =\sum_{j=1}^N \sum_{k_j, l_j,m_j=1}^{\mathrm{dim}\,V_j}  
\Big(c^{(j)}_{1k_j m_j} x^{(j)}_{m_jl_j}-c^{(j)}_{2m_j l_j} x^{(j)}_{k_jm_j}\Big)\partial_{x^{(j)}_{k_jl_j}},$$
and $L_\xi$ considered as a tangent vector field on $\mathbb V^\vee=\mathbb V$ is 
$$L^*_\xi =\sum_{j=1}^N \sum_{k_j, l_j,m_j=1}^{\mathrm{dim}\,V_j}  \Big(c^{(j)}_{1m_j l_j} 
x^{(j)}_{k_jm_j}-c^{(j)}_{2k_j m_j} x^{(j)}_{m_jl_j}\Big)\partial_{x^{(j)}_{k_jl_j}}.$$
For any $\omega\in \Omega_n$, we have $\omega L_\xi \in \Omega_n$ since
the filtration is $H$-stable and hence stable under the action of the Lie derivative by $L_\xi$. We have
\begin{eqnarray*}
&&\Big(\omega\otimes \prod \partial^{i^{(j)}_{p_jq_j}}_{x^{(j)}_{p_jq_j}}\Big) L^*_\xi\\
&=& \sum_{j=1}^N \sum_{k_j, l_j,m_j=1}^{\mathrm{dim}\,V_j} \Big(\omega\otimes \prod \partial^{i^{(j)}_{p_jq_j}}_{x^{(j)}_{p_jq_j}}\Big) 
\Big(-c^{(j)}_{1m_j l_j} 
\partial_{x^{(j)}_{m_jk_j}}+c^{(j)}_{2k_j m_j} \partial_{x^{(j)}_{l_jm_j}}\Big)x^{(j)}_{l_jk_j}\\
&\stackrel{(\ref{Fouromega2})}
\equiv& \sum_{j=1}^N \sum_{k_j, l_j,m_j=1}^{\mathrm{dim}\,V_j} \omega x^{(j)}_{l_jk_j}\otimes \Big(-c^{(j)}_{1m_j l_j} 
\partial_{x^{(j)}_{m_jk_j}}+c^{(j)}_{2k_j m_j} \partial_{x^{(j)}_{l_jm_j}}\Big) \prod \partial^{i^{(j)}_{p_jq_j}}_{x^{(j)}_{p_jq_j}} 
\mod \Omega_n\otimes_{\overline K}\overline K[\partial_{x'^{(j)}_{k_jl_j}}]\\
&\equiv& -\omega\otimes L_\xi \prod \partial^{i^{(j)}_{p_jq_j}}_{x^{(j)}_{p_jq_j}}
\mod  \Omega_n\otimes_{\overline K}\overline K[\partial_{x'^{(j)}_{k_jl_j}}] \\
&\equiv& -L_\xi\Big(\omega\otimes  \prod \partial^{i^{(j)}_{p_jq_j}}_{x^{(j)}_{p_jq_j}}\Big)
-\omega L_\xi \otimes  \prod \partial^{i^{(j)}_{p_jq_j}}_{x^{(j)}_{p_jq_j}}
\mod \Omega_n\otimes_{\overline K}\overline K[\partial_{x'^{(j)}_{k_jl_j}}] \\
&\equiv& -L_\xi\Big(\omega\otimes  \prod \partial^{i^{(j)}_{p_jq_j}}_{x^{(j)}_{p_jq_j}}\Big)
\mod \Omega_n\otimes_{\overline K}\overline K[\partial_{x'^{(j)}_{k_jl_j}}].
\end{eqnarray*}
But $-L_\xi\Big(\omega\otimes  \prod \partial^{i^{(j)}_{p_jq_j}}_{x^{(j)}_{p_jq_j}}\Big)$ vanishes in 
$\Gamma(\mathbb V, \mathcal N)=\Gamma(\mathbb V, (\omega_Y\otimes_{\mathcal{O}_{\mathbb V}}\mathcal D_{\mathbb V}
)\big/\sum_{\xi\in\mathfrak L(H)}L_\xi (\omega_{Y}\otimes_{\mathcal{O}_{\mathbb V}}
\mathcal D_{\mathbb V}))$. 
Thus $L_\xi$ lies in the annihilator of $\mathrm{Gr}(\mathcal M)$.
\end{proof}

Theorem \ref{thm:hypD} is a consequence of the following. 

\begin{theorem}\label{modifiedhyp} The hypergeometric $\mathcal D$-module $\mathcal Hyp_!$
and the modified hypergeometric $\mathcal D$-module 
$\mathcal M$ are holonomic, and are integrable connections on $\mathbb V^{\mathrm{gen}}$. Moreover, we have
$$\mathrm{rank}\,(\mathcal Hyp_!\big|_{\mathbb V^{\mathrm{gen}}})\leq 
\mathrm{rank}\,(\mathcal M\big|_{\mathbb V^{\mathrm{gen}}})\leq d!\int_{\Delta_\infty\cap\mathfrak C} 
\prod_{\alpha \in R^+} \frac{\lambda(H_\alpha)^2}{\rho(H_\alpha)^2}\mathrm d\lambda.$$
\end{theorem}

\begin{proof} Since $\mathcal Hyp_!$ is a direct factor fact of $\mathcal M$, it suffices to prove the assertions
for $\mathcal M$. By Proposition \ref{Nequi}, 
$\mathcal N$ is an $H$-equivariant $\mathcal D$-module with support contained in $X$. By Lemma 
\ref{lm:Kapranovorbit} (ii), $X$ has finitely many $H$-orbits. So 
$\mathcal N$ is regular and holonomic by \cite[II 5 Theorem]{Hotta}.  
Its Fourier transform $\mathcal M$ is holonomic. 
By Lemmas \ref{finiteom} and \ref{lm:GrM}, we have an epimorphism 
\begin{eqnarray}\label{dominantepi}
\pi'_*\pi^*i_0^* f'_*\omega_{Y'}\otimes_{\pi'_*\mathcal{O}_{T^*\mathbb V^\vee}}
\Big(\pi'_*\mathcal O_{T^*{\mathbb V^\vee}}\Big/\sum_{\xi\in\mathfrak g\times \mathfrak g}
L_\xi\pi'_*\mathcal O_{T^*{\mathbb V^\vee}}\Big)\twoheadrightarrow \mathrm{Gr}(\mathcal M).
\end{eqnarray}
Suppose $A'$ is a point in $\mathbb V^{\mathrm{gen}}\subset \mathbb V^\vee$. Let $f_{A'}: \mathbb V\to\mathbb A^1$ be the morphism
$$f_{A'}:\mathbb V\to\mathbb A^1,\quad A' \mapsto\sum_{j=1}^N\mathrm{Tr}(A_jA'_j),$$
and let $F_{A'}$ be the linear subspace of $\mathbb V$ defined by 
$$F_{A'}=\bigcap_{(\xi_1,\xi_2)\in\mathfrak g\times\mathfrak g}\Big\{A\in\mathbb V: 
\frac{\mathrm d}{\mathrm dt}\Big|_{t=0} 
f_{A'}(e^{t\xi_1} Ae^{-t\xi_2})=L_\xi(A, A')=0\Big\}.$$ Then each $A\in F_{A'}$ is a critical point for the function 
$f_{A'}$ restricted to the $H$-orbit $GAG$. By Lemma \ref{lm:Kapranovorbit} (iii) and the fact that $A'\in \mathbb V^{\mathrm{gen}}$, we have
$$C(\overline X_\infty)\cap F_{A'}=\{0\}.$$ 
$\pi'_*\pi^*i^*_0f'_*\omega_{Y'}\otimes_{\pi'_*\mathcal{O}_{T^*\mathbb V^\vee}}
\big(\pi'_*\mathcal O_{T^*{\mathbb V^\vee}}\big/\sum_{\xi\in\mathfrak g\times \mathfrak g}
L_\xi\pi'_*\mathcal O_{T^*{\mathbb V^\vee}}\big)$ is a coherent $\pi'_*\mathcal O_{T^*{\mathbb V^\vee}}$-module.
It defines a coherent $\mathcal O_{T^*{\mathbb V^\vee}}$-module on $T^*{\mathbb V^\vee}\cong \mathbb V\times \mathbb V^\vee$,
which we denote by the same notation. 
Let $i_{A'}$ be the closed immersion 
$$i_{A'}: \mathbb V\to \mathbb V\times 
\mathbb V^\vee,\quad A\mapsto (A, A').$$ 
We have 
\begin{eqnarray*}
i_{A'}^*\Big(\pi'_*\pi^*i_0^*f'_*\omega_{Y'}\otimes_{\pi'_*\mathcal{O}_{T^*\mathbb V^\vee}}
\Big(\pi'_*\mathcal O_{T^*{\mathbb V^\vee}}\Big/\sum_{\xi\in\mathfrak g\times \mathfrak g}
L_\xi\pi'_*\mathcal O_{T^*{\mathbb V^\vee}}\Big)\Big)\cong 
i_0^*f'_*\omega_{Y'}
\otimes_{\mathcal{O}_{\mathbb V}}\mathcal{O}_{F_{A'}},
\end{eqnarray*} which is supported at $C(\overline X_\infty)\cap F_{A'}=\{0\}$. This is true for all 
$A'\in\mathbb V^{\mathrm{gen}}\subset \mathbb V^\vee$. 
So we have 
$$(\mathbb V\times \mathbb V^{\mathrm{gen}})\cap \mathrm{supp}\, 
\Big(\pi'_*\pi^*i_0^*f'_*\omega_{Y'}\otimes_{\pi'_*\mathcal{O}_{T^*\mathbb V^\vee}}
\Big(\pi'_*\mathcal O_{T^*{\mathbb V^\vee}}\Big/\sum_{\xi\in\mathfrak g\times \mathfrak g}
L_\xi\pi'_*\mathcal O_{T^*{\mathbb V^\vee}}\Big)\Big)
=0\times \mathbb V^{\mathrm{gen}}.$$
Let $\mathrm{char}\,\mathcal M$ be the characteristic cycle of $\mathcal M$ in 
$T^*(\mathbb V^\vee)\cong \mathbb V\times \mathbb V^\vee$. 
Since we have the epimorphism (\ref{dominantepi}), $(\mathbb V\times 
\mathbb V^{\mathrm{gen}})\cap(\mathrm{char}\,\mathcal M)$ is
$0\times \mathbb V^{\mathrm{gen}}$ with multiplicity bounded above by 
\begin{eqnarray}\label{bound}
\dim_{\overline K} (i_0^*f'_*\omega_{Y'}
\otimes_{\mathcal{O}_{\mathbb V}} \mathcal{O}_{F_{A'}}). 
\end{eqnarray}	
So $\mathcal M|_{\mathbb V^{\mathrm{gen}}}$ is an integrable connection with rank bounded above by (\ref{bound}).	
In the vector space $\mathbb V'=\overline K\times\mathbb V$, we have 
\begin{eqnarray*}
X'\cap (\{0\}\times F_{A'})&=& C(\overline X_\infty)\cap (\{0\}\times F_{A'})= 0,\\
\dim_{\overline K} (i_0^*f'_*\omega_{Y'}
\otimes_{\mathcal{O}_{\mathbb V}} \mathcal{O}_{F_{A'}})&=& \dim_{\overline K} (f'_*\omega_{Y'}
\otimes_{\mathcal{O}_{\mathbb V'}} \mathcal{O}_{\{0\}\times F_{A'}})
\end{eqnarray*}
Choose a linear subspace $E$ of $\mathbb V'$ containing $0\times F_{A'}$ 
such that $$\mathrm{codim}\, E=\mathrm{dim}\, X' =d+1,\quad  X'\cap E=\{0\}.$$ 
($X'\cap E$ is conical. Then intersection is $\{0\}$ if it is zero dimensional.)
Then from the bound (\ref{bound}), we get
\begin{eqnarray}\label{2ndbound}
\mathrm{rank}(\mathcal M|_{\mathbb V^{\mathrm{gen}}})\leq \dim_{\overline K}(f'_*\omega_{Y'}
\otimes_{\mathcal{O}_{\mathbb V'}} \mathcal{O}_E).
\end{eqnarray}
Recall that for any two coherent $\mathcal O_{\mathbb V'}$-modules $\mathcal F$ and $\mathcal G$ on 
$\mathbb V'$ such that the intersection of their supports has dimension zero,  their intersection multiplicity is defined 
to be the alternating sum 
\[\chi(\mathcal F,\mathcal G)=\sum_{i}(-1)^i\mathrm{length}(\mathcal Tor_i^{\mathcal{O}_{\mathbb V'}}(\mathcal F,
\mathcal G)).\]
Let $q': Y'\to X'$ be the canonical morphism, let $r=[K(\mathbb G_m\times G): K(X')]$ be the degree of the 
extension of the function fields of $Y'$ and $X'$, and let $\eta$ be the generic point of $X'$.
Then $$(q'_*\omega_{Y'})_{\eta}\cong K(\mathbb G_m\times G)\cong \mathcal O_{X',\eta}^r.$$
Since $X'$ is affine, we may find global section 
$s_1,\ldots, s_r\in \Gamma(X', q'_*\omega_{Y'})$ so that their germs at $\eta$ form a basis 
of $(q'_*\omega_{Y'})_{\eta}$. Let $\phi:\mathcal O_{X'}^r\to q'_*\omega_{Y'}$ be the 
morphism mapping the standard basis of $\mathcal O_{X'}^r$ to $\{s_1, \ldots, s_r\}$. Then 
the dimensions of the 
supports of the sheaves  
$\mathrm{ker}\,\phi$ and $\mathrm{coker}\,\phi$ 
are strictly smaller than $\mathrm{dim}\,X'=\mathrm{codim}\,E$. By Serre's theorem \cite[V B Theorem 1 (3)]{Serre2}
on the vanishing of the intersection multiplicity, we have 
$$\chi(\mathrm{ker}\,\phi,\mathcal{O}_E)=\chi(\mathrm{coker}\,\phi,\mathcal{O}_E)=0.$$ So we have 
\begin{eqnarray}\label{deletequot}
\chi(f'_*\omega_{Y'},\mathcal{O}_E)
=r\chi(\mathcal{O}_{X'},\mathcal{O}_E).
\end{eqnarray} Here for convenience, we denote the direct images in $\mathbb V'$ of the sheaves $\mathrm{ker}\,\phi$,
$\mathrm{coker}\,\phi$, $\mathcal O_E$ and $\mathcal{O}_{X'}$ by the same symbols. Let $\overline X'$ and $\overline E$ be 
the closure of $X'$ and $E$ in $\mathbb P(\overline K\oplus \mathbb V')$, respectively. 
Then $\overline E$ is a projective space.  Using the fact that $\overline X'$ is conical, one can check that 
$$\overline X'\cap \overline E=X'\cap E=\{0\}.$$ So  
the intersection number of $\overline E$ and $\overline X'$ is given by 
$\chi(\mathcal{O}_{X'},\mathcal{O}_E)$.
This intersection number is also $\mathrm{deg}(\overline X')$. 
The Newton polytope at $\infty$ for the family of representations $\rho'_0,\rho'_1,\ldots, \rho'_N$ of $\mathbb G_m\times G$ is 
$\Delta'_\infty=\bigcup_{t\in [0,1]}t(1\times \Delta_{\infty})$. Moreover, we have
$$d=2|R^+|+\mathrm{dim}\,\Lambda_{\mathbb R}.$$
By Lemma \ref{degY} below, we have
\begin{eqnarray*}
\mathrm{deg}(\overline X')&=&\frac{(d+1)!}{r}\int_{\bigcup_{t\in [0,1]}t(1\times \Delta_{\infty}\cap\mathfrak C)}
\prod_{\alpha \in R^+} \frac{\lambda(H_\alpha)^2}{\rho(H_\alpha)^2}\mathrm dt \mathrm d\lambda\\
&=&
\frac{(d+1)!}{r}\int_0^1  t^d\mathrm d t \Big(\int_{\Delta_{\infty}\cap\mathfrak C} 
\prod_{\alpha \in R^+} \frac{\lambda(H_\alpha)^2}{\rho(H_\alpha)^2}\mathrm d\lambda\Big)\\
&=& \frac{d!}{r}\int_{\Delta_{\infty}\cap\mathfrak C} 
\prod_{\alpha \in R^+} \frac{\lambda(H_\alpha)^2}{\rho(H_\alpha)^2}\mathrm d\lambda,
\end{eqnarray*} 
So we have
\begin{eqnarray}\label{degree}
\chi(\mathcal{O}_{X'},\mathcal{O}_E)=\frac{d!}{r}\int_{\Delta_\infty\cap\mathfrak C} \prod_{\alpha \in R^+} 
\frac{\lambda(H_\alpha)^2}{\rho(H_\alpha)^2}\mathrm d\lambda.
\end{eqnarray}
By \cite[15.20]{Timashev} and \cite[5.70]{birational}, 
$\omega_{Y'}$ is Cohen-Macaulay.
By \cite[IV B Proposition 11]{Serre2},  $f'_*\omega_{Y'}$ is also Cohen-Macaulay. 
Let $\{l_1,\dots, l_d\}$ be a basis of $(\mathbb V'/E)^*$. 
Since $\mathrm{dim}\, X'\cap E=0$, $\{l_1,\dots, l_d\}$ is a regular sequence for 
$\omega_{Y'}$ by \cite[IV B Theorem 2]{Serre2}. The Koszul complex of this sequence computes 
$\mathcal Tor_i^{\mathcal{O}_{\mathbb V'}}(f'_*\omega_{Y'},\mathcal{O}_E)$.
So by \cite[IV A Proposition 2]{Serre2}, we have $$\mathcal Tor_i^{\mathcal{O}_{\mathbb V'}}
(f'_*\omega_{Y'},\mathcal{O}_E)=0 \quad 
(i\geq 1).$$ We conclude
\begin{eqnarray}\label{deletetor}
\dim_{\overline K}(f'_*\omega_{Y'}\otimes_{\mathcal{O}_{\mathbb V'}} \mathcal{O}_E)
=\chi(\omega_{Y'},\mathcal{O}_E).
\end{eqnarray}
Combining (\ref{deletequot}), (\ref{degree}) and (\ref{deletetor}) together, we get 
\begin{eqnarray*}
\dim_{\overline K} (f'_*\omega_{Y'}\otimes_{\mathcal{O}_{\mathbb V}} \mathcal{O}_E)
= d!\int_{\Delta_\infty\cap\mathfrak C} 
\prod_{\alpha \in R^+} \frac{(\lambda, \alpha)^2}{(\rho, \alpha)^2}\mathrm d\lambda.
\end{eqnarray*}
Taking into account of (\ref{2ndbound}), we get the required upper bound for the generic rank of $\mathcal M$. 
\end{proof}

\begin{lemma}\label{degY} Suppose the homogeneity condition holds for the family of representations 
$\rho_j: G\to \mathrm{GL}(V_j)$ $(j=1, \ldots, N)$. Let $X$ be the Zariski closure of 
$$\iota: G\to \mathbb V=\prod_{j=1}^N  \mathrm{End}(V_j),$$  let $r=[K(G): K(X)]$, and let $\overline X$ be the closure of 
$X$ in $\mathbb P(\overline K\oplus \mathbb V)$.
We have
	$$\deg (\overline X)=\frac{d!}{r}\int_{\Delta_{\infty}\cap\mathfrak C} 
	\prod_{\alpha \in R^+} \frac{\lambda(H_\alpha)^2}{\rho(H_\alpha)^2}\mathrm d\lambda.$$
\end{lemma}

\begin{proof}
Let $\overline Z$ be the normalization of $\overline X$,  let $\pi: \overline Z\to\overline X$ be the canonical morphism, 
and let $Z=\pi^{-1}(X)$. Then $Z$ is the normalization 
of $X$. Let $\tilde G$ be the image of the homomorphism 
$$G\to \prod_{j=1}^N\mathrm{GL}(V_j),\quad g\mapsto (\rho_1(g),\ldots, \rho_N(g)).$$
Then $Z$ is a normal affine algebraic monoid containing the reductive algebraic group $\tilde G$
as an open subgroup. We may identify the weight
lattice $\tilde \Lambda$ of $\tilde G$ with a sub-lattice of the weight lattice $\Lambda$ of $G$ 
with index $$[\Lambda:\tilde \Lambda]=r=[K(G): K(X)].$$ We may identify 
$\tilde \Lambda_{\mathbb R}$ with $\Lambda_{\mathbb R}$. 
The Lebesgue measure $\mathrm d\tilde \lambda$ on $\Lambda_R$ such that 
$\mathrm{vol}(\Lambda_{\mathbb R}/\tilde \Lambda)=1$ is identified with 
$\frac{1}{r}\mathrm d\lambda$. 

Let $H=G\times G$ act on $\overline K$ trivially. Then the action of $H$ on $\overline K\oplus \mathbb V$ 
induces actions of $H$ on $\mathbb P(\overline K\oplus \mathbb V)$ and on $\overline X$.  
Let $$D_{\infty}=\overline X\cap\{[0:v]: v\in\mathbb V-\{0\}\}.$$  
Then $D_\infty$ is $H$-stable. By the homogeneity condition, $X$ is conical, and locally the 
immersion $D_{\infty}\hookrightarrow \overline X$ 
is of the form $U\xrightarrow{0\times \mathrm{id}} \mathbb{A}^1\times U$, where $U$ is an open subset of the projectivization 
$\mathbb P(X)$ of $X$. In particular $D_{\infty}$ is a Cartier divisor and its generic point is regular in 
$\overline X$. We thus have a valuation $v_{D_{\infty}}: K(X)^*\to \mathbb Z$. 
By the homogeneity condition, the group $\mathbb G_m$ contained in the center of 
$G$ acts on $X$. For any nonzero $\mathbb{G}_m$-eigen-rational-function $f\in K(X)^*$ 
with eigencharacter 
$$\mathbb{G}_m\to \mathbb{G}_m,\quad \lambda\mapsto \lambda^n,$$ we have $v_{D_{\infty}}(f)=-n$.
Since the pullback by a birational map doesn't change the degree of a divisor (\cite[1.35 (6)]{birational}), we have 
$$\deg(\overline X)=\deg((\pi^*D_{\infty})^d).$$	
By \cite[4.1 Th\'eor\`eme]{Brion2} or \cite[Theorem 18.8]{Timashev} (which relies on the Weyl dimension formula) 
and the discussion on the 
measure $\mathrm d\tilde \lambda$ at the beginning, it suffices to show that the polytope 
$$P(\pi^*D_{\infty})=\{\lambda\in \Lambda_{\mathbb{R}}: v_{D}(\lambda)\geq 0\text{ for any }
(B^+\times B^-)\hbox{-stable divisor } D \hbox{ in } Z, \;\langle v_{D_{\infty}},\lambda\rangle\geq -1\}$$
associated to the divisor $\pi^*D_{\infty}$ is exactly $\Delta_{\infty}\cap\mathfrak C$, 
where for any $(B^+\times B^-)$-stable divisor $D$ in $Z$, 
$v_D$ is the valuation corresponding to $D$, and $v_D$ defines an element in $\mathrm{Hom}_{\mathbb Z}(\Lambda, \mathbb Z)$ by setting 
$$\langle v_D, \chi_f\rangle =v_D(f)$$ 
for any $(B^+\times B^-)$-eigen-rational-function $f$ on $Z$ with character $(\chi_f,-\chi_f)$.  
	 
Let $\mathcal C$ be the cone in $\mathrm{Hom}_{\mathbb Z}(\Lambda, \mathbb Z)\otimes_{\mathbb Z}\mathbb R$
generated by $v_D$ for all $(B^+\times B^-)$-stable divisor $D$ in $Z$.
By \cite[Theorem 27.12]{Timashev}, $\mathcal C$ is a strictly convex cone generated by all simple coroots 
$H_\alpha$ ($\alpha\in R^+$ is simple) and finitely many antidominant vectors. 
We may identify the dual convex cone 
$$\mathcal C^\vee=\{\lambda\in \Lambda_{\mathbb{R}}: \langle v_D, \lambda\rangle
\geq 0\text{ for all $(B^+\times B^-)$-stable divisor $D$ in $Z$}\}$$ of $\mathcal C$ 
with the cone spanned by $\chi_f$ as $f$ goes over the set 
$\mathcal O_{Z}(Z)^{(B^+\times B^-)}$ of $(B^+\times B^-)$-eigen-regular-functions 
on $Z$. Let
$\mathcal K$ be the cone generated by the weights of $V_j$ $(j=1, \ldots, N)$. 

\medskip
\noindent\emph{Claim}: We have 
$$\mathcal C^\vee=\mathcal K\cap \mathfrak C.$$

\medskip
Each representation 
$\tilde G\to \mathrm{GL}(V_j)$ is  
extendible to $Z$. By \cite[Proposition 27.16 (3)]{Timashev}, we have $\lambda_j\in \mathcal C^\vee$. 
So for any  vector $v$ in $\mathcal C$, we have 
$$\langle v, \lambda_j\rangle\geq 0.$$  
Any weights of $V_j$ is of the form $\lambda_j-\sum_{\alpha\in R^+} k_{j\alpha}\alpha$ for some 
$k_{j\alpha}\in \mathbb Z_{\geq 0}$. So any elements in $\mathcal K$ is of the form 
$$\sum_{j=1}^N r_j \Big(\lambda_j-\sum_{\alpha\in R^+} k_{j\alpha}\alpha\Big)$$ for some $r_j\in \mathbb R_{\geq 0}$. 
For any anti-dominant vector $v$ in $\mathcal C$, we have $\langle v,\lambda_j\rangle\geq 0$ and $\langle v,\alpha\rangle\le 0$ 
$(\alpha\in R^+)$. So we have 
$$\Big\langle v,\sum_{j=1}^N r_j\Big(\lambda_j-\sum_{\alpha\in R^+} k_{j\alpha}\alpha\Big)\Big\rangle \geq 0.$$
If $\sum_{j=1}^N r_j (\lambda_j-\sum_{\alpha\in R^+} k_{j\alpha}\alpha)$ lies in the dominant Weyl chamber, then we have 
$$\Big\langle H_\alpha, \sum_{j=1}^N r_j \Big(\lambda_j-\sum_{\alpha\in R^+} k_{j\alpha}\alpha\Big) \Big\rangle\geq 0$$ 
for all simple positive roots $\alpha$. Since $\mathcal C$ is generated by $H_\alpha$ $(\hbox{simple }\alpha\in R^+)$ and 
some anti-dominant vectors, we have $\mathcal C^\vee\supset \mathcal K\cap \mathfrak C$. 

Conversely, for any $f\in \mathcal O_Z(Z)^{(B^+\times B^-)}$, we can find an equation 
\begin{eqnarray}\label{inteqn}
f^n+a_1f^{n-1}+\cdots+a_n=0
\end{eqnarray} such that $a_i\in \mathcal O_X(X)$ are regular functions on $X$. 
The torus $T\times T$ acts on 
$\mathcal O_X(X)$ and each function 
in $\mathcal O_X(X)$ can be written as a linear combination of $(T\times T)$-eigenfuctions. Writing down the weight 
$(\chi_{f^n}, -\chi_{f^n})$-component 
of the equation (\ref{inteqn}), we conclude that there exists a positive integer $m$ such that $(m\chi_f,-m\chi_f)$ is a weight for the representation 
$\mathcal O_X(X)$.  As a $\overline K$-algebra, $\mathcal O_X(X)$
is generated by functions of the form $$X=(X_1, \ldots, X_N)\mapsto \mathrm{Tr}(A_jX_j)$$ 
with $A_j\in\mathrm{End}(V_j)$ $(j=1, \ldots, N)$. It follows that $(T\times T)$-weights for the representation 
$\mathcal O_X(X)$ are of the form $(\lambda,  -\mu)$ with $\lambda,\mu\in \mathcal K$. 
Hence $m\chi_f$  lies in $\mathcal K$. We have  $f\in \mathcal O_G(G)^{(B^+\times B^-)}$ and
$$\mathcal O_G(G)\cong \bigoplus_V \mathrm{End}(V),$$
where $V$ goes over the set of irreducible representations of $G$. So $(\chi_f, -\chi_f)$ is the weight of an element in 
$\mathrm{End}(V)^{(B^+\times B^-)}$ for some irreducible $V$. Hence $\chi_f\in \mathfrak{C}$.
This finishes the proof of the claim. 

By the homogeneity condition, the restriction of each $\rho_j$ to the subgroup $\mathbb G_m$ of $G$ 
is the scalar multiplication. So the Newton polytope $\Delta$ lies in a hyperplane of $\Lambda_{\mathbb R}$ not passing 
through the origin, and the Newton polytope $\Delta_\infty$ at $\infty$ is given by 
$$\Delta_\infty=\bigcup_{t\in [0,1]} t\Delta.$$ 
We have
$$\mathcal C^\vee=\mathcal K\cap \mathfrak C=\mathbb R_{\geq 0} \Delta\cap \mathfrak{C}.$$ 
The torus $\mathbb G_m$  in $G$ 
acts on each $V_j$ via the identity character by the homogeneity condition.
So $\langle v_{D_\infty},\lambda_j\rangle=-1$. The Weyl group $W$ acts trivially on
$\mathbb{G}_m$. So $\langle v_{D_\infty},\sigma(\lambda_j)\rangle=-1$ for any $\sigma\in W$.
We thus have $v_{D_\infty}|_\Delta=-1$. Therefore
\[P(\pi^*D_{\infty})=\{\lambda\in\mathcal C^\vee: \langle v_{D_\infty}, \lambda\rangle\geq -1\}
=\bigcup_{t\in [0,1]}t\Delta\cap \mathfrak{C}= \Delta_\infty\cap\mathfrak C.\qedhere\]
\end{proof}

\begin{proposition}\label{char0mainthm} Suppose $\overline K=\mathbb C$ and the homogeneity condition holds. 
The hypergeometric sheaves $\mathrm{Hyp}_{*,\mathbb C}$ and $\mathrm{Hyp}_{!,\mathbb C}$
are perverse sheaves on $\mathbb V$. Their restrictions to $\mathbb V^{\mathrm{gen}}$ 
come from local systems with rank $\leq d!\int_{\Delta_\infty\cap\mathfrak C} 
	\prod_{\alpha \in R^+} \frac{\lambda(H_\alpha)^2}{\rho(H_\alpha)^2}\mathrm d\lambda.$
\end{proposition}

\begin{proof} Under the homogeneity condition, $\iota_!\omega_G$ is regular holonomic and monodromic. So
its Fourier transform
$\mathcal Hyp_{!,\mathbb C}$ is regular holonomic by \cite[7.24]{Brylinski}.  
By Proposition \ref{basichyp_D} (iii), $\mathrm{Hyp}_{*, \mathbb C}$ is the solution complex of 
$\mathcal Hyp_{!,\mathbb C}$. By the 
the Riemann-Hilbert correspondence, $\mathrm{Hyp}_{*,\mathbb C}$ is a perverse sheaf on $\mathbb V$.
By Proposition \ref{modifiedhyp}, $\mathcal Hyp_{!,\mathbb C}$ is an integrable connection on 
$\mathbb V^{\mathrm{gen}}$ and its generic rank is bounded by $d!\int_{\Delta_\infty\cap\mathfrak C} 
\prod_{\alpha \in R^+} \frac{\lambda(H_\alpha)^2}{\rho(H_\alpha)^2}\mathrm d\lambda$.  So 
the restriction of $\mathrm{Hyp}_{*,\mathbb C}$ to $\mathbb V^{\mathrm{gen}}$ 
comes from a local system whose rank has the same bound. By duality, 
$\mathrm{Hyp}_{!,\mathbb C}$ has similar properties. 
\end{proof}

\section{Application to exponential sums}

In the rest of this paper, we work over a Dedekind domain $R$ whose fraction field $K$ is of characteristic $0$
and whose residue fields at maximal ideals are finite. 
Let $G_R$ be a split reductive group 
$R$-scheme, and let
$\rho_{j, R}: G_R\to \mathrm{GL}(V_{j, R})$ $(j=1, \ldots, N)$ be representations such that 
$\rho_{j,\overline K}$ are irreducible. Assume the morphism 
$$\iota_R: G_R\to \mathbb V_R=\prod_{j=1}^N \mathrm{End}_R(V_{j, R}), \quad g\mapsto (\rho_{1,R}(g),\ldots, \rho_{N,R}(g))$$
is quasi-finite. 

\begin{proposition}\label{char0mainthm'} Suppose the homogeneity condition holds over $R$.
The restrictions $\mathrm{Hyp}_{*,R}|_{\mathbb V_K}$ and 
$\mathrm{Hyp}_{!,R}|_{\mathbb V_K}$ to $\mathbb V_K:=\mathbb V\otimes_RK$
of the $\ell$-adic hypergeometric sheaves over $R$
are perverse sheaves on $\mathbb V_K$. Their restrictions to $\mathbb V_K^{\mathrm{gen}}$ 
come from lisse sheaves with ranks $\leq d!\int_{\Delta_\infty\cap\mathfrak C} 
\prod_{\alpha \in R^+} \frac{\lambda(H_\alpha)^2}{\rho(H_\alpha)^2}\mathrm d\lambda.$
\end{proposition}

\begin{proof} Choose an algebraic closed field $F$ containing both $\overline K$ and $\mathbb C$. 
Choose a complex reductive group $G_{\mathbb C}$ and representations 
$\rho_{j,\mathbb C}: G_{\mathbb C}\to \mathrm{GL}(V_{j, \mathbb C})$ $(j=1, \ldots, N)$ so that 
their base changes to $F$ coincide with the base changes of $\rho_{j,\overline K}: 
G_{\overline K}\to \mathrm{GL}(V_{j, \overline K})$ to $F$. Let $\iota_{\mathbb C}$ be the morphism 
$$\iota_{\mathbb C}:G_{\mathbb C}\to \mathbb V_{\mathbb C}:=\prod_{j=1}^NV_{j,\mathbb C},\quad 
g\mapsto (\rho_{1,\mathbb C}(g), \ldots, \rho_{N, \mathbb C}(g)).$$
Define the $\ell$-adic hypergeometric sheaves
$\mathrm{Hyp}^{(\ell)}_{*,\mathbb C}$ and 
$\mathrm{Hyp}^{(\ell)}_{!,\mathbb C}$ on $\mathbb V_{\mathbb C}:=\prod_{j=1}^NV_{j,\mathbb C}$ by 
$$\mathrm{Hyp}^{(\ell)}_{*,\mathbb C}=\mathrm{Four}_B(R\iota_{\mathbb C,*}\overline{\mathbb Q}_\ell[d]),
\quad \mathrm{Hyp}^{(\ell)}_{!,\mathbb C}=\mathrm{Four}_B(R\iota_{\mathbb C,!}\overline{\mathbb Q}_\ell[d]).$$
Since the Wang-Fourier transform commutes with base change, the base change of 
$\mathrm{Hyp}_{*,R}|_{\mathbb V_K}$ (resp. $\mathrm{Hyp}_{!,R}|_{\mathbb V_K}$) 
to $\mathbb V_K\otimes_KF$ coincides with the base change of 
$\mathrm{Hyp}^{(\ell)}_{*,\mathbb C}$ (resp. $\mathrm{Hyp}^{(\ell)}_{!,\mathbb C}$) 
to $\mathbb V_{\mathbb C}\otimes_{\mathbb C}F$. To prove our assertion, it suffices 
to prove the same holds for $\mathrm{Hyp}^{(\ell)}_{*,\mathbb C}$ and 
$\mathrm{Hyp}^{(\ell)}_{!,\mathbb C}$.

For any $\mathbb{C}$-scheme $X$ of finite type, 
we have a morphism of sites $\epsilon: X(\mathbb C)\to X_{\mathrm{et}}$, 
where $X(\mathbb C)$ is defined by the usual topology on the complex variety $X(\mathbb C)$. 
By \cite[6.1.2(C$'$)]{BBD}, $\epsilon^*$ commutes with Grothendieck's six functors. 
So $$\epsilon^*(\mathrm{Hyp}^{(\ell)}_{*,\mathbb C})\cong \mathrm{Hyp}_{*,\mathbb C},\quad
\epsilon^*(\mathrm{Hyp}^{(\ell)}_{!,\mathbb C})\cong \mathrm{Hyp}_{!,\mathbb C}.$$
(Here we use the fact that $\overline{\mathbb Q}_\ell\cong\mathbb C$ as fields). 
By Proposition \ref{char0mainthm}, $\mathrm{Hyp}_{*,\mathbb C}$ and $\mathrm{Hyp}_{!,\mathbb C}$ are perverse. 
So are $\mathrm{Hyp}^{(\ell)}_{*,\mathbb C}$ and 
$\mathrm{Hyp}^{(\ell)}_{!,\mathbb C}$ . 
To prove their restrictions to $\mathbb V_{\mathbb C}^{\mathrm{gen}}$ come from local systems,
we may apply Proposition \ref{perverseranklisse} in the Appendix and use the fact that $\epsilon^*$ preserves stalks.
\end{proof}

\begin{proposition}\label{mainprop} Suppose the homogeneity condition holds over $R$. Let 
$$U=\mathbb V_R-\overline {\mathbb V_K^{\mathrm{degen}}}$$ be the complement of the 
Zariski closure of $\mathbb V_K^{\mathrm{degen}}$ in $\mathbb V_R$.
There exists a finite set $S$ of maximal ideals of $R$ so that the restriction of $\mathrm{Hyp}_{R, !}$ 
to $U\times_{\mathrm{Spec}\, R}(\mathrm{Spec}\, R-S)$ is of the form $\mathcal H[\mathrm{dim}\,\mathbb V_K]$ for a lisse 
sheaf $\mathcal H$ with rank
$\leq d!\int_{\Delta_\infty\cap\mathfrak C} 
\prod_{\alpha \in R^+} 
\frac{(\lambda, \alpha)^2}{(\rho, \alpha)^2}\mathrm d\lambda.$
\end{proposition}

\begin{proof}
Let $\mathcal H^i=\mathcal H^i(\mathrm{Hyp}_{!,R})$, and let $\eta$ be the generic point of $\mathbb V^{\mathrm {gen}}_{K}$. 
Choose a sequence 
$$\emptyset=Z_0\subset Z_1\subset Z_2\subset\cdots \subset Z_s=U$$ of closed subsets in $U$ so that $\mathcal H^i|_{Z_j-Z_{j-1}}$ are 
lisse for all $i, j$. Let $\eta_{j1},\ldots, \eta_{jm_j}$ be the generic points of the irreducible components of $Z_j-Z_{j-1}$,
and let $S$ be the finite set of maximal ideals which are images of some $\eta_{jl}$. Suppose $x_1$ is a point in $U$ such that 
its image in $\mathrm{Spec}\,R$ lies outside $S$. We claim that a specialization morphism 
$$\mathcal H^i_{\bar x_1}\to \mathcal H^i_{\bar \eta}$$ is an isomorphism. 
By Proposition \ref{Qllisse} in the Appendix, this implies that each 
$\mathcal H^i|_{U\times_{\mathrm{Spec}\, R}(\mathrm{Spec}\, R-S)}$ is a lisse sheaf. 
By Proposition \ref{char0mainthm'}, $\mathcal H^i|_{U\times_{\mathrm{Spec}\, R}(\mathrm{Spec}\, R-S)}$ vanishes for 
$i\not=-\mathrm{dim}\,\mathbb V_K$. Let $\mathcal H= \mathcal H^{-\mathrm{dim}\,\mathbb V_K}|_{U\times_{\mathrm{Spec}\, R}(\mathrm{Spec}\, R-S)}$. 
Then $\mathrm{Hyp}_{!, R}|_{U\times_{\mathrm{Spec}\, R}(\mathrm{Spec}\, R-S)}\cong 
\mathcal H[\mathrm{dim}\,\mathbb V_K]|_{U\times_{\mathrm{Spec}\, R}(\mathrm{Spec}\, R-S)}$  and the rank of $\mathcal H$ 
has the required bounded by Proposition \ref{char0mainthm'}.
	
Choose $\eta_{j_1l_1}$ so that $x_1\in (Z_{j_1}-Z_{j_1-1})\cap \overline{\{\eta_{j_1l_1}\}}$. Let 
$\pi: \mathbb V_R\to \mathrm{Spec}\,R$ be the structure morphism. 
If $\pi(\eta_{j_1l_1})$ is a maximal ideal $\mathfrak m$ of $R$, then 
$\mathfrak m\in S$ by the definition of $S$, and $\pi(x_1)=\mathfrak m$ since $\pi(x_1)\in \overline{\{\pi(\eta_{j_1l_1})\}}$. 
This contradicts to the fact that $x_1$ lies over $\mathrm{Spec}\, R-S$. Thus $\pi(\eta_{j_1l_1})$
must be the generic point of $\mathrm{Spec}\,R$. So $$\eta_{j_1l_1}\in U\cap  \mathbb V_K= \mathbb V^{\mathrm {gen}}_{K}.$$ 
By Corollary \ref{char0mainthm'}, 
$\mathcal H^i|_{\mathbb V^{\mathrm {gen}}_{K}}$ is lisse. So each specialization morphism
$$\mathcal H^i_{\bar \eta_{j_1l_1}}\to \mathcal H^i_{\bar \eta}$$
is an isomorphism. Each specialization morphism 
$$\mathcal H^i_{\bar x_1}\to \mathcal H^i_{\bar \eta_{j_1l_1}}$$
is also an isomorphism since $\mathcal H^i$ is lisse on $Z_{j_1}-Z_{j_1-1}$.  This proves our claim. 
\end{proof}

In the rest of this paper, we do not assume the homogeneity condition condition holds 
for the data $(G_R, \rho_{1,R}, \ldots, \rho_{N,R})$. Consider the group scheme $\mathbb{G}_{m,R}\times_R G_R$ and 
the representations $$\rho'_{j,R}: \mathbb{G}_{m,R}\times_R G_R\to\mathrm{GL}(V_{j,R}), \quad (t,g)\mapsto t\rho_j(g)
\quad (j=1, \ldots, N)$$ 
together with the representation $$\rho'_{0,R}: \mathbb{G}_{m,R}\times_R G_R\to \mathrm{GL}(V_{0,R}),\quad (t,g)\mapsto t$$
with $V_{0,R}=R$. 
The tuple $$(\mathbb{G}_{m,R}\times_R G_R,\rho'_{0,R}, \rho'_{1,R},\ldots, \rho'_{N,R})$$ satisfies the homogeneity condition. 
Let $$\mathbb V'_R:=\prod_{j=0}^N\mathrm{End}(V_{j, R})
\cong  \mathbb{A}^1_{R}\times_R \mathbb{V}_R,$$ let $\iota'_R$ be the morphism
$$\iota'_R: \mathbb{G}_{m,R}\times_R G_R\to \prod_{j=0}^N\mathrm{End}(V_{j, R})
\cong  \mathbb{A}^1_{R}\times_R \mathbb{V}_R,\quad (t, g)\mapsto (t, t\rho_1(g), \ldots, t\rho_N(g)),$$
and let $$\mathrm{Hyp}'_{!,R}=\mathrm{Four}_B(R\iota'_{R, !}\overline{\mathbb Q}_\ell[d+1])$$ be the corresponding hypergeometric sheaf
over $R$. Let $\mathfrak m$ be a maximal ideal of $R$, let $k=R/\mathfrak m$, let $\iota'_k$ be the base change of $\iota'_R$, and let 
$$\mathrm{Hyp}'_{!,k}=\mathfrak F_{\psi, \mathbb V'_k}(R\iota'_{k, !}\overline{\mathbb Q}_\ell)[d+1],$$
where $\mathfrak F_{\psi, \mathbb V'_k}$ is the Deligne-Fourier transform for the vector bundle
$$\mathbb V'_k=\prod_{j=0}^N\mathrm{End}(V_{j, k})
\cong  \mathbb{A}^1_{k}\times_k \mathbb{V}_k\to \mathrm{Spec}\,k.$$ Finally let $\mathrm{Hyp}_{!,k}$ be 
the $\ell$-adic hypergeometric sheaf for the 
representations $\rho_{j, k}:G_k\to\mathrm{GL}(V_{j,k})$ obtained from $\rho_{j, R}$ modulo $\mathfrak m$.

\begin{proposition}\label{homo_nonhomo} Let $i_0$ and $i_{-1}$ be the closed immersions
\begin{eqnarray*}
i_0: \mathbb V_k\to \mathbb V'_k=\mathbb A^1_k\times_k \mathbb V_k, && x\mapsto (0,x),\\
i_{-1}: \mathbb V_k\to \mathbb V'_k=\mathbb A^1_k\times_k \mathbb V_k, && x\mapsto (-1,x),
\end{eqnarray*}
We have
$$
i^*_{-1}\mathfrak F_{\psi, \mathbb{A}^1_{\mathbb{V}_k}}(\mathrm{Hyp}'_{!,k})
\cong \mathrm{Hyp}_{!,k}[1](-1),\quad 
i_0^*\mathfrak F_{\psi, \mathbb{A}^1_{\mathbb{V}_k}}(\mathrm{Hyp}'_{!,k})|_{\{0\}\times \mathbb{V}_k}
\cong 0.
$$
where $\mathfrak F_{\psi, \mathbb{A}^1_{\mathbb{V}_k}}$ denotes the Deligne-Fourier transform for the line bundle 
$\mathbb V'_k\cong \mathbb{A}^1_{k} \times_k \mathbb{V}_k\to \mathbb{V}_k$.
\end{proposition}

\begin{proof} 
Let $\mathfrak F_{\psi, \mathbb V'_k/\mathbb A^1_k}$ be the Deligne-Fourier transform for the vector bundle 
$$\mathbb V'_k=\mathbb A^1_k\times_k\mathbb V_k\to\mathbb A^1_k.$$
One can check
$$\mathfrak F_{\psi, \mathbb V'_k}\cong \mathfrak F_{\psi, \mathbb{A}^1_{\mathbb{V}_k}}\circ 
\mathfrak F_{\psi, \mathbb V'_k/\mathbb A^1_k}.$$
So we have 
\begin{eqnarray*}
\mathrm{Hyp}'_{!,k}
=\mathfrak F_{\psi, \mathbb V'_k}(R\iota'_{k, !}\overline{\mathbb Q}_\ell)[d+1]
\cong  \mathfrak F_{\psi, \mathbb{A}^1_{\mathbb{V}_k}}
\mathfrak F_{\psi,\mathbb V'_k/\mathbb A^1_k}(R\iota'_{k, !}\overline{\mathbb Q}_\ell)[d+1].
\end{eqnarray*}
By \cite[1.2.2.1]{L}, we have
$$\mathfrak F_{\psi, \mathbb{A}^1_{\mathbb V_k}}(\mathrm{Hyp}'_{!,k})
\cong a_*\mathfrak F_{\psi, \mathbb{V}'_k/\mathbb{A}^1_{k}}(R\iota'_{k, !}\overline{\mathbb Q}_\ell)[d+1](-1),$$
where $a$ is the morphism $$a:  \mathbb{A}^1_{k}\times_k \mathbb{V}\to  \mathbb{A}^1_{k}\times_k\mathbb{V},\quad (t,v)\mapsto (-t,v).$$ 
Since the Deligne-Fourier transform commutes with the base change, and 
$$i_1^*R\iota'_{k, !}\overline{\mathbb Q}_\ell\cong R\iota_{k, !}\overline{\mathbb Q}_\ell,
\quad i_0^*R\iota'_{k, !}\overline{\mathbb Q}_\ell= 0,$$ we have
\begin{align*}
i^*_1\mathfrak F_{\psi, \mathbb{V}'_k/\mathbb{A}^1_{k}}(R\iota'_{k, !}\overline{\mathbb Q}_\ell)\cong\;&
\mathfrak F_{\psi, \mathbb{V}_k}(R\iota_{k, !}\overline{\mathbb Q}_\ell)= \mathrm{Hyp}_{!,k}[-d],\\
i_0^*\mathfrak F_{\psi, \mathbb{V}'_k/\mathbb{A}^1_{k}}(R\iota'_{k, !}\overline{\mathbb Q}_\ell)\cong\:&
\mathfrak F_{\psi, \mathbb{V}_k}(0)=0.
\end{align*}
Our assertion follows.
\end{proof} 

\begin{lemma}\label{finiteoverbase}
Let $p_2: \mathbb{A}^1_R\times_R\mathbb{V}_{R}\to \mathbb{V}_{R}$ be the projection,
let $\mathbb V'^{\mathrm{degen}}_K$ (resp. $\mathbb V^{\mathrm{degen}}_K$) be the closed subset of $\mathbb V'_K$
(resp. $\mathbb V_K$) consisting of those points $(a, A)$ (resp. $A$) so that the Laurent polynomial 
$$(t, g)\mapsto at+t \sum_{j=0}^N \mathrm{Tr}(A\rho_{j,K}(g))
\quad(\hbox{resp. }g\mapsto \sum_{j=0}^N \mathrm{Tr}(A\rho_{j,K}(g)))$$ 
on $\mathbb G_{m, K}\times_KG_K$ (resp. $G_K$) is degenerate, and let 
$$U'= \mathbb V'_R- \overline{\mathbb V'^{\mathrm{degen}}_K}\quad (\hbox{resp. }
U= \mathbb V_R- \overline{\mathbb V^{\mathrm{degen}}_K})$$ be the complement of the Zariski 
closure of $\mathbb V'^{\mathrm{degen}}_K$ (resp. $\mathbb V^{\mathrm{degen}}_K$) in $\mathbb V'_R$ (resp. $\mathbb V_R$).
Then after replacing $\mathrm{Spec}\,R$ by a dense open subset, we may find a chain 
$$\emptyset=F_m\subset\cdots \subset F_1\subset F_{0}=U$$
of closed subsets in $U$ so that each $F_i-F_{i+1}$ is smooth over $\mathrm{Spec}\,R$, 
and $$(\mathbb{V}'_{R}-U')\cap p_2^{-1}(F_i-F_{i+1})$$ is finite over $F_i-F_{i+1}$.
\end{lemma}

\begin{proof} 
We claim that the projection $p_2: \mathbb V'_K\to \mathbb V_K$ induces a quasi-finite morphism 
$$\mathbb V'^{\mathrm{degen}}_K\cap p_2^{-1}(\mathbb V^{\mathrm{gen}}_K)
\to \mathbb V^{\mathrm{gen}}_K.$$
It suffices to show that 
that for any face $\tau'$ of 
$\Delta'_\infty=\bigcup_{t\in [0,1]} t(1\times \Delta_{\infty})$ not containing the origin 
and any $A\in  \mathbb{V}_{K}^{\mathrm{gen}}(\overline K)$, the subset
$$Z_{\tau',A}=\{a\in \mathbb{A}^1({\overline K}): f_{\tau', (a, A)} \text{ has critical points}\}$$
is a finite set, where $f_{\tau', (a, A)}$ is (\ref{ftauA}) for the family of representations 
$\rho'_{0,K},\rho'_{1,K}\ldots, \rho'_{N,K}$
of $\mathbb G_{m, K}\times_K G_K$. We have $\tau'= \{1\}\times \tau$ for a face $\tau\prec\Delta_{\infty}$. 	
If $\tau$ doesn't contain $0$, then $$f_{\tau', (a, A)}((s,g),(t,h))=st f_{\tau,A}(g,h).$$ Since 
$f_{\tau,A}$ has no critical points, so is $f_{\tau', (a,A)}$ and $Z_{\tau',A}$ is empty.
If $0\in \tau$, then $$f_{\tau', (a,A)}((s,g),(t,h))=st(a+f_{\tau,A}(g,h)),$$ and 
$$Z_{\tau',A}=\{(-f_{\tau,A}(g,h),A): g,h\in G(\overline K),\; \mathrm d f_{\tau,A}(g,h)=0\}.$$
So $Z_{\tau',A}$ can be identified with the set of critical values of $f_{\tau,A}$, which is a finite set 
by \cite[III 10.6]{Hartshorne}. 

By the claim, there exists a nonempty open subset $W_{0,K}\subset \mathbb V^{\mathrm{gen}}_K$ such that 
$\mathbb V'^{\mathrm{degen}}_K\cap p_2^{-1}(W_{0,K})$ is finite over $W_{0,K}$. Let 
$F_{1, K}=\mathbb V^{\mathrm{gen}}_K-W_{0,K}$. By the claim again, there exists a smooth dense open subset
$W_{1,K}$ of $F_{1, K}$ such that $\mathbb V'^{\mathrm{degen}}_K\cap p_2^{-1}(W_{1,K})$ is finite over $W_{1,K}$.
Let $F_{2, K}=F_{1, K}-W_{1,K}$. In this way, we get a chain 
$$\emptyset=F_{m,K}\subset\cdots \subset F_{1,K}\subset F_{0,K}=V^{\mathrm{gen}}_K$$
of closed subsets in $\mathbb V^{\mathrm{gen}}_K$ so that each $F_{i,K}-F_{i+1,K}$ is smooth over $\mathrm{Spec}\,K$, 
and $$\mathbb V'^{\mathrm{degen}}_K\cap p_2^{-1}(F_{i,K}-F_{i+1,K})$$ is finite over $F_{i,K}-F_{i+1,K}$.
The proposition follows by the standard passage to limit argument. 
\end{proof}

\begin{lemma}\label{tameramification} Let $U= \mathbb V_R- \overline{\mathbb V^{\mathrm{degen}}_K}$.
For all but finitely many maximal ideals $\mathfrak m$ of $R$, and 
for any $A\in U({\bar k})$, $\mathrm{Hyp}'_{!,k}|_{\mathbb{A}^1_k\times\{A\}}$ is a lisse sheaf on a dense open 
subset of $\mathbb{A}^1_k\times\{A\}$ with tame ramification at $\infty$,
where $k=R/\mathfrak m$. 
\end{lemma}

\begin{proof} Replacing $\mathrm{Spec}\,R$ by a dense open subset, we may assume 
$\mathrm{Hyp}'_{!, R}|_{U'}$ comes from a lisse sheaf by Proposition \ref{mainprop}, 
and there exists a chain
$\emptyset=F_m\subset\cdots\subset F_1 \subset F_{0}=U$ satisfying the conclusion of Lemma \ref{finiteoverbase}. Choose $i$ so that
$A\in (F_{i}-F_{i+1})(k')$ where $k'$ is a finite extension of $k$. Suppose $A$ lies over $\mathfrak m\in\mathrm{Spec}\,R$, and 
let $R'$ be a complete discrete valuation ring unramified over
$R_{\mathfrak m}$ with residue field $k'$. Since $F_{i}-F_{i+1}$ is smooth, we may lift $A$ 
to a section $\tilde A\in (F_{i}-F_{i+1})(R')$. By Lemma \ref{finiteoverbase}, 
$\mathrm{Hyp}'_{!,R}|_{\mathbb{A}^1_{R'}\times\{\tilde A\}}$ 
is lisse outside a closed subset finite over $R'$. By Abhyankar's Lemma \cite[XIII 2.3 a), 5.5]{SGA1},
$\mathrm{Hyp}'_{!,R}|_{\mathbb{A}^1_{k'}\times \{A\}}$ is tamely ramified at $\infty$. By Proposition \ref{basichyp_D} (ii),
$\mathrm{Hyp}'_{!,k}|_{\mathbb{A}^1_k\times\{A\}}$ has the same property.
\end{proof}

\begin{proposition}\label{mainpropo} Let $U= \mathbb V_R- \overline{\mathbb V^{\mathrm{degen}}_K}$.
For all but finitely many maximal ideals $\mathfrak m$ of $R$, and for any $A\in U({\bar k})$, 
$(\mathrm{Hyp}_{!,k})_{A}$ is concentrated in degree $-\mathrm{dim}\,\mathbb V_K$ and its rank 
equals to the generic rank of $\mathrm{Hyp}'_{!,k}$, and is bounded by 
$d!\int_{\Delta_{\infty}\cap\mathfrak C} 
\prod_{\alpha \in R^+} \frac{\lambda(H_\alpha)^2}{\rho(H_\alpha)^2}\mathrm d\lambda,$ where $k=R/\mathfrak m$.
\end{proposition}

\begin{proof}
We identify $\mathbb{A}^1_k\times \{A\}$ with $\mathbb A^1_{\bar k}$. Let
$$K=\mathrm{Hyp}'_{!,k}|_{\mathbb{A}^1_k\times \{A\}},\quad K'=
(\mathfrak F_{\psi,\mathbb A^1_{\mathbb V_k}}(\mathrm{Hyp}'_{!,k}))|_{\mathbb{A}^1_k\times \{A\}}\cong \mathfrak F_{\psi,
\mathbb A^1_{\bar k}}(K).$$ 
By Lemma \ref{tameramification} and \cite[2.3.1.3]{L}, $\mathcal H^i(K')$ are lisse on $\mathbb{A}^1_{\bar k}-\{0\}$ for all $i$. 
By Proposition \ref{homo_nonhomo}, we have $K'|_0=0$. Let $\bar\eta_0$ be a geometric point over the generic point of the 
strict localization of $\mathbb{A}^1_{\bar k}$ at $0$. By \cite[2.3.2.1]{L}, we have
$$K'|_{\bar\eta_0}\cong R\Psi_{\bar\eta_0}(K')\cong R\Phi_{\bar\eta_0}(K')\cong 
\mathfrak F_{\psi}^{(\infty,0)}(K|_{\bar \eta_\infty}),$$ where $\eta_\infty$ is the generic point of
the strict henselizatioin of $\mathbb{P}^1_{\bar k}$ at $\infty$.  Since 
$A\in U(\bar k)$, by Lemma \ref{tameramification}, $K|_{\bar \eta_\infty}$ is tamely ramified. 
By Proposition \ref{mainprop}, $K|_{\bar\eta_\infty}$ is concentrated 
in degree $-(\mathrm{dim}\,\mathbb V_K+1)$, and its rank equals to the generic rank of $\mathrm{Hyp}'_{!,k}$, and 
hence is bounded by 
\begin{align*}
(d+1)!\int_{\bigcup_{t\in [0,1]}t(1\times \Delta_{\infty}\cap\mathfrak C)}
\prod_{\alpha \in R^+} \frac{\lambda(H_\alpha)^2}{\rho(H_\alpha)^2}\mathrm dt \mathrm d\lambda=
d!\int_{\Delta_{\infty}\cap\mathfrak C} 
\prod_{\alpha \in R^+} \frac{\lambda(H_\alpha)^2}{\rho(H_\alpha)^2}\mathrm d\lambda.
\end{align*} (Confer the calculation in the proof of Theorem \ref{modifiedhyp}.)
By \cite[2.4.3 (ii) b)]{L}, 
$\mathfrak F_{\psi}^{(\infty,0)}(K|_{\bar\eta_\infty})$ is concentrated in degree $-(\mathrm{dim}\,\mathbb V_K+1)$ and its rank 
equals to the generic rank of $\mathrm{Hyp}'_{!,k}$. So $K'|_{\mathbb A^1_{\bar k}-\{0\}}$ is concentrated in degree 
$-(\mathrm{dim}\,\mathbb V_K+1)$ and 
$\mathcal H^{\mathrm{dim}\,\mathbb V_K+1}(K')|_{\mathbb A^1_{\bar k}-\{0\}}$ is lisse with the rank bounded by $d!\int_{\Delta_{\infty}\cap\mathfrak C} 
\prod_{\alpha \in R^+} \frac{(\lambda, \alpha)^2}{(\rho, \alpha)^2}\mathrm d\lambda$. By Proposition \ref{homo_nonhomo},
our assertion for $(\mathrm{Hyp}_{!,k})_A$ follows by taking the stalk of $K'$ at $(-1, A)$.  
\end{proof}

The main conclusion of this section is the following.

\begin{theorem}\label{mainthm} Suppose $G_R$ is a split reductive group $R$-scheme, 
$\rho_{j, R}: G_R\to \mathrm{GL}(V_{j, R})$ $(j=1, \ldots, N)$ are representations such that 
$\rho_{j,\overline K}$ are irreducible, and the morphism 
$$\iota_R: G_R\to \mathbb V_R=\prod_{j=1}^N \mathrm{End}_R(V_{j, R}), \quad g\mapsto (\rho_{1,R}(g),\ldots, \rho_{N,R}(g))$$
is quasi-finite. Let 
$$U=\mathbb V_R-\overline {\mathbb V_K^{\mathrm{degen}}}$$ be the complement of the 
Zariski closure  in $\mathbb V_R$ of the set of degenerate Laurent polynomials. 
There exists a finite set $S$ of maximal ideals of $R$ so that for any maximal ideal $\mathfrak m\not\in S$, the restriction
to $U\otimes_R (R/\mathfrak m)$ of the 
hypergeometric sheaf $\mathrm{Hyp}_{k, !}$ on $\mathbb V_{k}=V_R\otimes_R R/\mathfrak m$ comes from a lisse sheaf of rank
$\leq d!\int_{\Delta_\infty\cap\mathfrak C} 
\prod_{\alpha \in R^+} 
\frac{\lambda(H_\alpha)^2}{\rho(H_\alpha)^2}\mathrm d\lambda.$ 
\end{theorem}

\begin{proof} By Proposition \ref{mainpropo} and Proposition \ref{perverseranklisse} in the Appendix, after removing finitely many 
maximal ideal of $R$, the perverse sheaf $\mathrm{Hyp}_{!, k}$ comes from a lisse sheaf when restricted to $U\otimes_R R/\mathfrak m$, 
and its rank has the claimed bound. 
\end{proof}

\begin{proof}[Proof of Theorem \ref{thm:expsum}] After removing finitely many maximal ideals, we may assume $V_{j, R}$ are 
free $R$-modules. Let $C_R$ and $C_K$ be the affine coordinate ring of $\mathbb V_R$ 
and $\mathbb V_K$, respectively. They are polynomials rings over $R$ and over $K$, respectively, and we may identify 
$C_R$ with a subring of $C_K$. 
Let $\mathfrak a_K$ be the ideal of $C_K$ defining the closed subset $\mathbb V_K^{\mathrm{degen}}$.
Then the ideal of $C_R$ define the closed set $\overline{\mathbb V_K^{\mathrm{degen}}}$ is 
$C_R\cap \mathfrak a_K$. Since the Laurent polynomial $f_A(g)=\sum_j \mathrm{Tr}(A_j\rho_j(g))$ is nondegenerate, there exists a polynomial 
$g \in \mathfrak a_K$ such that $g(A)\not=0$. Multiplying $g$ by a common denominator of its coefficients, we may assume 
$g \in C_R$. Then $g(A)$ is a nonzero element in $R$, and hence there are only finitely maximal ideals of $R$ containing $g(R)$. 
Let $\mathfrak m$ be a maximal ideal of $R$ not containing $g(A)$. 
As $g\in C_R\cap \mathfrak a_K$, the point $A\mod \mathfrak m$ lies in 
$U=\mathbb V_R- \overline{\mathbb V_K^{\mathrm{degen}}}$. 
By Theorem \ref{mainthm}, after deleting more finitely many maximal ideals of $R$, we may assume  the restriction
to $U\otimes_R (R/\mathfrak m)$ of the 
hypergeometric sheaf $\mathrm{Hyp}_{k, !}$ on $\mathbb V_{k}=V_R\otimes_R R/\mathfrak m$ comes from a lisse sheaf of rank
$\leq d!\int_{\Delta_\infty\cap\mathfrak C} 
\prod_{\alpha \in R^+} 
\frac{\lambda(H_\alpha)^2}{\rho(H_\alpha)^2}\mathrm d\lambda.$ We thus have 
$\mathcal H^i(\mathrm{Hyp}_{k, !})|_{U\otimes_R (R/\mathfrak m)}=0$ for $i\not=- m$ and 
$\mathcal H^{-m}(\mathrm{Hyp}_{k, !})|_{U\otimes_R (R/\mathfrak m)}$ is a lisse sheaf. 
We then have 
\begin{align*}
\sum_{g\in G(k')} \psi \Big(\mathrm{Tr}_{k'/k}\Big(\sum_{j=1}^N 
\mathrm{Tr}(A_{j}\rho_j(g))\Big)\Big)
&= (-1)^{d+m} \mathrm{Tr}(\mathrm{Frob}_A^{[k':k]}, (\mathrm{Hyp}_{\psi, !})_{\bar A}) \\
&=(-1)^d\mathrm{Tr}(\mathrm{Frob}_A^{[k':k]}, \mathcal H^{-m}(\mathrm{Hyp}_{\psi, !})_{\bar A}).
\end{align*}
Since $\mathrm{Hyp}_{\psi, !}$ is mixed of weight $\leq d+m$, all eigenvalues of 
$\mathrm{Frob}_A$ on $\mathcal H^{-m}(\mathrm{Hyp}_{\psi, !})_{\bar A}$ have 
absolute value $\leq q^{d/2}$, where $q$ is the number of elements of $k$.
The number of eigenvalues counted with multiplicity is $\leq  d!\int_{\Delta_{\infty}\cap\mathfrak C} 
\prod_{\alpha \in R^+} \frac{(\lambda, \alpha)^2}{(\rho, \alpha)^2}\mathrm d\lambda$.
We thus have 
\[\Big\vert\sum_{g\in G(k')} \psi \Big(\mathrm{Tr}_{k'/k}\Big(\sum_{j=1}^N 
\mathrm{Tr}(A_{j}\rho_j(g))\Big)\Big)\Big\vert \leq q'^{d/2} d!\int_{\Delta_{\infty}\cap\mathfrak C} 
\prod_{\alpha \in R^+} \frac{(\lambda, \alpha)^2}{(\rho, \alpha)^2}\mathrm d\lambda. \qedhere\]
\end{proof}

\section{Appendix}

Let $K$ be a field of characteristic $0$, $G$ a reductive group over $K$, 
$\rho_j:G\to\mathrm{GL}(V_j)$ $(j=1, \ldots, N)$ a family of representations, 
$\mathbb V=\prod_{j=1}^N\mathrm{End}(V)$, $$\mathrm{pr}: G\times G \times \mathbb V\to \mathbb V$$
the projection, $\Delta_\infty$ the Newton polytope at $\infty$, $\tau$ a face of $\Delta_\infty$ not
containing the origin, and $F_\tau$ the morphism 
$$F_\tau: G\times G\times\mathbb V\to \mathbb A^1,\quad (g, h, (A_1, \ldots, A_N))\mapsto 
\sum_{j=1}^N \mathrm{Tr}(A_j \rho_j(g)e(\tau)_j\rho_j(h^{-1})).$$
The equation $\mathrm{d}_{(g,h)}F_\tau=0$ 
defines a closed subscheme of $G\times G \times \mathbb V$ which we denote by 
the same equation, where $\mathrm{d}_{(g,h)}$ denotes the relative exterior differentiation with respect to
the $(g,h)$-variable. Define
$$\mathbb V^{\mathrm{degen}}=\bigcup_{0\not \in \tau\prec\Delta_\infty}\mathrm{pr}
(\mathrm{d}_{(g,h)}F_\tau=0).$$

\begin{proposition}\label{degenclosed} 
The set $\mathbb V^{\mathrm{degen}}$ is a Zariski closed subset of $\mathbb V$. 
\end{proposition} 

\begin{proof} Let $\rho'_j$ $(j=1,\ldots, N)$ be the representations 
$$\rho'_j: \mathbb G_m\times G\to \mathrm{GL}(V_j), \quad (t, g)\mapsto t\rho_j(g).$$
and let $X$ be the Zariski closure of the image of the morphism 
$$\iota: G\to \mathbb V, \quad g\mapsto (\rho_1(g), \ldots, \rho_N(g)).$$
Embed $\mathbb V$ into $\mathbb P(\mathbb V'):=\mathbb P(\mathbb A^1\oplus \mathbb V)$ via 
$v\mapsto [1:v]$. Let $\overline X$ be the closure of $X$ in $\mathbb P(\mathbb V')$ and let 
$$\overline X_\infty=\{v\in\mathbb P(\mathbb V): [0: v]\in \overline X\}.$$
By Proposition \ref{lm:Kapranovorbit} (iii), we have 
$$\overline X_\infty=\bigcup_{0\not\in\tau \prec\Delta_\infty}
\mathbb P((\mathbb G_{m}\times G)e(\tau)(\mathbb G_{m}\times G)),$$
where $\mathbb G_{m}\times G$ acts on $\mathbb V$ via 
$(\rho'_{1},\ldots, \rho'_{N})$, and $\mathbb P( (\mathbb G_{m}\times G)e(\tau)(\mathbb G_{m}\times G))$ 
is the projectivization of $(\mathbb G_{m}\times G)e(\tau)(\mathbb G_{m}\times G)$. Let 
$$S=\{([B],A)\in \overline X_\infty \times\mathbb V: 
\frac{\mathrm d}{\mathrm dt}\Big|_{t=0}\Big(\sum_{j=1}^N \mathrm{Tr}(A e^{t\xi_1} Be^{-t\xi_2})\Big)
=0 \hbox{ for all } (\xi_1, \xi_2)\in \mathfrak g\times\mathfrak g\}.$$
Then $S$ is a closed subset of $\overline X_\infty\times\mathbb V$, and $\mathbb V^{\mathrm{degen}}$ 
is the image of $S$ under the projection $\mathrm{pr}': \overline X_\infty \times\mathbb V\to 
\mathbb V$, which is a proper morphism. So $\mathbb V^{\mathrm{degen}}$ is closed. 
\end{proof}

\begin{proposition}\label{Qllisse}
Let $X$ be an integral normal noetherian scheme, 
let $\eta$ be its generic point, and let $\mathcal F$ be a $\overline{\mathbb{Q}}_{\ell}$-sheaf on $X$. 
The following conditions are equivalent.
\begin{enumerate}[(i)]
\item  $\mathcal F$ is a lisse sheaf.\label{Qllisse(1)}
\item For any $x \in X$ and any specialization morphism $\bar \eta\to \tilde{X}_{\bar x}$, the specialization map 
$\mathcal F_{\bar x}\to \mathcal F_{\bar \eta}$ is an isomorphism, where $\tilde{X}_{\bar x}$ is the strict 
localization of $X$ at $x$.
\label{Qllisse(2)}
\item For any $x \in X$, there exists a specialization morphism 
$\bar \eta\to \tilde{X}_{\bar x}$ so that the specialization map 
$\mathcal F_{\bar x}\to \mathcal F_{\bar \eta}$ is an isomorphism.\label{Qllisse(3)}
\end{enumerate} 
\end{proposition}

\begin{proof}
For (\ref{Qllisse(1)})$\Rightarrow$(\ref{Qllisse(2)}), we may reduce to the case of 
$\mathbb Z/\ell^n\mathbb Z$-sheaves, and the assertion is then clear. 
Since $X$ is normal, for any $x\in X$, $\tilde{X}_{\bar x}$ is integral. So $\tilde{X}_{\bar x}\times_X \eta$ has only one point. Thus 
any two specialization morphisms $\bar \eta\to \tilde{X}_{\bar x}$ differ by an automorphism of 
$\bar \eta$, and hence (\ref{Qllisse(2)})$\Leftrightarrow$(\ref{Qllisse(3)}). It remains to prove (\ref{Qllisse(2)})$\Rightarrow$(\ref{Qllisse(1)}).
Choose a finite extension
$E$ of ${\mathbb Q}_\ell$ contained in $\overline{\mathbb Q}_\ell$ such that $\mathcal F$ comes from $E$-sheaf $\mathcal F_E$. 
Let $D$ be the discrete valuation ring in $E$ and let $\lambda$ be an uniformizer.	
Assume $\mathcal F_E=\mathcal G\otimes_D E$, where $\mathcal G=(\mathcal G_n)_{n\geq 1}$ and each 
$\mathcal G_n$ is a flat constructible sheaf of $(D/\lambda^{n}D)$-modules, equipped with 
isomorphisms $\mathcal G_{n+1}\otimes_{D/\lambda^{n+1}D} D/\lambda^{n}D\cong \mathcal G_n$.

\medskip
\noindent\emph{Claim:} There exists a nonnegative integer $s$ 
such that for any $x\in X$, the image of a specialization map
$\mathrm{sp}: \mathcal G_{\bar x}\to \mathcal G_{\bar \eta}$ contains $\lambda^{s} \mathcal G_{\bar\eta}$. 	
\medskip

Choose a sequence 
$$\emptyset=Z_0\subset Z_1\subset\cdots \subset Z_s=X$$ of closed subsets in $X$ so that $Z_j-Z_{j-1}$ is irreducible 
and $\mathcal G_1|_{Z_j-Z_{j-1}}$ is locally constant for each $j$. Then $\mathcal G_n|_{Z_j-Z_{j-1}}$ is locally constant for each $j$ and each $n$. Let 
$\eta_j$ be the generic point of $Z_j-Z_{j-1}$. Fix a specialization morphism ${\bar \eta}\to \tilde{X}_{\bar \eta_j}$ which induces 
a special map  $\mathcal G_{\bar \eta_j}\to \mathcal G_{\bar\eta}$.
For any $x\in X$, choose $j$ so that $x\in Z_j-Z_{j-1}$. Each specialization map $\mathcal G_{\bar x}\to \mathcal G_{\bar\eta_j}$ 
is an isomorphism since each $\mathcal G_n|_{Z_j-Z_{j-1}}$ is locally constant. 
Up to an automorphism of $\bar\eta$, any specialization map 
$\mathcal G_{\bar x}\to \mathcal G_{\bar\eta}$ can be factorized as
$\mathcal G_{\bar x}\stackrel\cong\to \mathcal G_{\bar\eta_j}\to \mathcal G_{\bar\eta}$.	
It suffices to take a sufficiently large $s$ so that for each $j$, the the image of 
$\mathrm{sp}: \mathcal G_{\bar\eta_j}\to \mathcal G_{\bar\eta}$ contains $\lambda^{s} \mathcal G_{\bar\eta}$. 
Such an $s$ exists since by our assumption, we have isomorphisms
$$\mathrm{sp}: \mathcal G_{\bar\eta_j}\otimes_D E\cong \mathcal F_{\bar\eta_j}\stackrel\cong \to \mathcal F_{\bar\eta}\cong
\mathcal G_{\bar\eta}\otimes_D E.$$
This proves the claim.
	
Let $i: \eta\to X$ be the canonical morphism. For any pair $m\geq n$, one can check that 
the canonical epimorphism $\mathcal G_{m+s}\to \mathcal G_{n+s}$
induces an isomorphism 
\[\frac{\mathrm{ker}(\mathcal G_{m+s}\to \mathcal G_{s}\to i_*i^{-1}\mathcal G_{s})}{\lambda^{n} 
\mathrm{ker}(\mathcal G_{m+s}\to \mathcal G_{s}\to i_*i^{-1}\mathcal G_{s})}\stackrel\cong\to
\frac{\mathrm{ker}(\mathcal G_{n+s}\to \mathcal G_{s}\to i_*i^{-1}\mathcal G_{s})}{\lambda^{n} 
\mathrm{ker}(\mathcal G_{n+s}\to \mathcal G_{s}\to i_*i^{-1}\mathcal G_{s})}.\]
Let $\mathcal G'_n$ be the righthand side of the isomorphism. Then for any $n$, we have
$$\mathcal G'_{n+1}\otimes_{D/\lambda^{n+1}D} D/\lambda^{n}D\cong \mathcal G'_n.$$ 
For any $x\in X$, fix a specialization morphism $\bar \eta\to \tilde{X}_{\bar x}$. Recall that 
$\tilde{X}_{\bar x}\times_X \eta$ has only one point. So $\bar \eta\to \tilde{X}_{\bar x}\times_X \eta$ is faithfully flat, and the canonical map
\[(i_*i^{-1}\mathcal G_{s})_{\bar x}\cong \Gamma(\tilde{X}_{\bar x}\times_X \eta, \mathcal G_s)\to 
\Gamma(\bar \eta, \mathcal G_s)\cong \mathcal G_{s,\bar \eta}\]
is injective. So we may identify $\mathrm{ker}(\mathcal G_{n+s}\to \mathcal G_{s}\to i_*i^{-1}\mathcal G_{s})_{\bar x}$ with 
$\mathrm{ker}(\mathcal G_{n+s,\bar x}\to \mathcal G_{s, \bar x}\to \mathcal G_{s,\bar \eta})$, and we have
\[\mathcal G'_{n,\bar x}\cong\frac{\mathrm{ker}(\mathcal G_{n+s,\bar x}\to \mathcal G_{s, \bar x}\to \mathcal G_{s,\bar \eta})}{\lambda^{n} 
\mathrm{ker}(\mathcal G_{n+s,\bar x}\to \mathcal G_{s, \bar x}\to \mathcal G_{s,\bar \eta})}\cong 
\frac{\mathrm{ker}(\mathcal G_{\bar x}\to \mathcal G_{s,\bar \eta})}
{\lambda^{n} \mathrm{ker}(\mathcal G_{\bar x}\to \mathcal G_{s,\bar \eta})},\] where the second isomorphism 
is induced by the epimorphism $\mathcal G_{\bar x}\to \mathcal G_{n+s,\bar x}$. 
Since the specialization map $\mathrm{sp}: \mathcal G_{\bar x}\to \mathcal G_{\bar\eta}$ is injective, we may identify 
$\mathrm{ker}(\mathcal G_{\bar x}\to \mathcal G_{s,\bar \eta})$ with its image in $\mathcal G_{\bar \eta}$, which is just 
$$\mathrm{im}(\mathrm{sp}: \mathcal G_{\bar x}\to \mathcal G_{\bar \eta})\cap \lambda^{s} \mathcal G_{\bar \eta}=
\lambda^{s} \mathcal G_{\bar \eta}$$ by the above claim. So we have 
$$\mathcal G'_{n,\bar x}\cong \lambda^{s} \mathcal G_{\bar \eta}/  \lambda^{n+s} \mathcal G_{\bar \eta}.$$
This shows the specialization map $\mathrm{sp}: \mathcal G'_{n,\bar x}\to \mathcal G'_{n,\bar \eta}$ is an isomorphism. So 
$\mathcal G'_n$ is locally constant by \cite[5.8.9]{etale}.
Thus $\mathcal G'=(\mathcal G'_n)$ is a lisse $\lambda$-adic sheaf. 
The inclusions $$\mathrm{ker}(\mathcal G_{n+s}\to \mathcal G_{s}\to i_*i^{-1}\mathcal G_{s})\hookrightarrow \mathcal G_{n+s}$$ 
induce a morphism 
$\mathcal G'\to \mathcal G$ of $\lambda$-adic sheaf. At any $x\in X$, the stalk of this morphism is
$$\mathrm{ker}(\mathcal G_{\bar x}\to \mathcal G_{s,\bar \eta})\to \mathcal G_{\bar x}.$$ It becomes an isomorphism after tensoring with $E$.
We have $$\mathcal F\cong \mathcal G\otimes_DE\cong \mathcal G'\otimes_D E.$$ So $\mathcal F$ is lisse. 
\end{proof}

\begin{proposition}\label{perverseranklisse}
Let $X$ be a $d$ dimensional integral normal scheme of finite type over a field, 
and let $\mathcal F$ be a $\overline{\mathbb Q}_\ell$-sheaf
on $X$. Suppose the stalks of $\mathcal F$ at closed points of $X$ are of the same dimension, and $\mathcal F[d]$ is perverse. 
Then $\mathcal F$ is a lisse sheaf.
\end{proposition}

\begin{proof}
For any $x\in X$, we may find a locally closed subset $S$ of $X$ containing $x$ so that $\mathcal F|_S$ is lisse. 
Let $x'$ be a closed point of $X$ lying in $S$. Then $\mathcal F_{\bar x}$ and 
$\mathcal F_{\bar x'}$ have the same dimension. So the stalks of $\mathcal F$ at all geometric points of $X$ are the same. 
Let $\eta$ be the generic point of $X$. 

\medskip
\noindent\emph{Claim}: Each specialization map
$$\mathrm{sp}: \mathcal F_{\bar x}\to \mathcal F_{\bar \eta}$$ is injective. 
\medskip	
	
Choose a finite extension
$E$ of ${\mathbb Q}_\ell$ contained in $\overline{\mathbb Q}_\ell$ such that $\mathcal F$ comes from $E$-sheaf $\mathcal F_E$. 
Let $D$ be the discrete valuation ring in $E$ and let $\lambda$ be an uniformizer.	
Assume $\mathcal F_E[d]=K\otimes_D E$, where $K$ is a complex of $\lambda$-adic sheaves. 
Replacing $K$ by ${^p\mathcal H}^0(K)$, we may assume $K$ is a perverse sheaf. 
Replacing $K$ by its quotient by the maximal torsion perverse subsheaf, which exists by \cite[4.0 (b)]{BBD}, 
we may assume $K$ also lies in the heart of the t-structure $[p^+_{1/2}]$ by \cite[3.3.4 (i)-(ii)]{BBD}. By \cite[4.0 (a)]{BBD}, the 
Verdier dual of $K$ is also perverse. So each $K\otimes^L_D D/\lambda^nD$ is a perverse sheaf. 
In particular $K\otimes^L_D D/\lambda^nD$ lies in $D^{[-d,0]}_{\mathrm{ctf}}(X, D/\lambda^nD)$. 
It suffices to show each
\[\mathrm{sp}: \mathcal H^{-d}(K\otimes_D D/\lambda^nD)_{\bar x}\to \mathcal H^{-d}(K\otimes_D D/\lambda^nD)_{\bar \eta}\]
is injective. If this is not true, we may find a connected \'etale neighborhood $U$ of $x$ and a section 
$s\in H^{-d}(U, K\otimes_D D/\lambda^nD)$ with support strictly smaller than $U$. Let $y$ be the generic point of an irreducible 
component of the support of $s$, and let $i_y:\mathrm{Spec}\,k(y)\to X$ be the canonical mophism. Then $s$ defines a nonzero section in 
$$i_{y}^!\mathcal H^{-d}(K\otimes_D D/\lambda^nD)\cong \mathcal H^{-d}(Ri_{y}^!(K\otimes_D D/\lambda^nD)),$$
which contradicts to the fact that $\mathcal H^j(Ri_{y}^!(K\otimes_D D/\lambda^nD))=0$ for 
$j<-\mathrm{dim}\overline{\{y\}}$. This proves the claim.
	
Since $\mathcal F_{\bar x}$ and $\mathcal F_{\bar \eta}$ have the same dimension, the claim implies that 
$\mathrm{sp}$ is an isomorphism. By Proposition \ref{Qllisse}, $\mathcal F$ is lisse.
\end{proof}

\end{document}